\documentclass[12pt,reqno]{amsart} 

\usepackage[utf8]{inputenc} 
\usepackage{bbm}
\usepackage[labelfont=bf]{caption}

\usepackage{graphicx} 

\usepackage{booktabs} 
\usepackage{array} 
\usepackage{paralist} 
\usepackage{verbatim} 
\usepackage{subfig} 
\usepackage{mathrsfs}
\usepackage{amssymb}
\usepackage{amsthm}
\usepackage{amsmath,amsfonts,amssymb,esint}
\usepackage{graphics,color}
\usepackage{enumerate}
\usepackage{mathtools,centernot}
\usepackage{cases}

\usepackage{float}

\usepackage[nocenter]{qtree}
\usepackage{tree-dvips}
\usepackage[linguistics]{forest}
\usetikzlibrary{calc} 
\newcounter{levelcount} 
\newcommand\countnodes[2]{
    \draw let \p{L} = (#1) in (-#2,\y{L}) node {\pgfmathparse{int(\value{levelcount}-1)}\pgfmathresult};
    \stepcounter{levelcount} 
}

\newtheorem{innercustomthm}{Theorem}
\newenvironment{customthm}[1]
  {\renewcommand\theinnercustomthm{#1}\innercustomthm}
  {\endinnercustomthm}

\newtheorem*{remark 1}{Remark}

\newtheorem{lemma}{Lemma}
\newtheorem{theorem}{Theorem}

\newtheorem{proposition}{Proposition}
\newtheorem{remark}{Remark}

\newcommand{\norm}[1]{\left\|#1\right\|}
\newcommand{\abs}[1]{\left|#1\right|}

\usepackage{ color, graphicx}
\usepackage{mathrsfs, dsfont}

\usepackage{hyperref}

\numberwithin{equation}{section}

\newcommand{\RR}{\mathbb{R}}

\newcommand{\p}{\partial}

\newcommand{\CA}{\mathcal{A}}

\newcommand{\CB}{\mathcal{B}}

\newcommand{\R}{\mathbb{R}}

\newcommand{\CG}{\mathcal{G}}

\newcommand{\Z}{\mathbb{Z}}

\renewcommand\>{\rangle}

\newcommand\nc{\newcommand}
\nc\hd{\widehat{D}}
\nc\kp{\kappa}
\nc\8{\infty}

\DeclareMathSymbol{\Gamma}{\mathalpha}{letters}{"00}
\DeclareMathSymbol{\Theta}{\mathalpha}{letters}{"02}
\DeclareMathSymbol{\Lambda}{\mathalpha}{letters}{"03}
\DeclareMathSymbol{\Omega}{\mathalpha}{letters}{"0A}
\definecolor{cr}{rgb}{1,0,0}

\title{Cancellations of Resonances and Long Time Dynamics of Cubic Schr\"odinger Equation on $\mathbb{T}$}
\author{Kexin Jin}
\address{Princeton University}
\email{kexinj@math.princeton.edu}
\author{Xiao Ma}
\address{Princeton University}
\email{xiaom@math.princeton.edu}

\begin{document}

\begin{abstract}
We prove a vanishing property of the normal form transformation of the 1D cubic nonlinear Schr\"odinger (NLS) equation with periodic boundary conditions on $[0,L]$. We apply this property to quintic resonance interactions and obtain a description of dynamics for time up to $T=\frac{L^2}{\epsilon^4}$, if $L$ is sufficiently large and size of initial data $\epsilon$ is small enough. Since $T$ is the characteristic time of wave turbulence, this result implies the absence of wave turbulence behavior of 1D cubic NLS. Our approach can be adapted to other integrable systems without too many difficulties. In the proof, we develop a correspondence between Feynman diagrams and terms in normal forms, which allows us to calculate the coefficients inductively. 
\end{abstract}

\maketitle

\setcounter{tocdepth}{1}

\tableofcontents

\section{Introduction}\label{Sec:Intro}

We consider the nonlinear Schr\"odinger (NLS) equation with cubic nonlinearity on $[0,L]$ with periodic boundary conditions, that is 
\begin{equation}\label{1NLS0}
\left\{ \begin{array}{l}
- i \partial_t v + \frac{1}{2\pi} \partial_{xx} v = \sigma |v|^{2} v, \qquad   x \in  \mathbb{T}_L=[0,L] , \\
v(t=0) =\epsilon v_0,
\end{array} \right.
\end{equation}
where $\sigma=\pm1$ denoting the focusing case or the defocusing case. This equation was discovered by physicists as a fundamental equation in nonlinear optics and Bose-Einstein condensation. One of the very first mathematical results for this equation was done by Zakharov-Shabat \cite{ZS} by inverse scattering method on $\mathbb{R}$. 
Inverse scattering on $\mathbb{R}$ is a powerful method that reduces this equation to linear equations. It gives explicit formulas for all soliton solutions and completely characterizes the asymptotic behavior of the spatial localized solution on $\mathbb{R}$, that is the soliton resolution conjecture. Although later development in nonlinear PDEs provided many other methods that could partially recover the soliton resolution conjecture for nonintegrable system, no method could recover the full conjecture. So it turns out inverse scattering method is the most suitable tool in the study of integrable system on $\mathbb{R}$. 

But inverse scattering method is not so powerful in the case of periodic boundary condition. The counterpart of inverse scattering method on $\mathbb{T}$ involves very difficult spectrum theory, like the Bloch spectrum and very restrictive assumption on the initial data, like finite gap property. For more details of inverse scattering method on $\mathbb{T}$, see \cite{NMPZ1984}. Here pure PDE method seems to be more powerful. For example in \cite{Bourgain1993}, Bourgain used number theory method and restricted norm method to derive the Strichartz estimates and local well-posedness in $L^2(\mathbb{T})$ (Local well-posedness in $H^{s}(\mathbb{T})$ with $s>\frac{1}{2}$ is obvious and global well-posedness follows from energy conservation laws in the defocusing case). For more refined local well-posedness results see for example \cite{Christ}, \cite{ET} and \cite{KST}. 

After well-posedness is established, the next step is to describe the long time dynamics of these solutions. For example Bourgain proposed many promising approaches. In \cite{Bourgain1994}, he constructed an invariant measure for 1D NLS, which is the preliminary of any ergodic theory for this equation. In \cite{Bourgain1994'}, he applied symplectic nonsqueezing to exclude asymptotic stability of NLS on $\mathbb{T}$. He also initiated the construction of quasi-periodic in time solutions. See for example \cite{ProcesiProcesi} for a recent result of quasi-periodic solutions.

In this paper, we shall also describe the long dynamics on $\mathbb{T}_L$, torus of length $L$. The large box long time dynamics of dispersive PDEs has been studied by many authors, see for example \cite{BGHS1}, \cite{BGHS2}, \cite{FGH}, \cite{fange1}, and \cite{fange2}. In what follows, it's more convenient to consider the new variable $u$, $v=\epsilon u$. And we shall always take $\sigma=1$, which makes no difference from taking $\sigma=-1$. Thus $u$ satisfies
\begin{equation}\label{1NLS}
\left\{ \begin{array}{l}
- i \partial_t u + \frac{1}{2\pi} \Delta u = \epsilon^2|u|^{2} u, \qquad   x \in  \mathbb{T}_L , \\
u(t=0) = u_0.
\end{array} \right.
\end{equation} 

For $u$, we have the following result.

\begin{theorem}\label{th:main}
Fix $\ell> 1$, $0<\gamma\ll1$, and $M>0$. Let $f_{0}\in X^{\ell,2}$ and $B=\norm{f_{0}}_{X^{\ell,2}(\mathbb{R})}$. Let $u$ be a solution of NLS (\ref{1NLS}) with initial data $u_{0}(x)=\frac{1}{L}\sum_{\mathbb{Z}_{L}}f_{0}(K)e(Kx)$ and Fourier coefficients $u_K(t)$. Then for $L$ sufficiently large and $\epsilon^2L^{\gamma}$ sufficiently small, depending on $M$, $B$, there exists a constant $C_{\gamma}$, such that for any $t\in[0,MT_{R}]$, $T_{R}=\frac{L^2}{\epsilon^4}$,  
\begin{equation}
\begin{split}
    &\norm{u_{K}(t)-f_0(K)e\left(\Big(K^2+\frac{\epsilon^2}{\pi L}\norm{u}_{L^2(\mathbb{T})}^2+\frac{\epsilon^2}{\pi L^2}|f_0(K)|^4\Big)t-\frac{\epsilon^2}{2\pi L^2}\int^t_{0}Q_{K}(f)(s) ds\right)}_{X^{\ell}(\mathbb{Z}_{L})}
    \\
    &\lesssim C_{\gamma}(\epsilon^2 L^\gamma+\frac{1}{L^{1-\gamma}}).
\end{split}
\end{equation}
where $Q_{K}(f)$ is a complicated but explicit function of $f(K)$ defined in Proposition \ref{prop2}, $f(K)$ is defined by
$$f(K)=e(\frac{1}{2\pi}|f_0(K)|^4t)f_0(K),$$
$\|f\|_{X^{\ell,N}(\mathbb{R})}=\sum_{0\leq |\alpha|\leq N} \|\langle x \rangle^\ell\nabla^\alpha f(x)\|_{L^\infty(\mathbb{R})}$ and  $\|u\|_{X^{\ell}(\mathbb{T}_L)}=\sup_{K\in\mathbb{Z}_L}\| \langle K \rangle^\ell u_K \|$ where $u_K$ is the Fourier coefficients of $u$.
\end{theorem}
\begin{remark}
This system has infinitely many conserved quantities. In particular, $\norm{u}_{L^2(\mathbb{T})}$ is conserved. More precisely, we have $\norm{u}_{L^2(\mathbb{T})}=\norm{u_0}_{L^2(\mathbb{T})}=\frac{1}{L^2}\sum_{\mathbb{Z}_{L}}\abs{f_{0}(K)}^2<\infty$. 
\end{remark}
\begin{remark}
If we fix a function $u_0$ with $\text{supp } u_0\subseteq [-1,1]$, we can of course also think it as a periodic function on $\mathbb{T}_L=[-\frac{L}{2},\frac{L}{2}]$ (after translation we may also think it as function on $[0,L]$). If our initial data is of this form, $f_0(K)$ can be taken as the Fourier transform of $u_0$ on $\mathbb{R}$. More generally, for all spatially localized initial data, we can find $f_{0}(K)$ as in the assumption of Theorem \ref{th:main}. The existence of $f_0(K)$ is a type of spatial localization assumption.
\end{remark}

\begin{remark}
Recall that for a periodic function on $[0,L]$, its Fourier coefficients constitute a discrete function defined on $\mathbb{Z}_L$. As $L\rightarrow+\infty$, these Fourier coefficients are defined on a denser and denser set. So one may expect that there is asymptotically a continuous profile of this discrete function. Our result formulates this continuous profile. 
\end{remark}

The following parts of the introduction is devoted to explain three important aspects of our paper: the time in which the phenomenon described by our result happens, what is this phenomenon, and the tools of proving this phenomenon.

\subsection{The times scales of weakly nonlinear large box limit} The first important aspect in our paper is the time scale $T_R$. Now let us explain the power of $L$ and $\epsilon$ in $T_R=\frac{L^2}{\epsilon^4}$.

We have two trivial time scales:

\underline{The weak nonlinearity time scale:} By Duhamel's principle, we have
$$u(x,t)=e^{\frac{it}{2\pi}\partial^2} u_0+\varepsilon^{2}\int_0^te^{\frac{i(t-s)}{2\pi}\partial^2}|u|^{2}u ds.$$
If $t\in[0,\frac{1}{\varepsilon^{2}}]$, the local well-posedness can be established by doing contraction mapping argument which also gives us $u\sim e^{\frac{it}{2\pi}\partial^2} u_0$. So when $t\in[0,\frac{1}{\varepsilon^{2}}]$, the dynamics of the solution on torus is similar to that of a linear equation.

\underline{The large box time scale:} We may think that in physical space the solution is supported in a ball of radius $1$. Thus the dispersion relation $|\xi|^2$ tells us the speed of the wave is less than $1$. Hence it takes time $O(L)$ for the solution to see the boundary effect. Therefore, when $t\le L$, the dynamics of $u$ is similar to that of the NLS equation on $\mathbb{R}$.

In summary,  when $t\le \frac{1}{\epsilon^2}$ or $t\le L$, the solution can be either approximated by linear Schr\"odinger equation on $\mathbb{T}_L$ or Nonlinear Schrodinger equation on $\mathbb{R}$, so the dynamics are trivial. To go beyond these trivial time scales, for instance Faou-Germain-Hani in \cite{FGH} and Buckmaster-Germain-Hani-Shatah in \cite{BGHS1} derived the $continuous\ resonance$ (CR) equation, which describes the dynamics of NLS on rational torus in dimension higher than $1$ up to the resonant time-scale $L^2/\epsilon^2$. One could interpret the $L^2$ here as the period of the linear solution. In \cite{BGHS1}, they adapted the Hardy-Littlewood circle method to PDE setting and obtained a equidistribution result for the algebraic variety of resonance interactions. See for instance \cite{BGHS2} for more analysis of the CR equation. In \cite{BGHS3}, the authors upgrade the time scale to $\frac{L^{2.65}}{\epsilon^2}$ in higher dimensional case, given initial data with random phase on generic irrational torus. They derived the wave kinetic equation to describe the evolution of the moduli of the Fourier coefficients of the solution up to this time scale.

In this paper, we examine the one dimensional cubic equation up to the time-scale $L^2/\epsilon^4$. Although typically the long time behavior problem is more difficult in lower dimensional space, we improve the time scale in $\epsilon$ than that of previous results. The main idea to get this improvement is to take advantage of the integrability structure of 1D cubic NLS. Remarkably, our time scale $T_R=L^2/\epsilon^4$ is the characteristic time scale of the wave turbulence. In this time scale, we obtain a quasi-periodic behavior of the Fourier coefficients of the solution. We call this behavior modified scattering because of its similarity with the results of modified scattering on Euclidean space. In following two subsections, we shall discuss the this behavior and a structure of the equation leading to this behavior which is related to integrability.

\subsection{Modified scattering} Many researchers have discovered that the solutions of low dimensional dispersive PDEs with small initial data exhibit the modified scattering behavior, i.e. Fourier transform of the solution converges to the linear solution with a phase correction, contrary to the usual scattering results in which the limit is the linear solution without correction, see for example Kato-Pausader \cite{KP}, Ionescu-Pausader \cite{IoPa1}, \cite{IoPa2}. In our paper we use a change of phase argument to obtain a very similar result up to time $T_R=\frac{L^2}{\epsilon^4}$, the characteristic time of wave turbulence. One implication of our result is the solution exhibits modified scattering behavior instead of wave turbulence, when time of evolution is less than $T_R$, the wave turbulence time. This coincides with the intuition of physicists that integrable systems should not have wave turbulence behavior. 

Recall that in the work of Kato and Pusateri \cite{KP}, they derived the modified scattering result of 1D cubic NLS on $\R$. More precisely, they proved a scattering result in $L^\infty$ norm after doing a change phase correction which essentially due to the $t^{-1}$ decay in the time integration of Duhamel formula. The last term $\frac{\epsilon^2}{2\pi L^2}\int^t_{0}Q_{K}(f)(s) ds$ in our theorem also gives a $L^+$ phase factor after integration. So from this point of view two results are similar, although the proof are quite different. This is another reason why we call the behavior of the solution here the modified scattering. 

\subsection{A vanishing property of normal form transformation}
The basic tool of proving modified scattering here is the normal form transformation. This technique was first introduced by Shatah in \cite{Shatah1} in its original form of Poincar\'e normal form. A crucial improvement was first made by Germain-Shatah-Masmoudi \cite{GMS} with a new set of ideas called space-time resonance method. In \cite{GMS}, the vector field method and normal form transformation are elegantly unified in a single technique, in which differentiation by vector fields are replaced by integration by parts in space variables and normal form transformation are replaced by integration by parts in time variable. There are also other form of normal form transformation, for example the differentiation by parts approach in \cite{BIT}. But there are no essential difference between different approaches. 

It has been well known that normal form transformation can be applied to exploit the hidden smoothing effect in the nonlinearity to prove low regularity well-posedness results. This approach was initiated by Bourgain \cite{Bourgain1993}, then reintepreted as standard normal form transformation in \cite{Bourgain2004}, which also pointed out the connection with $I-$method introduced in \cite{CKSTT'}. For later development, see for example, \cite{BIT}, \cite{GKO}, \cite{OW} and section 4.1 of \cite{TaoBook}. 

Our main idea here is to approximate our system by a simpler ODE system (\ref{eq:cr}). Here the right hand side of (\ref{eq:cr}) is extracted from the main term given by the normal form transformation. We use the differentiation by parts approach in \cite{BIT} to perform this transformation. When working with this transformation to extract the main term, the fact that it is an integrable system plays a significant role. It has been studied that most of the normal form coefficients vanish for integrable systems. For the purpose of proving theorem \ref{th:main}, we only apply this vanishing property for quintic forms. But more generally, we prove this vanishing property holds for arbitrary order normal form transformation. Because of the vanishing properties of higher order term, it might be possible to control the dynamics up to a much longer time $\frac{L^3}{\epsilon^4}$, but further ideas are needed. For more explanation, see section \ref{sec:ideaproof}. Roughly speaking, we proved the following theorem. For a complete version of this theorem, see Theorem \ref{MainTheorem1}.

\begin{customthm}{2'}
For nonlinear Schrodinger equation (\ref{1NLS}), the coefficients $H_{K_1...K_{2d+3}}^{d+1}$ of its normal form vanish for $(K_1,...,K_{2d+3})$ outside some lower dimensional submanifolds of the resonance surface $\mathscr{R}$ for any $d\le P$. Here the resonance surface is defined by 

\begin{equation}
    \mathscr{R}=\{(K_1,...,K_{2d+3}):K_1-K_2+\cdots+K_{2d+3}=K,\ K_1^2-K_2^2+\cdots+K_{2d+3}^2=K^2\}
\end{equation}
\end{customthm}

To proof this Theorem, we need a very precise description of the normal form coefficients. This is done by the correspondence between Feynman diagrams and terms in normal form transformation established in Lemma \ref{l:formulaH^d}. The vanishing property is established by an inductive argument on the levels of Feynman diagrams. Feynman diagram is also a convenient tool in estimating the error in the normal form transformation (See Lemma \ref{l:numbertheory1}), although this application not as essential as the vanishing property. 

The vanishing property has been known by physicist as the cancellation of collision coefficients. Collision coefficients are not a rigorous mathematical definition but it is clearly related to normal form coefficients. For example, in the work of Zakharov and Schulman \cite{ZS}, \cite{ZSc} (Theorem 2.2.1), using perturbation expansion and S-matrix, they proved that the collision coefficients vanish for any systems on $\R$ if it possesses an additional conservation law other than energy and momentum. 
Their method does not apply to bounded domain like $\mathbb{T}$, because physically there are no scattering on this domain so it makes no sense to talk about S-matrix (scattering matrix). From mathematical point of view, this failure is because one heuristic identity
$$\lim_{\epsilon\rightarrow0}\frac{e^{iEt}}{E-i\epsilon}=\pi\delta(E),$$ is not true in any rigorous or non-rigorous sense on $\mathbb{T}$. For the vanishing property of quintic term, \cite{ZOCO} provides a calculation for all common integrable system, but it will be difficult to generalize the proof to higher order terms.

Feynman diagram is a very common tool in the investigation of normal form transformation, see for instance \cite{BGHS3}, \cite{LS}, \cite{GKO} and \cite{OW}. But for our purpose, we shall develop a new formula (\ref{eq:formulaH^d}) of normal form coefficients in terms of Feynman diagram. Here to prove the vanishing property, we shall exploit the graph structure of the Feynman diagram carefully.   

\subsection{Organization and Notations}\label{sec:notation}

This paper is roughly organized in the following way. In section \ref{sec:NF-section}, we perform a change of phase argument, normal form transformation, and obtain recurrence formulas of the coefficients. In section \ref{sec:Feynman}, we introduce Feynman diagrams, use them to formulate explicit formulas of norm form coefficients, and prove the vanishing property. In section \ref{sec:dynamics}, we establish our main long time dynamics result.

We will use following notations frequently.

\begin{itemize}
    \item $\mathbb{T}_L := [0,L]$ be the torus of length $L$.
    \item $\mathbb{Z}_L =  \frac{\Z}{L}=\{\frac{k}{L}:k\in\mathbb{Z}\}$.
    \item $e(z)=e^{2\pi i z}$ for $z\in\mathbb{C}$.
    \item  Given two quantities $A$ and $B$, we denote $\displaystyle \<A\> = \sqrt{1 +A^2}$ and $\displaystyle A \lesssim B$ if $\exists\ C$, a universal constant, such that $A \leq CB$.

    \item $L^+$ is the short hand of the quantity $A$ that satisfies $\forall\ \gamma\ll 1$, $\exists\ C_{\gamma}$ independent of $L$, s.t. $|A|\le C_{\gamma}L^{\gamma}$.

\noindent\underline{Fourier transform:}
    \item The Fourier transform of a function $f$ on $\R$ is 
    $$\mbox{if $\xi \in \mathbb{R}$}, \quad \mathcal{F} f (\xi) = \widehat{f} (\xi) = \int_{\mathbb{R}^d} e(-x \cdot \xi) f(x)\,dx,$$
    where $e(z)=e^{2\pi i z}$ for $z\in\mathbb{C}$.
    \item The inverse Fourier transform of of a function $f$ on $\R$ 
    $$\mathcal{F}^{-1} g (x) = \check{g} (x) =  \int_{\mathbb{R}^d} e(x \cdot \xi)  g(\xi)\,d\xi.$$
    \item The Fourier coefficients of a function $f$ defined on torus $\mathbb{T}_L$ is 
    $$f_K = \int_{\mathbb{T}_L} f(x) e(- K \cdot x) \,dx, \quad K \in \mathbb{Z}_L =  \frac{\Z}{L}.$$

The inversion formula is
\[
f(x) = \frac{1}{L} \sum_{K \in \mathbb{Z}_L} f_K e(K  \cdot x) .
\]

The Plancherel identity is 
$$||f||_{L^2{(\mathbb{T}_L)}}^2=\frac{1}{L}\sum_{K\in\mathbb{Z}_L} |f_K|^2.$$

We note that there are more than one definition of the Fourier coefficients. For example, $f_K = \frac{1}{L}\int_{\mathbb{T}_L} f(x) e(- K \cdot x) \,dx, \quad K \in \mathbb{Z}_L =  \frac{\Z}{L}$. Our choice of the definition is to ensure the presence of $L^{-1}$ in front of the inversion formula. This factor is necessary for the sum on the right hand side converging to a integral.

\noindent\underline{The functional spaces:}

\item  For function $f$ defined on $\R$, 
$$\| f \|_{X^{\ell}(\mathbb{R})} = \| \langle x \rangle^\ell f(x) \|_{L^\infty(\mathbb{R})},\ \ \|f\|_{X^{\ell,N}(\mathbb{R})}=\sum_{0\leq |\alpha|\leq N} \|\nabla^\alpha f\|_{X^\ell(\mathbb{R})}.$$
\item For function $a_K$ defined on $\mathbb{Z}_L$, $\| a_K \|_{X^\ell} = \sup_{K\in\mathbb{Z}_L}\| \langle K \rangle^\ell a_K \|.$ And for function $u$ defined on $\mathbb{T}_L$ with Fourier coefficients $u_K$, $\| u \|_{X^\ell(\mathbb{T}_L)} = \| u_K \|_{X^\ell}.$

\noindent\underline{Normal form transformation:}

       \item We will encounter following linear form and quadratic form,
$$\mathcal S_{2d+1,K}(K_1,\dots,K_{2d+1}) = K_1 - K_2  + \dots - K_{2d} + K_{2d+1} - K,$$
$$\Omega_{2d+1,K}(K_1,\dots, K_{2d+1}) =   K_1^2 - K_2^2 + \dots - K_{2d}^2+ K_{2d+1}^2 -K^2.$$
We write $\mathcal S_{2d+1,K}$ and $\Omega_{2d+1,K}$ if there is no confusion about the indices.

      \item We shall modify the phase of $u_K$ for several times, and each time we have a corresponding equation, for $a_K$ see (\ref{eq:profile0}), for $b_K$ see (\ref{eq:profileB}), for $d_K$ see (\ref{eq:profiled}). $c_K$ the variable derived from $d_K$ after applying $P$ times normal form, see (\ref{eq:transformedmain}).
       \item For the cubic quadratic form with time dependent correction, 
       $$\widetilde{\Omega}_{2d+1,K}=\Omega_{2d+1,K}-\frac{\epsilon^2}{2\pi L^2}D_{2d+1,K}(t),$$
where $$D_{2d+1,K}(t)=|d_{K_1}|^2-|d_{K_2}|^2+ \dots+|d_{K_{2d+1}}|^2-|d_K|^2.$$
      \item We use $H^d_{K_1K_2...K_{2d+1}}$, see (\ref{eq:formulaH^d}), and $G^d_{K_1K_2...K_{2d+1}}$, see (\ref{eq:formulaG^d}), to denote the coefficients of terms after $d$ times normal form transformation. 
      \item We use $S_K^d(u)$, $H_K^d(u)$, $G_K^d(u)$, $\widetilde{H}^{P+1}_{K}(u)$, $\widetilde{G}^{P+1}_{K}(u)$, $F_K^d(u)$, and $E_K^d(u)$ to denote polynomials in $u$ after normal form transformation, see (\ref{eq:recurrenceS^d_K}), (\ref{eq:Hd_u}), (\ref{eq:Gd_u}), (\ref{eq:HP_u}), (\ref{eq:GP_u}), (\ref{eq:Fd_u}), and (\ref{eq:Ed_u}).

\noindent\underline{Feynman diagrams:}

\item  $\displaystyle T\mbox{ is a tree in a forest of Feynman diagrams}$.
\item  $\displaystyle \CG_d \mbox{ is the forest of Feynman diagrams after d times normal form transformation}$. 
\item  $A_1$, $A_2$, $A_3$ are defined in (\ref{eq:defAi}) and $B_1$, $B_2$, $B_3$, $B_4$ are defined in (\ref{eq:defBi}).
\item  $*$, $\CA_k(T)$, $\CB_k(T)$, $\Omega_{2k+3,K}(T)$, $l_k^T$ and $S_{T_*}$ are defined in Lemma \ref{l:formulaH^d}.

\end{itemize}

{\bf Acknowledgments} The authors are very grateful to Jalal Shatah and their advisor Alexandru Ionescu for suggesting this question and the most helpful discussions. The authors would like to thank Pierre Germain for his illuminating questions. The authors are also grateful to Fan Zheng for his constant help.

\section{Normal Form Transformation}\label{sec:NF-section}

In this section, we shall obtain recurrence formulas of coefficients of the normal form transformation. Meanwhile, we extract the main term and the error terms from the nonlinearity. We remove the cubic resonance interaction by change of phase, but there will be a term left which forces us to further change phase with dependence on $K$. Later, we do normal form transformation in the same way as in \cite{BGHS1}. The main difficulty here is to deal with the extra term coming from the $K$ dependent phase correction. To deal with this, we will define coefficients $G^d_{K_1\cdots K_{2d+1}}$ in addition to coefficients $H^d_{K_1\cdots K_{2d+1}}$ introduced in \cite{BGHS1}. 

\subsection{Extraction of the Modified Phase}\label{sec:changephase} In this subsection, we first introduce the profile by removing the linear part. Then we separate the cubic nonlinearity into resonance and non-resonance interaction. While we keep the non-resonance interaction which will be removed later by normal form, the resonance interaction will be removed by changing phase.
In Fourier space, we have $u(t,x)=\frac{1}{L}\sum_{K\in\Z_L}u_K(t)e(Kx)$, where $K\in\Z_L=\frac{\Z}{L}$, $u_K$ satisfies
\begin{equation}\label{eq:uk}
- i \partial_t u_K +2\pi K^2 u_K= \frac{\epsilon^2}{L^{2}}\sum_{\mathcal{S}_{3,K} =0}  
u_{K_1} \overline{u_{K_2}} u_{K_3},
\end{equation}
where $K_i\in  \mathbb{Z}_L$, for $i=1, 2, 3$. 

Consider the profile $a_K=e(-K^2t)u_K$, $a_K$ satisfies 
\begin{equation}\label{eq:profile0}
- i \partial_t a_K = \frac{\epsilon^2}{L^{2}}\sum_{\mathcal{S}_{3,K} =0}  
a_{K_1} \overline{a_{K_2}} a_{K_3} e(\Omega_{3,K}t).
\end{equation} 

If we take $\epsilon=0$, the profile is a constant function in time. The profile is simply obtain by removing the effect of linear evolution, so it should grow much slower than $u_K$. But due to the presence of nonlinearity, variation of the profile still happens. 
The resonance interaction is the most important part of the nonlinearity which derives this variation.
In the cubic term, we separate the dynamics of $a_K$ into resonance and non-resonance parts, 
\begin{equation}\label{eq:profile1}
- i \partial_t a_K = \frac{\epsilon^2}{L^{2}} \sum_{\substack{\mathcal{S}_{3,K} =0\\ \Omega_{3,K}=0}} a_{K_1} \overline{a_{K_2}} a_{K_3} +\frac{\epsilon^2}{L^{2}} \sum_{\substack{\mathcal{S}_{3,K} =0\\ \Omega_{3,K}\ne0}} a_{K_1} \overline{a_{K_2}} a_{K_3} e(\Omega_{3,K}t).
\end{equation}

Since for any fixed $K$,
\begin{equation}\label{pNLS}
\left\{ \begin{array}{l}
\mathcal{S}_{3,K}(K_{1},K_{2},K_{3}) =0 \\
\Omega_{3,K}(K_1,K_2,K_3)=0
\end{array} \right.
\end{equation}
only has degenerated solutions $\{K_1,K_3\}=\{K,K_2\}$. The equation of $a_{K}$ can be written as:

\begin{align*}
- i \partial_t a_K &=\frac{\epsilon^2}{L^{2}} \sum_{\substack{\mathcal{S}_{3,K} =0\\ \Omega_{3,K}=0}} a_{K_1} \overline{a_{K_2}} a_{K_3} +\frac{\epsilon^2}{L^{2}} \sum_{\substack{\mathcal{S}_{3,K} =0\\ \Omega_{3,K}\ne0}} a_{K_1} \overline{a_{K_2}} a_{K_3} e(\Omega_{3,K}t)\\
&=\frac{\epsilon^2}{L^{2}}\left(2\sum_{K_1\in\Z_L} |a_{K_1}|^2-|a_K|^2\right)a_K+\frac{\epsilon^2}{L^{2}} \sum_{\substack{\mathcal{S}_{3,K} =0\\ \Omega_{3,K}\ne0}} a_{K_1} \overline{a_{K_2}} a_{K_3} e(\Omega_{3,K}t),
\end{align*}

Then multiplying the profile by a phase factor, one can remove the cubic resonance interaction. Notice that $\sum_{K_1\in\Z_L} |a_{K_1}|^2=L\norm{u}_{L^2}^2$ is conserved
and let $$b_K=e(-\frac{\epsilon^2}{\pi L^{2}}\sum_{K_1\in\Z_L} |a_{K_1}|^2t)a_K,$$ we have
\begin{align}\label{eq:profileB}
\begin{split}
- i \partial_t b_K =& e\left(-\frac{\epsilon^2}{\pi L^{2}}\sum_{K_1\in\Z_L} |a_{K_1}|^2t\right)\left(-i\partial_t a_K-\frac{2\epsilon^2}{L^{2}}\sum_{K_1\in\Z_L} |a_{K_1}|^2a_K\right)\\
=&-\frac{\epsilon^2}{L^{2}}|b_K|^2b_K+ \frac{\epsilon^2}{L^{2}} \sum_{\substack{\mathcal{S}_{3,K} =0\\ \Omega_{3,K}\ne0}} b_{K_1} \overline{b_{K_2}} b_{K_3} e(\Omega_{3,K}t).
\end{split}
\end{align}

Let 
$$d_K:=b_Ke(\frac{\epsilon^2}{2\pi L^2}\int_0^t|b_K(s)|^2ds)=a_Ke(-\frac{\epsilon^2}{\pi L}\norm{u}_{L^2}^2t+\frac{\epsilon^2}{2\pi L^2}\int_0^t|a_K(s)|^2ds),$$
note that $|d_K|=|b_K|=|a_K|=|u_K|$.

Define
$$\widetilde{\Omega}_{3,K}=\Omega_{3,K}-\frac{\epsilon^2}{2\pi L^2}(|d_{K_1}|^2-|d_{K_2}|^2+|d_{K_3}|^2-|d_K|^2)$$

and $$D_{3,K}(t)=|d_{K_1}|^2-|d_{K_2}|^2+|d_{K_3}|^2-|d_K|^2.$$

Then $d_K$ satisfies
\begin{align}\label{eq:profiled}
- i \partial_t d_K &=\frac{\epsilon^2}{L^{2}} \sum_{\substack{\mathcal{S}_{3,K} =0\\ \Omega_{3,K}\ne0}} d_{K_1} \overline{d_{K_2}} d_{K_3} e(\int_0^t\widetilde{\Omega}_{3,K}(s)ds).
\end{align}
So far, we have removed the cubic resonance interaction.

\subsection{Normal Form Transformation of the Cubic Term}\label{sec:nftcubic}
The normal form transformation means the process of separating the resonance and non-resonance interaction of the nonlinearity and using differentiation by parts argument on the nonresonance interaction. The nonresonance interaction are small due to cancellation of the phase factor. In the same spirit as oscillatory integral, an appropriate way of exploiting this cancellation is by integration by parts or differentiation by parts.

As in the oscillatory integral setting, if we have $\int_{\mathbb{R}} e(\lambda\phi(x)) f(x)dx$, as $\lambda$ goes to infinity, the oscillation in $x$ of the exponential factor $e(\lambda\phi(x))$ is getting stronger and stronger. So the value of this integral should be smaller and smaller. A common idea of exploiting the oscillation is to do the integration by parts argument in $x$, i.e. $e(\lambda\phi(x))=\frac{\partial_t e(\lambda\phi(x))}{2\pi i\lambda\phi'(x)}$. Each time of applying this argument we gain one $\frac{1}{\lambda}$ , provided no stationary points $x_0$ s.t. $\phi'(x_0)=0$.

Now our situation is very similar.We have a term like $d_{K_1} \overline{d_{K_2}}d_{K_{3}} e(\Omega_{3,K}t)$ which is oscillating rapidly when $t$ is large. We apply the identity $\frac{\partial_t e(\Omega_{3,K}t)}{\Omega_{3,K}}$ and get
$$d_{K_1} \overline{d_{K_2}} d_{K_{3}} e(\Omega_{3,K}t)
=\partial_t\left(\frac{d_{K_1} \overline{d_{K_2}} d_{K_{3}}e(\Omega_{3,K}t)}{\Omega_{3,K}}\right)-\frac{  e(\Omega_{3,K}t) }{\Omega_{3,K}}  \partial_t(d_{K_1} \overline{d_{K_2}}d_{K_{3}}).$$
The first term can be absorbed by $- i \partial_t d_K$ on the left hand side and the second term will give us $\epsilon^2$ if we substitute the equation of $\partial_t d_K$. But notice that because we changed phase, we have $e(\int_0^t\widetilde{\Omega}_{3,K}(s)ds)$ instead of $e(\Omega_{3,K}t)$. We do differentiation by parts on $e(\Omega_{3,K}t)$ since we have the restriction that $e(\Omega_{3,K}t)\ne0$ and leave a factor $e(-\frac{\epsilon^2}{2\pi L^2}\int_0^t D_{3,K}(s)ds)$. When the time derivative hit this factor, we gain $\epsilon^2/L^2$.

\begin{align}
- i \partial_t d_K =&\frac{\epsilon^2}{L^{2}} \sum_{\substack{\mathcal{S}_{3,K} =0\\ \Omega_{3,K}\ne0}} d_{K_1} \overline{d_{K_2}} d_{K_3} e(\int_0^t\widetilde{\Omega}_{3,K}(s)ds)
\\
=&\frac{\epsilon^2}{L^{2}} \sum_{\substack{\mathcal{S}_{3,K} =0\\ \Omega_{3,K}\ne0}} d_{K_1} \overline{d_{K_2}} d_{K_3} \frac{\partial_t e(\Omega_{3,K}t)}{2\pi i \Omega_{3,K}}e(-\frac{\epsilon^2}{2\pi L^2}\int_0^t D_{3,K}(s)ds)
\\
=&\frac{\epsilon^2}{L^{2}} \sum_{\substack{\mathcal{S}_{3,K} =0\\ \Omega_{3,K}\ne0}} \partial_t\big(d_{K_1} \overline{d_{K_2}} d_{K_3}\frac{1}{2\pi i\Omega_{3,K}}e(\int_0^t\widetilde{\Omega}_{3,K}(s)ds) \big)
\\
-&\frac{\epsilon^2}{L^{2}} \sum_{\substack{\mathcal{S}_{3,K} =0\\ \Omega_{3,K}\ne0}} \frac{1}{2\pi i\Omega_{3,K}} \partial_t\left(d_{K_1} \overline{d_{K_2}} d_{K_3}e(-\frac{\epsilon^2}{2\pi L^2}\int_0^t D_{3,K}(s)ds)\right)e(\Omega_{3,K}t)
\end{align}

Absorb the second term on the right hand side which is of the form $\partial_t(...)$ by the left hand side, 

\begin{align}\label{eq:ck}
\begin{split}
&- i \partial_t \Big(d_K-\frac{\epsilon^2}{L^{2}} \sum_{\substack{\mathcal{S}_{3,K} =0\\ \Omega_{3,K}\ne0}} d_{K_1} \overline{d_{K_2}} d_{K_3}\frac{1}{\Omega_{3,K}} e(\int_0^t\widetilde{\Omega}_{3,K}(s)ds) \Big)
\\
&=-\frac{\epsilon^2}{L^{2}} \sum_{\substack{\mathcal{S}_{3,K} =0\\ \Omega_{3,K}\ne0}} \frac{1}{2\pi i\Omega_{3,K}} \partial_t\left(d_{K_1} \overline{d_{K_2}} d_{K_3}e(-\frac{\epsilon^2}{2\pi L^2}\int_0^t D_{3,K}(s)ds)\right)e(\Omega_{3,K}t)
\\
&=-\frac{\epsilon^2}{L^{2}} \sum_{\substack{\mathcal{S}_{3,K} =0\\ \Omega_{3,K}\ne0}} \frac{1}{2\pi i\Omega_{3,K}} \partial_t\left(d_{K_1} \overline{d_{K_2}} d_{K_3}\right)e(\int_0^t\widetilde{\Omega}_{3,K}(s)ds)
\\
&+\frac{\epsilon^4}{L^{4}} \sum_{\substack{\mathcal{S}_{3,K} =0\\ \Omega_{3,K}\ne0}} \frac{1}{2\pi \Omega_{3,K}} d_{K_1} \overline{d_{K_2}} d_{K_3} D_{3,K}(t)e(\int_0^t\widetilde{\Omega}_{3,K}(s)ds)
\end{split}
\end{align}

Let
\begin{equation}
c_K^1=d_K-\frac{\epsilon^2}{L^{2}} \sum_{\substack{\mathcal{S}_{3,K} =0\\ \Omega_{3,K}\ne0}} d_{K_1} \overline{d_{K_2}} d_{K_3}\frac{1}{2\pi\Omega_{3,K}} e(\int_0^t\widetilde{\Omega}_{3,K}(s)ds),
\end{equation}
then 
\begin{align}\label{eq:profileC}
\begin{split}
- i \partial_t c^1_K &=-\frac{\epsilon^2}{L^{2}} \sum_{\substack{\mathcal{S}_{3,K} =0\\ \Omega_{3,K}\ne0}} \frac{1}{2\pi i\Omega_{3,K}} \partial_t\left(d_{K_1} \overline{d_{K_2}} d_{K_3}\right)e(\int_0^t\widetilde{\Omega}_{3,K}(s)ds)
\\
&+\frac{\epsilon^4}{L^{4}} \sum_{\substack{\mathcal{S}_{3,K} =0\\ \Omega_{3,K}\ne0}} \frac{1}{2\pi \Omega_{3,K}} d_{K_1} \overline{d_{K_2}} d_{K_3} D_{3,K}(t)e(\int_0^t\widetilde{\Omega}_{3,K}(s)ds)
\end{split}
\end{align}

When the time derivative hits one of the $d_{K_j}$, for example when it hits $d_{K_3}$, we have
\begin{align}\label{eq:profileC1}
\begin{split}
&-\frac{\epsilon^2}{L^{2}} \sum_{\substack{\mathcal{S}_{3,K} =0\\ \Omega_{3,K}\ne0}} \frac{1}{2\pi   i\Omega_{3,K}}d_{K_1} \overline{d_{K_2}} \partial_td_{K_3} e(\int_0^t\widetilde{\Omega}_{3,K}(s)ds)
\\
=&-\frac{\epsilon^2}{L^{2}} \sum_{\substack{K_1-K_2+K_3 =K\\ K_1^2-K_2^2+K_3^2\ne K^2}} \frac{e(\int_0^t\widetilde{\Omega}_{3,K_3}(s)ds)}{2\pi i \Omega_{3,K}(K_1,K_2,K_3)}\frac{\epsilon^2}{L^{2}} d_{K_1} \overline{d_{K_2}}\Big(\sum_{\substack{K_3=K'_3-K_4+K_5}} d_{K'_3}\overline{d_{K_4}} d_{K_5}\Big)e(\int_0^t\widetilde{\Omega}_{3,K}(s)ds) 
\end{split}
\end{align}
Relabel the indices, we have
\begin{align}
\begin{split}
&-\frac{\epsilon^2}{L^{2}} \sum_{\substack{K_1-K_2+K_3 =K\\ K_1\ne K_2,K_1\ne K}} \frac{1}{2\pi i \Omega_{3,K}(K_1,K_2,K_3)}\frac{\epsilon^2}{L^{2}} d_{K_1} \overline{d_{K_2}}\times\\
&\ \ \ \ \ \ \ \ \ \ \ \ \ \ \ \ \ \ \ \ \ \ \ \ \ \ \ \ \ \ \ \ \ \ \ \ \times\Big(\sum_{\substack{K_3=K'_3-K_4+K_5}} d_{K'_3}\overline{d_{K_4}} d_{K_5}e(\int_0^t(\widetilde{\Omega}_{3,K_3}(s)+\widetilde{\Omega}_{3,K}(s))ds)\Big)
\\
=&-\frac{\epsilon^4}{L^{4}} \sum_{\substack{K_1-K_2+K_3-K_4+K_5 =K\\ K_1\ne K_2,K_1\ne K}} \frac{d_{K_1} \overline{d_{K_2}}d_{K_3}\overline{d_{K_4}} d_{K_5}}{2\pi i \Omega_{3,K}(K_1,K_2,K_3-K_4+K_5)}\times\\ 
&\ \ \ \ \ \ \ \ \ \ \ \ \ \ \ \ \ \ \ \ \ \ \ \ \ \ \ \ \ \ \ \ \ \ \ \ \ \ \ \ \ \ \times e(\int_{0}^t\widetilde{\Omega}_{5,K}(K_1,K_2,K_3,K_4,K_5)(s)ds)
\end{split}
\end{align}
Do this to each term in the expansion of (\ref{eq:profileC}) and relabel the indices. We have
\begin{align}\label{eq:profileC3}
\begin{split}
&- i \partial_t c^1_K = \frac{i\epsilon^4}{2\pi L^{4}} \sum_{\substack{\mathcal{S}_{5,K} =0}} d_{K_1} \overline{d_{K_2}} d_{K_3} \overline{d_{K_4}} d_{K_5} e(\int_0^t\widetilde{\Omega}_{5,K}(s)ds)
\Big[\frac{1_{K_1\ne K,K_5\ne K, K_2\ne K_3, K_3\ne K_4}}{  \Omega_{3,K}(K_1,K_2-K_3+K_4,K_5)}\\
&-\frac{1_{K_1\ne K_2, K_2\ne K_3, K_4\ne K_5,K_5\ne K}}{  \Omega_{3,K}(K_1-K_2+K_3,K_4,K_5)}-\frac{1_{K_1\ne K_2,K_1\ne K, K_3\ne K_4, K_4\ne K_5}}{  \Omega_{3,K}(K_1,K_2,K_3-K_4+K_5)}\Big]
\\
&+\frac{\epsilon^4}{L^{4}} \sum_{\substack{\mathcal{S}_{3,K} =0\\ \Omega_{3,K}\ne0}} \frac{1}{2\pi \Omega_{3,K}} d_{K_1} \overline{d_{K_2}} d_{K_3} D_{3,K}(t)e(\int_0^t\widetilde{\Omega}_{3,K}(s)ds).
\end{split}
\end{align}
By doing normal form transformation once, we have transformed the cubic non-resonance interaction into a quintic form. 

Let us introduce the following sets

\begin{equation}\label{eq:defBi}
\begin{split}
    &B_1=\{K_1=K_2=K_3,K_3\ne K_4,K_4\ne K_5\},\\
    &B_2=\{K_1\ne K_2,K_2=K_3=K_4,K_4\ne K_5\},\\
    &B_3=\{K_1\ne K_2,K_2\ne K_3,K_3=K_4=K_5\},\\
    &B_4=\{K_1\ne K_2,K_2\ne K_3,K_4=K_5=K\}.
\end{split}
\end{equation}

Then the second term of the left hand side may be reformulated as 

\begin{equation}
\begin{split}
    &\frac{\epsilon^4}{L^{4}} \sum_{\substack{\mathcal{S}_{3,K} =0\\ \Omega_{3,K}\ne0}} \frac{1}{2\pi \Omega_{3,K}} d_{K_1} \overline{d_{K_2}} d_{K_3} D_{3,K}(t)e(\int_0^t\widetilde{\Omega}_{3,K}(s)ds)\\
    =&\frac{\epsilon^4}{L^{4}} \sum_{\substack{\mathcal{S}_{5,K} =0}} \frac{1_{B_1}-1_{B_2}+1_{B_3}-1_{B_4}}{2\pi \Omega_{5,K}} d_{K_1} \overline{d_{K_2}} d_{K_3} \overline{d_{K_4}} d_{K_5} e(\int_0^t\widetilde{\Omega}_{5,K}(s)ds)
\end{split}
\end{equation}

Denote $G^2_{K_1...K_5}=\frac{1_{B_1}-1_{B_2}+1_{B_3}-1_{B_4}}{2\pi \Omega_{5,K}}$ and
\begin{align*}
\begin{split}
&H^2_{K_1K_2\cdots K_5}=\frac{1_{A_2}}{  2\pi\Omega_{3,K}(K_1,K_2-K_3+K_4,K_5)}\\
&-\frac{1_{A_1}}{  2\pi\Omega_{3,K}(K_1-K_2+K_3,K_4,K_5)}-\frac{1_{A_3}}{  2\pi\Omega_{3,K}(K_1,K_2,K_3-K_4+K_5)},
\end{split}
\end{align*}

where
\begin{equation}\label{eq:defAi}
\begin{split}
&A_1=\{K_1\ne K_2,K_2\ne K_3,K_4\ne K_5,K_5\ne K\},
\\
&A_2=\{K_1\ne K,K_2\ne K_3,K_3\ne K_4,K_5\ne K\},
\\
&A_3=\{K_1\ne K_2,K_1\ne K,K_3\ne K_4,K_4\ne K_5\}.
\end{split}
\end{equation}

We rewrite (\ref{eq:profileC3}) as

\begin{equation}\label{eq:ck1}
 \begin{split}
- i \partial_t c_K^1&=\frac{\epsilon^4}{L^{4}} \sum_{\substack{\mathcal{S}_{5,K} =0}}H^2_{K_1K_2\cdots K_5} d_{K_1} \overline{d_{K_2}} d_{K_3} \overline{d_{K_4}} d_{K_5} e(\int_0^t\widetilde{\Omega}_{5,K}(s)ds)\\
&+\frac{\epsilon^4}{L^{4}} \sum_{\substack{\mathcal{S}_{5,K} =0}} G^2_{K_1...K_5} d_{K_1} \overline{d_{K_2}} d_{K_3} \overline{d_{K_4}} d_{K_5} e(\int_0^t\widetilde{\Omega}_{5,K}(s)ds).
\end{split}   
\end{equation}

\subsection{Higher Order Normal Form Transformations} 
In this subsection, we continue with normal form transformation up to $P$ times to remove the quintic form into higher order forms. Notice that in (\ref{eq:ck1}), we don't have the restriction that prevents $\Omega_{5,K}$ to vanish. So we shall do higher order normal form transformation by first splitting the sums in (\ref{eq:ck1}) into $\sum_{\substack{\mathcal{S}_{5,K} =0}}=\sum_{\substack{\mathcal{S}_{5,K} =0\\\Omega_{5,K}=0}}+\sum_{\substack{\mathcal{S}_{5,K} =0\\\Omega_{5,K}\ne0}}$, then always leaving the resonance interaction i.e. $\Omega_{5,K}=0$, and doing differentiation by parts on the nonresonance interaction i.e. $\Omega_{5,K}\ne0$. In doing the differentiation by parts we exploit
$$e(\int_0^t\widetilde{\Omega}_{5,K}(s)ds)=e(-\frac{\epsilon^2}{2\pi L^2}\int_{0}^t D_{5,K}(s) ds) \frac{\partial_{t}e(\Omega_{5,K}t)}{2\pi i\Omega_{5,K}}$$ and use differentiation by parts to move the $\partial_t$ to $d_{K_1}\overline{d_{K_2}}\cdots d_{K_{5}}$ or $e(-\frac{\epsilon^2}{2\pi L^2}\int_{0}^t D_{5,K}(s) ds)$. When the derivative hits the product of $d_{K_j}$, we substitute the equation for $\partial_td_{K_j}$ to produce higher order terms and keep do differentiation by parts to these term. When the derivative hits $e(-\frac{\epsilon^2}{2\pi L^2}\int_{0}^t D_{5,K}(s) ds)$, we get an extra $\frac{\epsilon^2}{L^2}D_{5,K}(s)$ factor. We denote this term by $S^{2}_K(u)$ and leave it unchanged in following normal form transformations. Then repeat this process for $7$ order terms and higher order terms. Let $c^d_K$ be the variable after $d$ times normal form transformation. By about analysis we may assume that the equations after $d$ normal form transformations writes

\begin{equation*}
\begin{split}
    -i\partial_tc^d_K=&\frac{\epsilon^4}{L^{4}}\sum_{\substack{\mathcal{S}_{5,K} =0\\\Omega_{5,K}=0}} H^2_{K_1K_2\cdots K_5} d_{K_1}\overline{d_{K_2}}\cdots d_{K_5}e(-\frac{\epsilon^2}{2\pi L^2}\int_0^tD_{5,K}(s)ds)
    \\
    +&\frac{\epsilon^6}{L^{6}}\sum_{\substack{\mathcal{S}_{7,K} =0\\\Omega_{7,K}=0}} (H^3_{K_1K_2\cdots K_7}+G^3_{K_1K_2\cdots K_7})d_{K_1}\overline{d_{K_2}}\cdots d_{K_7}e(-\frac{\epsilon^2}{2\pi L^2}\int_0^tD_{7,K}(s)ds)+\cdots
    \\
    +&\frac{\epsilon^{2d}}{L^{2d}}\sum_{\substack{\mathcal{S}_{2d+1,K} =0\\\Omega_{2d+1,K}=0}} (H^d_{K_1K_2...K_{2d+1}}+G^d_{K_1K_2...K_{2d+1}}) d_{K_1}\overline{d_{K_2}}\cdots d_{K_{2d+1}}e(-\frac{\epsilon^2}{2\pi L^2}\int_0^tD_{2d+1,K}(s)ds)
    \\
    +&\frac{\epsilon^{2d+2}}{L^{2d+2}}\sum_{\substack{\mathcal{S}_{2d+3,K} =0}} H^{d+1}_{K_1K_2...K_{2d+3}}d_{K_1}\overline{d_{K_2}}\cdots d_{K_{2d+3}} e(\int_0^t\widetilde{\Omega}_{2d+3,K}(s)ds)
    \\
    +&\frac{\epsilon^{2d+2}}{L^{2d+2}}\sum_{\substack{\mathcal{S}_{2d+3,K} =0}} G^{d+1}_{K_1K_2...K_{2d+3}}d_{K_1}\overline{d_{K_2}}\cdots d_{K_{2d+3}} e(\int_0^t\widetilde{\Omega}_{2d+3,K}(s)ds)
    \\
    +&\sum_{k=3}^{d+1} \frac{\epsilon^{2k}}{L^{2k}} S^{k}_{K}(u),
\end{split}
\end{equation*}
where \begin{equation*}
    \widetilde{\Omega}_{2d+3,K}(t)=\Omega_{2d+1,K}-\frac{\epsilon^2}{2\pi L^2}\sum_{j=1}^{2d+1}(-1)^{j-1}|d_{K_j}|^2(t),
\end{equation*}
\begin{equation*}
    D_{2d+1,K}=\sum_{j=1}^{2d+1}(-1)^{j-1}|d_{K_j}|^2(t).
\end{equation*}
Note that because $G^2_{K_1...K_5}=\frac{1_{B_1}-1_{B_2}+1_{B_3}-1_{B_4}}{2\pi \Omega_{5,K}}$ vanishes on the resonance surface $\{\mathcal{S}_{5,K} =0, \Omega_{5,K}=0\}$, we don't have $G^2_{K_1...K_5}$ in the first term on the right hand side. Do differentiation by parts in the fourth line and fifth line in above equation,

\begin{equation}
\begin{split}
    &d_{K_1}\overline{d_{K_2}}\cdots d_{K_{2d+3}} e(\int_0^t\widetilde{\Omega}_{2d+3,K}(s)ds)
    \\
    =&d_{K_1}\overline{d_{K_2}}\cdots d_{K_{2d+3}}e(-\frac{\epsilon^2}{2\pi L^2}\int_{0}^t D_{2d+3,K}(s) ds) \frac{\partial_{t}e(\Omega_{2d+3,K}t)}{2\pi i\Omega_{2d+3,K}}
    \\
    =&i\partial_{t}\left(d_{K_1}\overline{d_{K_2}}\cdots d_{K_{2d+3}} \frac{e(\int_0^t\widetilde{\Omega}_{2d+3,K}(s)ds)}{2\pi\Omega_{2d+3,K}}\right)
    \\
    -&\partial_{t}\left(d_{K_1}\overline{d_{K_2}}...d_{K_{2d+3}}e(-\frac{\epsilon^2}{2\pi L^2}\int_{0}^t D_{2d+3,K}(s) ds)\right)\frac{e(\Omega_{2d+3,K}t)}{2\pi i\Omega_{2d+3,K}}.
\end{split}
\end{equation}

If the $\partial_t$ in the second term hits $d_{K_1}\overline{d_{K_2}}...d_{K_{2d+3}}$. Then by substitute the equation of $\partial_td_K$ (\ref{eq:profiled}), we have

\begin{equation}
\begin{split}
    &-\frac{\epsilon^{2d+2}}{L^{2d+2}}\sum_{\substack{\mathcal{S}_{2d+3,K} =0}} H^{d+1}_{K_1K_2...K_{2d+3}}\partial_t(d_{K_1}\overline{d_{K_2}}\cdots d_{K_{2d+3}}) e(\int_0^t\widetilde{\Omega}_{2d+3,K}(s)ds)
    \\
    =&-\frac{\epsilon^{2d+4}}{L^{2d+4}}\sum_{j=1}^{2d+3}\sum_{\substack{\mathcal{S}_{2d+3,K} =0\\\Omega_{2d+3,K}\ne0}} H^{d+1}_{K_1K_2...K_{2d+3}}d_{K_1}...\\
    &\ \ \ \ \ \ \ \  \times\left(\sum_{\substack{S_{3,K_j}(M,N,L)=0}} d_M\overline{d_N}d_L e(\int_0^t\widetilde{\Omega}_{3,K_j}(M,N,L)(s) ds)\right)...d_{K_{2d+3}} \frac{e(\int_0^t\widetilde{\Omega}_{2d+3,K}(s)ds)}{2\pi \Omega_{2d+3,K}}
    \\
    =&\frac{\epsilon^{2d+4}}{L^{2d+4}}\sum_{\substack{\mathcal{S}_{2d+5,K} =0}}\Big(\sum_{j=1}^{2d+3}\frac{(-1)^j1_{\Omega_{2d+3,K}(K_1,...,K_j-K_{j+1}+K_{j+2},...,K_{2d+5})\ne0}H^{d+1}_{K_1...(K_j-K_{j+1}+K_{j+2})...K_{2d+5}}}{2\pi\Omega_{2d+3,K}(K_1,...,K_j-K_{j+1}+K_{j+2},...,K_{2d+5})}\Big)
    \\
   &\times d_{K_1}\overline{d_{K_2}}...d_{K_{2d+5}}e(\int_0^t\widetilde{\Omega}_{2d+5,K}(s)ds).
\end{split}
\end{equation}

The term involving $G$ can be treated the same as above. If the $\partial_t$ in the second term hits $e(-\frac{\epsilon^2}{2\pi L^2}\int_{0}^t D_{2d+3,K}(s) ds)$, we get

\begin{equation}
    \sum_{\substack{\mathcal{S}_{2d+3,K} =0}} (H^{d+1}_{K_1K_2...K_{2d+3}}+G^{d+1}_{K_1K_2...K_{2d+3}}) d_{K_1}\overline{d_{K_2}}...d_{K_{2d+3}}\frac{D_{2d+3,K}(t)}{2\pi\Omega_{2d+3,K}} e(\int_0^t\widetilde{\Omega}_{2d+3,K}(s)ds).
\end{equation}

Combining all above calculations, we have

\begin{align*}
    &-i\partial_t\Big(c^d_K-\frac{\epsilon^{2d+2}}{L^{2d+2}}\sum_{\substack{\mathcal{S}_{2d+3,K}=0\\\Omega_{2d+3,K}\ne0}} (H^{d+1}_{K_1K_2...K_{2d+3}}+G^{d+1}_{K_1K_2...K_{2d+3}}) d_{K_1}\overline{d_{K_2}}...d_{K_{2d+3}} \frac{e(\int_0^t\widetilde{\Omega}_{2d+3,K}(s)ds)}{2\pi\Omega_{2d+3,K}} \Big)
    \\
    =&\frac{\epsilon^4}{L^{4}}\sum_{\substack{\mathcal{S}_{5,K} =0\\\Omega_{5,K}=0}} H^2_{K_1K_2\cdots K_5} d_{K_1}\overline{d_{K_2}}\cdots d_{K_5}e(-\frac{\epsilon^2}{2\pi L^2}\int_0^tD_{5,K}(s)ds)
    \\
    +&\frac{\epsilon^6}{L^{6}}\sum_{\substack{\mathcal{S}_{7,K} =0\\\Omega_{7,K}=0}} (H^3_{K_1K_2\cdots K_7}+G^3_{K_1K_2\cdots K_7})d_{K_1}\overline{d_{K_2}}\cdots d_{K_7}e(-\frac{\epsilon^2}{2\pi L^2}\int_0^tD_{7,K}(s)ds)+\cdots
    \\
    +&\frac{\epsilon^{2d+2}}{L^{2d+2}}\sum_{\substack{\mathcal{S}_{2d+3,K} =0\\\Omega_{2d+3,K}=0}} (H^{d+1}_{K_1K_2...K_{2d+3}}+G^{d+1}_{K_1K_2...K_{2d+3}}) d_{K_1}\overline{d_{K_2}}\cdots d_{K_{2d+3}}e(-\frac{\epsilon^2}{2\pi L^2}\int_0^tD_{2d+3,K}(s)ds)
    \\
    +&\frac{\epsilon^{2d+4}}{L^{2d+4}}\sum_{\substack{\mathcal{S}_{2d+5,K} =0}}\Big(\sum_{j=1}^{2d+3}\frac{(-1)^j1_{\Omega_{2d+3,K}(K_1,...,K_j-K_{j+1}+K_{j+2},...,K_{2d+5})\ne0}H^{d+1}_{K_1...(K_j-K_{j+1}+K_{j+2})...K_{2d+5}}}{2\pi\Omega_{2d+3,K}(K_1,...,K_j-K_{j+1}+K_{j+2},...,K_{2d+5})}\Big)
    \\
    &\times d_{K_1}\overline{d_{K_2}}...d_{K_{2d+5}}e(\int_0^t\widetilde{\Omega}_{2d+5,K}(s)ds)
    \\
    +&\frac{\epsilon^{2d+4}}{L^{2d+4}}\sum_{\substack{\mathcal{S}_{2d+5,K} =0}}\Big(\sum_{j=1}^{2d+3}\frac{(-1)^j1_{\Omega_{2d+3,K}(K_1,...,K_j-K_{j+1}+K_{j+2},...,K_{2d+5})\ne0}G^{d+1}_{K_1...(K_j-K_{j+1}+K_{j+2})...K_{2d+5}}}{2\pi\Omega_{2d+3,K}(K_1,...,K_j-K_{j+1}+K_{j+2},...,K_{2d+5})}\Big)
    \\
    &\times d_{K_1}\overline{d_{K_2}}...d_{K_{2d+5}}e(\int_0^t\widetilde{\Omega}_{2d+5,K}(s)ds)
    +\sum_{k=3}^{d+1} \frac{\epsilon^{2k}}{L^{2k}} S^{k}_{K}(u)
    \\
    +&\frac{\epsilon^{2d+4}}{L^{2d+4}}\sum_{\substack{\mathcal{S}_{2d+3,K} =0}} (H^{d+1}_{K_1K_2...K_{2d+3}}+G^{d+1}_{K_1K_2...K_{2d+3}}) d_{K_1}\overline{d_{K_2}}...d_{K_{2d+3}}\frac{D_{2d+3,K}(t)}{2\pi\Omega_{2d+3,K}} e(\int_0^t\widetilde{\Omega}_{2d+3,K}(s)ds)
\end{align*}

From above calculation, we get the recurrence formulas for $c^{d+1}_{K}$, $H^{d+2}$, $G^{d+2}$ and $S^{d+2}_K(u)$.

\begin{equation}\label{eq:recurrencec^d_K}
    c^{d+1}_K=c^d_K-\frac{\epsilon^{2d+2}}{L^{2d+2}}\sum_{\substack{\mathcal{S}_{2d+3,K}=0\\\Omega_{2d+3,K}\ne0}} (H^{d+1}_{K_1K_2...K_{2d+3}}+G^{d+1}_{K_1K_2...K_{2d+3}}) d_{K_1}\overline{d_{K_2}}\cdots d_{K_{2d+3}} e(\int_0^t\widetilde{\Omega}_{2d+3,K}(s)ds),
\end{equation}

\begin{equation}\label{eq:recurrenceH^d_K}
\begin{split}
    &H^{d+2}_{K_1...K_{2d+5}}=\\
    &\sum_{j=1}^{2d+3}\frac{(-1)^j1_{K_j\ne K_{j+1}, K_{j+1}\ne K_{j+2}, \Omega_{2d+3,K}(K_1,...,K_j-K_{j+1}+K_{j+2},...,K_{2d+5})\ne0}H^{d+1}_{K_1...(K_j-K_{j+1}+K_{j+2})...K_{2d+5}}}{2\pi\Omega_{2d+3,K}(K_1,...,K_j-K_{j+1}+K_{j+2},...,K_{2d+5})},
    \end{split}
\end{equation}

\begin{equation}\label{eq:recurrenceG^d_K}
\begin{split}
    &G^{d+2}_{K_1...K_{2d+5}}=\\
    \sum_{j=1}^{2d+3}&\frac{(-1)^j1_{K_j\ne K_{j+1}, K_{j+1}\ne K_{j+2}, \Omega_{2d+3,K}(K_1,...,K_j-K_{j+1}+K_{j+2},...,K_{2d+5})\ne0}G^{d+1}_{K_1...(K_j-K_{j+1}+K_{j+2})...K_{2d+5}}}{2\pi\Omega_{2d+3,K}(K_1,...,K_j-K_{j+1}+K_{j+2},...,K_{2d+5})}.
\end{split}
\end{equation}

\begin{equation}\label{eq:recurrenceS^d_K}
\begin{split}
    &S^{d+2}_K(u)=\\
    &\frac{\epsilon^{2d+4}}{L^{2d+4}}\sum_{\substack{\mathcal{S}_{2d+3,K} =0}} (H^{d+1}_{K_1K_2...K_{2d+3}}+G^{d+1}_{K_1K_2...K_{2d+3}}) d_{K_1}\overline{d_{K_2}}...d_{K_{2d+3}}\frac{D_{2d+3,K}(t)}{2\pi\Omega_{2d+3,K}} e(\int_0^t\widetilde{\Omega}_{2d+3,K}(s)ds).
\end{split}
\end{equation}

Initially, 

\begin{equation}\label{eq:initialc}
    c_K^1=d_K-\frac{\epsilon^2}{L^{2}} \sum_{\substack{\mathcal{S}_{3,K} =0\\ \Omega_{3,K}\ne0}} d_{K_1} \overline{d_{K_2}} d_{K_3}\frac{1}{2\pi\Omega_{3,K}} e(\int_0^t\widetilde{\Omega}_{3,K}(s)ds),
\end{equation}

\begin{equation}\label{eq:initialH}
\begin{split}
H^2_{K_1K_2\cdots K_5}=&\frac{1_{A_2}}{  \Omega_{3,K}(K_1,K_2-K_3+K_4,K_5)}\\
&-\frac{1_{A_1}}{  \Omega_{3,K}(K_1-K_2+K_3,K_4,K_5)}-\frac{1_{A_3}}{  \Omega_{3,K}(K_1,K_2,K_3-K_4+K_5)},
\end{split}    
\end{equation}

\begin{equation}\label{eq:initialG}
    G^2_{K_1...K_5}=\frac{1_{B_1}-1_{B_2}+1_{B_3}-1_{B_4}}{2\pi \Omega_{5,K}}
\end{equation}

with $A_1,A_2,A_3$ defined in (\ref{eq:defAi}) $B_1,B_2,B_3.B_4$ defined in (\ref{eq:defBi}).

Combining (\ref{eq:recurrencec^d_K}), (\ref{eq:recurrenceH^d_K}), (\ref{eq:recurrenceG^d_K}) (\ref{eq:initialc}), (\ref{eq:initialH}), (\ref{eq:initialG}). we can calculate all the quantities in the normal form transformations.

Assume that we have done normal form transformation for $P$ times. We also introduce the following notations

\begin{equation}\label{eq:Hd_u}
    H^{d}_{K}(u)=\sum_{\substack{\mathcal{S}_{2d+1,K} =0\\\Omega_{2d+1,K}=0}} H^d_{K_1K_2...K_{2d+1}}d_{K_1}\overline{d_{K_2}}...d_{K_{2d+1}}e(\int_0^tD_{2d+1,K}(s)ds),
\end{equation}

\begin{equation}\label{eq:Gd_u}
    G^{d}_{K}(u)=\sum_{\substack{\mathcal{S}_{2d+1,K} =0\\\Omega_{2d+1,K}=0}} G^d_{K_1K_2...K_{2d+1}}d_{K_1}\overline{d_{K_2}}...d_{K_{2d+1}}e(\int_0^tD_{2d+1,K}(s)ds),
\end{equation}

\begin{equation}\label{eq:HP_u}
    \widetilde{H}^{P+1}_{K}(u)=\sum_{\substack{\mathcal{S}_{2P+3,K} =0}} H^{P+1}_{K_1K_2...K_{2P+3}}d_{K_1}\overline{d_{K_2}}...d_{K_{2P+3}} e(\int_0^t\widetilde{\Omega}_{2P+3,K}(s)ds),
\end{equation}

\begin{equation}\label{eq:GP_u}
    \widetilde{G}^{P+1}_{K}(u)=\sum_{\substack{\mathcal{S}_{2P+3,K} =0}} G^{P+1}_{K_1K_2...K_{2P+3}}d_{K_1}\overline{d_{K_2}}...d_{K_{2P+3}} e(\int_0^t\widetilde{\Omega}_{2P+3,K}(s)ds),
\end{equation}

\begin{equation}\label{eq:Fd_u}
    F^{d}_{K}(u)=\sum_{\substack{\mathcal{S}_{2d+1,K}=0\\\Omega_{2d+1,K}\ne0}} H^{d}_{K_1K_2...K_{2d+1}}d_{K_1}\overline{d_{K_2}}...d_{K_{2d+1}} \frac{1}{2\pi\widetilde{\Omega}_{2d+1,K}} e(\int_0^t\Omega_{2d+1,K}(s)ds).
\end{equation}

\begin{equation}\label{eq:Ed_u}
    E^{d}_{K}(u)=\sum_{\substack{\mathcal{S}_{2d+1,K}=0\\\Omega_{2d+1,K}\ne0}} G^{d}_{K_1K_2...K_{2d+1}}d_{K_1}\overline{d_{K_2}}...d_{K_{2d+1}} \frac{1}{2\pi\widetilde{\Omega}_{2d+1,K}} e(\int_0^t\Omega_{2d+1,K}(s)ds).
\end{equation}

Denote $c_K^P$ by $c_K$, we have
\begin{equation}
    c_K=d_K-\sum_{d=1}^{P}\frac{\epsilon^{2d}}{L^{2d}}F^{d}_{K}(u)-\sum_{d=1}^{P}\frac{\epsilon^{2d}}{L^{2d}}E^{d}_{K}(u),
\end{equation}

And the transformed equation writes

\begin{equation}\label{eq:transformedmain}
\begin{split}
    -i\partial_tc_K=\frac{\epsilon^{4}}{L^{4}}H_K^2(u)+\sum_{d=3}^P\frac{\epsilon^{2d}}{L^{2d}}(H^{d}_{K}(u)+G^{d}_{K}(u))+\sum_{d=3}^{P+1} \frac{\epsilon^{2d}}{L^{2d}} S^{d}_{K}(u)+\frac{\epsilon^{2P+2}}{L^{2P+2}}(\widetilde{H}^{P+1}_{K}(u)+\widetilde{G}^{P+1}_{K}(u))
\end{split}
\end{equation}
We shall extract our main term from $\frac{\epsilon^{4}}{L^{4}}H_K^2(u)$ by using vanishing property.

\section{The Feynman Diagrams and Cancellations of Resonances}\label{sec:Feynman}

In this section, we introduce the Feynman diagrams and get explicit formulas for coefficients $H^{d}_{K_1K_2...K_{2d+1}}$ and $G^d_{K_1K_2...K_{2d+1}}$ of normal form transformation. This formula applies to all dispersive equations. Then we apply an induction argument on the level of trees in our Feynman diagram to prove a vanishing property of the coefficient of normal form transformation. 

Feynman diagrams are very convenient notation of terms in repeated normal form transformations. In our paper, they are forests with each tree in them corresponds to a term in $H^d$, $\widetilde{H}^P$, $G^d$ and $\widetilde{G}^P$. Each node of bottoms of these trees correspond to summation of indices in these terms. Sometimes we do not distinguish $H^d$ and $\widetilde{H}^P$, $G^d$ and $\widetilde{G}^P$ because they have the same coefficient. And in the later, we just explain how Feynman diagram technique for coefficients $H$, coefficients $G$ can be treated in a very similar way.

\subsection{Feynman Diagrams and Repeated Normal Form Transformations}\label{sec:3.1}

Let's now describe how Feynman diagrams work. Assume that We have done normal form transformations for $P$-times. We shall draw a forest to denote all terms. 

We start with $P=0$, the case when we only have a cubic nonlinearity (\ref{eq:profiled}). We draw root node $R$ and three children which correspond to three summation indices $K_1,K_2,K_3$ in the cubic term. We call the level of $K_1,K_2,K_3$ level 0 to refer that we have not done any normal form. Later on each node of level $d$ will correspond to a summation index of a $2d+3$-linear form. 

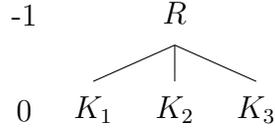
\begin{figure}[H]
\centering
\begin{forest}
[, phantom 
   [$R$, name=level-1 [$K_{1}$, name=level0  ]
                     [$K_{2}$ ]
                     [$K_{3}$ ]  
   ] 
]
{
\setcounter{levelcount}{0}
\countnodes{level-1}{2}
\countnodes{level0}{2}
}
\end{forest}
\caption{Feynman diagram before normal form transformation}
\end{figure}
When $P=1$, we have done the normal form transformation once. From (\ref{eq:profileC}), we know that three new terms will appear in $\widetilde{H}_{2}$, according to which $d_{K_j}$ gets $\partial_t$. We add three trees in our forests with branching at $K_j$ when $\partial_t$ hits $d_{K_j}$.

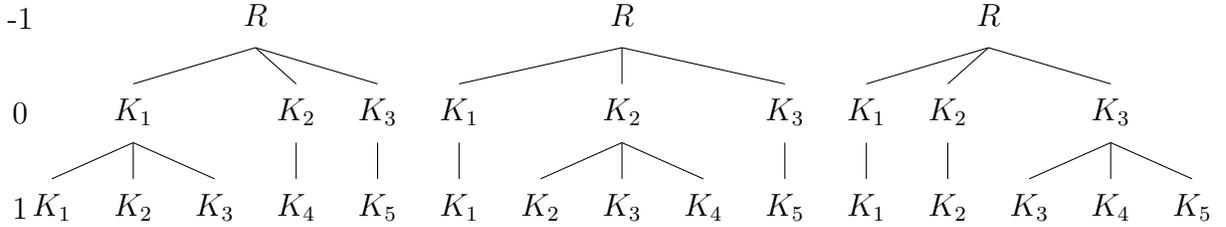
\begin{figure}[H]
\centering
\begin{forest}
[,phantom
  [$R$, name=level-1   [$K_{1}$, name=level0 [$K_{1}$, name=level1] 
              [$K_{2}$]
              [$K_{3}$]
         ]
         [$K_{2}$ [$K_{4}$]
         ]
         [$K_{3}$ [$K_{5}$]
         ]  
  ]
  [$R$   [$K_{1}$ [$K_{1}$]
         ]
         [$K_{2}$ [$K_{2}$] 
                [$K_{3}$]
                [$K_{4}$]
         ]
         [$K_{3}$ [$K_{5}$]
         ]  
  ]
  [$R$   [$K_{1}$ [$K_{1}$] 
         ]
         [$K_{2}$ [$K_{2}$]
         ]
         [$K_{3}$ [$K_{3}$] 
                  [$K_{4}$]
                  [$K_{5}$]
         ]  
  ]
]
{
\setcounter{levelcount}{0}
\countnodes{level-1}{8}
\countnodes{level0}{8}
\countnodes{level1}{8}
}
\end{forest}
\caption{Feynman diagram of one normal form transformation}\label{pic:one}
\end{figure}
Note that each tree corresponds to a $5-$linear term in $H^2$ (\ref{eq:profileC3}), (\ref{eq:initialH}). For example, the first tree in (\ref{pic:one})corresponds to 

$$-\frac{\epsilon^4}{2\pi L^{4}} \sum_{\substack{\mathcal{S}_{5,K} =0\\\Omega_{5,K} =0}} d_{K_1} \overline{d_{K_2}} d_{K_3} \overline{d_{K_4}} d_{K_5} e(-\frac{\epsilon^2}{2\pi L^2}\int_0^tD_{5,K}(s)ds)\frac{1_{K_4\ne K_5,K_5\ne K}}{  \Omega_{3,K}(K_1-K_2+K_3,K_4,K_5)}.$$

Nodes in the bottom of the first tree correspond to summation indices $K_1,...,K_5$ in above sum. The $K_1$, $K_2$, $K_3$ nodes are summation indices in previous cubic term $H^1$. Recall that above sum comes from substitute $\partial_t d_{K_1}$ in $H_1$ in $\sum\partial_td_{K_1}d_{K_2}d_{K_3}$ when $\partial_t$ hits the first one in $d_{K_1}d_{K_2}d_{K_3}$. When doing this, we get $\sum d_{K_1'}d_{K_2'}d_{K_3'}d_{K_2}d_{K_3}$ from $\sum\partial_td_{K_1}d_{K_2}d_{K_3}$ with $K_1=K_1'-K_2'+K_3'$. Then we relabel the indices $K_1'\rightarrow K_1$, $K_2'\rightarrow K_2$, $K_3'\rightarrow K_3$, $K_2\rightarrow K_4$, $K_3\rightarrow K_5$. So we find that $K_1$, $K_2$, $K_3$ on the bottom come from $K_1$ of level 0. This is they are children of $K_1$. For the same reason $K_4$ is a child of $K_2$, $K_5$ is a child of $K_5$. Because $K_1=K_1'-K_2'+K_3'$, $\Omega(K_1, K_2, K_3)\rightarrow\Omega_{3,K}(K_1-K_2+K_3,K_4,K_5)$

When $P=2$, we have done the normal form transformation twice. By (\ref{eq:recurrenceH^d_K}), one term in previous 5-linear term generates 5 new terms in $\widetilde{H}^{P+1}=\widetilde{H}^{3}$, according to which $d_{K_j}$ gets $\partial_t$ in the second transformation. We add trees accordingly. For example, from the first tree in (\ref{pic:one}), we have
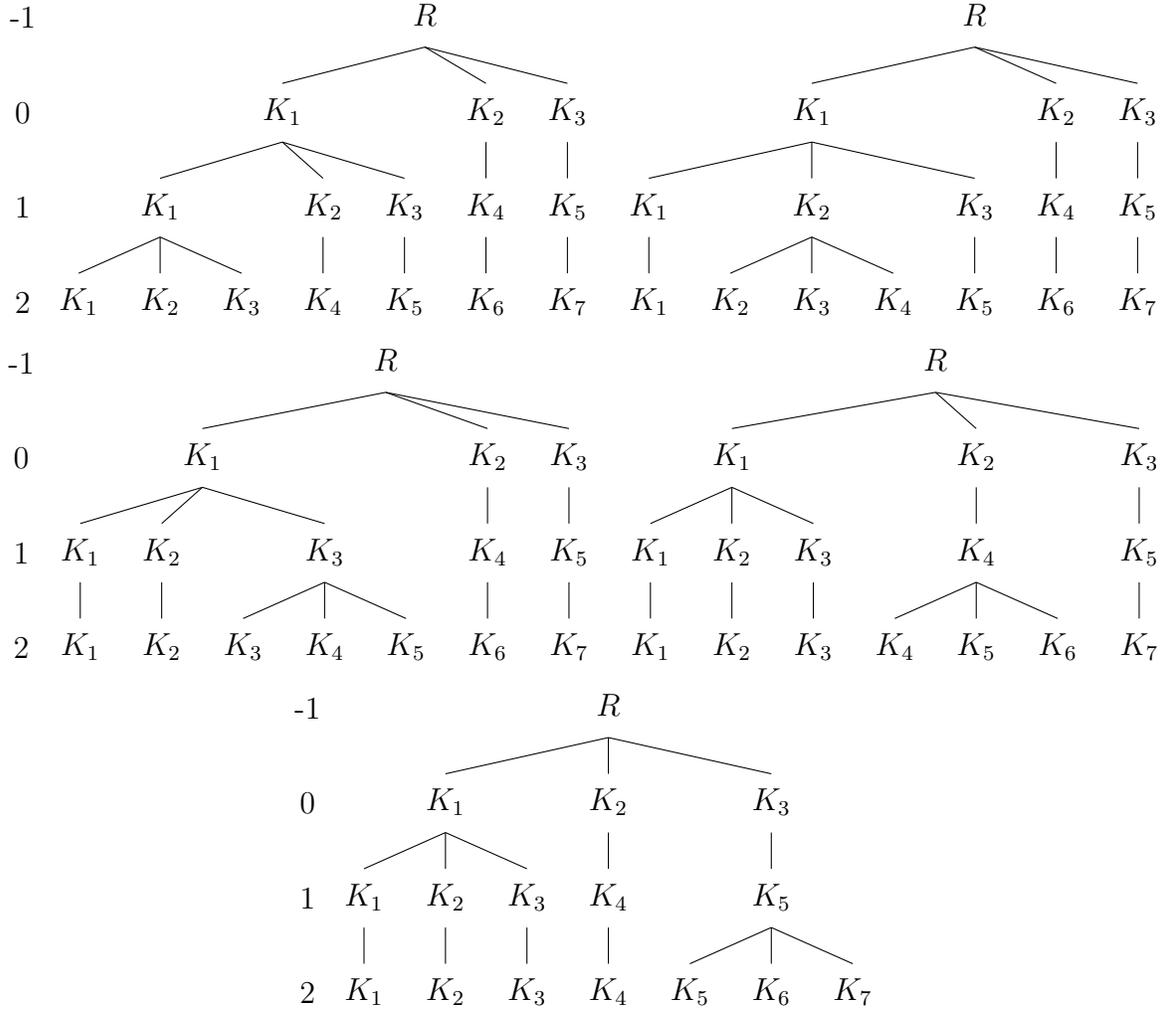
\begin{figure}[H]
\centering
\begin{forest}
[,phantom
  [$R$, name=level-1   [$K_1$, name=level0 [$K_1$, name=level1 [$K_1$, name=level2]
                       [$K_2$]
                       [$K_3$]
                ]
                [$K_2$ [$K_4$]]
                [$K_3$ [$K_5$]]
         ]
         [$K_2$ [$K_4$ [$K_6$]]
         ]
         [$K_3$ [$K_5$ [$K_7$]]
         ]  
  ]
  [$R$   [$K_1$ [$K_1$ [$K_1$]]
                [$K_2$ [$K_2$]
                       [$K_3$]
                       [$K_4$]
                ]
                [$K_3$ [$K_5$]]
         ]
         [$K_2$ [$K_4$ [$K_6$]]
         ]
         [$K_3$ [$K_5$ [$K_7$]]
         ]  
  ]
]
{
\setcounter{levelcount}{0}
\countnodes{level-1}{9}
\countnodes{level0}{9}
\countnodes{level1}{9}
\countnodes{level2}{9}
}
\end{forest}

\begin{forest}
[,phantom
  [$R$, name=level-1   [$K_1$, name=level0 [$K_1$, name=level1 [$K_1$, name=level2]] 
                [$K_2$ [$K_2$]]
                [$K_3$ [$K_3$]
                       [$K_4$]
                       [$K_5$]
                ]
         ]
         [$K_2$ [$K_4$ [$K_6$]]
         ]
         [$K_3$ [$K_5$ [$K_7$]]
         ]  
  ]
  [$R$   [$K_1$ [$K_1$ [$K_1$]]
                [$K_2$ [$K_2$]]
                [$K_3$ [$K_3$]]
         ]
         [$K_2$ [$K_4$ [$K_4$]
                       [$K_5$]
                       [$K_6$]
                ]
         ]
         [$K_3$ [$K_5$ [$K_7$]]
         ]  
  ]
]
{
\setcounter{levelcount}{0}
\countnodes{level-1}{8.5}
\countnodes{level0}{8.5}
\countnodes{level1}{8.5}
\countnodes{level2}{8.5}
}
\end{forest}

\begin{forest}
[,phantom
  [$R$, name=level-1   [$K_1$, name=level0 [$K_1$, name=level1 [$K_1$, name=level2]]
                [$K_2$ [$K_2$]]
                [$K_3$ [$K_3$]]
         ]
         [$K_2$ [$K_4$ [$K_4$]]
         ]
         [$K_3$ [$K_5$ [$K_5$]
                       [$K_6$]
                       [$K_7$]
                ]
         ] 
    ]
]  
]
{
\setcounter{levelcount}{0}
\countnodes{level-1}{4}
\countnodes{level0}{4}
\countnodes{level1}{4}
\countnodes{level2}{4}
}
\end{forest}
\caption{Feynman diagram of two normal form transformations}
\end{figure}

Nodes with the same root node correspond to summation indices of the same sum. Note that each time we do normal form transformation will introduce $\frac{1}{\Omega}$, see (\ref{eq:recurrenceH^d_K})). When finishing $P$ normal form transformations, we will have a product of $P$ different $\frac{1}{\Omega}$. Therefore, we have denominator of $\widetilde{H}^{P+1}$ in the form of  $$\Omega_{3,K}(*,*,*)\Omega_{5,K}(*,...,*)\cdots\Omega_{2P+1,K}(*,...,*).$$   
For example, when $P=2$, the denominator of $H^3$ is of the form
$$\Omega_{3,K}(*,*,*)\Omega_{5,K}(*,*,*,*,*).$$ These $*$ can be decided as following. Assume that we are considering the term corresponding to 
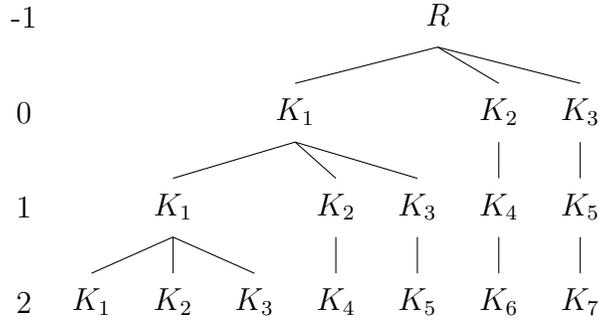
\begin{center}
\begin{forest}
[, phantom  
[$R$, name=level-1   [$K_1$, name=level0 [$K_1$, name=level1 [$K_1$, name=level2]
                     [$K_2$]
                     [$K_3$]
              ]
              [$K_2$ [$K_4$]]
              [$K_3$ [$K_5$]]
       ]
       [$K_2$ [$K_4$ [$K_6$]]
       ]
       [$K_3$ [$K_5$ [$K_7$]]
       ]  
]
]
{
\setcounter{levelcount}{0}
\countnodes{level-1}{5.5}
\countnodes{level0}{5.5}
\countnodes{level1}{5.5}
\countnodes{level2}{5.5}
}
\end{forest}
\end{center}
Because $\Omega_{3,K}(*,*,*)$ comes from the first normal form transformation, its three entries correspond to summation indices of $\widetilde{H}^{1}$ before normal form transformation which are $K_1$, $K_2$, $K_3$ on level 0. After doing the first normal form transformation, $K_1\rightarrow K_1-K_2+K_3$, $K_2\rightarrow K_4$ and $K_1\rightarrow K_5$, where these new $K_j$ are indices on level 1. After second normal form transformation, $K_1$ on level 1 becomes $K_1-K_2+K_3$, $K_1$ of level 0 becomes $K_1\rightarrow K_1-K_2+K_3\rightarrow K_1-K_2+K_3-K_4+K_5$. From this calculation we know that an index $K_j$ will finally become the alternating sum of the indices on the bottom of the subtree rooting on node $K_j$. For example,  $K_1$ of level 0 becomes $K_1\rightarrow K_1-K_2+K_3-K_4+K_5$ which is the alternating sum of the bottom of 
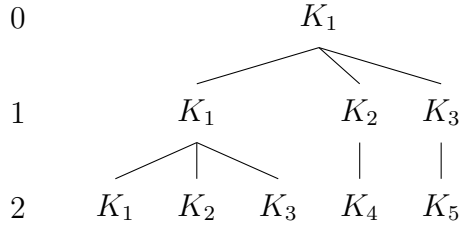
\begin{center}
\begin{forest}
[$K_1$, name=level-1 [$K_1$, name=level0 [$K_1$, name=level1]
              [$K_2$]
              [$K_3$]
       ]
       [$K_2$ [$K_4$]]
       [$K_3$ [$K_5$]]
]
{
\setcounter{levelcount}{1}
\countnodes{level-1}{4}
\countnodes{level0}{4}
\countnodes{level1}{4}
}
\end{forest}
\end{center}

$K_2$ of level 0 becomes $K_2\rightarrow K_6$ which is the alternating sum of the bottom of 
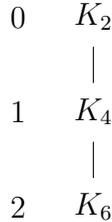
\begin{center}
\begin{forest}
[$K_2$, name=level-1 [$K_4$, name=level0 [$K_6$, name=level1]]]
{
\setcounter{levelcount}{1}
\countnodes{level-1}{1}
\countnodes{level0}{1}
\countnodes{level1}{1}
}
\end{forest}
\end{center}

Therefore, in $\Omega_{3,K}(*,*,*)$, $*$ is the alternating sum of the bottom of the subtree rooting on corresponding nodes.

$\Omega_{5,K}(*,*,*,*,*)$ is exactly the same as above, noticing that $\Omega_{5,K}$ comes from second transformation and five $*$ should correspond to five indices on level 1. Therefore, we have the denominator of $H^3$ is of the form
$$\Omega_{3,K}(K_1-K_2+K_3-K_4+K_5,K_6,K_7)\Omega_{5,K}(K_1-K_2+K_3,K_4,K_5,K_6,K_7).$$

\subsection{An Algorithmic Description of the Construction of Feynman Diagram} In this section, we shall provide an algorithm for drawing a Feynman diagram.

In what follows, we shall denote a Feynman diagram after $P-$th normal form transformation $\CG_P$ and it corresponds to terms in $\widetilde{H}^{P+1}_K(u)$. In general, we use $\CG_d$ to denote the forest after $d$ times normal form transformation. Notice that this means each tree in $\CG_d$ has $d+2$ level and the last level is level $d$.

\begin{enumerate}

    \item Initially, we put a $R$ node on level $-1$ with 3 children $K_1,K_2,K_3$, which correspond to the summation indices in $\widetilde{H}^{1}_K(u)$.  This completes $\CG_0$. We proceed inductively by drawing $\CG_{d+1}$ given $\CG_d$. 
   
   \item For a given tree $T$ in $\CG_d$, its bottom contains nodes $K_1,\cdots,K_{2d+3}$. We generate $2d+3$ trees with one more level. The $m$-th ($m=1,...,2d+3$) tree is generated by adding three unlabelled children of $K_j$ to $T$ if $j=m$ and adding one unlabelled children of $K_j$ if $j\ne m$. $\CG_{d+1}$ will be the collection of all these new trees. See also the figure below.

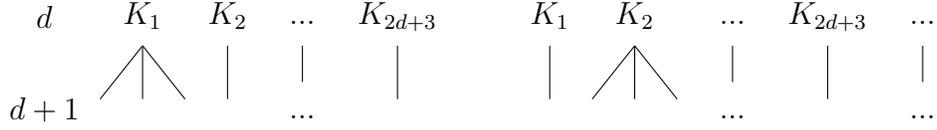
\begin{figure}[H]
\centering
 \begin{forest}
   [, phantom
     [$K_1$, name=leveld [, name=leveld1]
            []
            []
     ]
     [$K_2$ []] 
     [... [...]] 
     [$K_{2d+3}$ []]
   ]
   {
    \draw let \p{L} = (leveld) in (-3,\y{L}) node {$d$}; 
    \draw let \p{L} = (leveld1) in (-3,\y{L}) node {$d+1$};
   }
   \end{forest} \qquad \begin{forest}
   [, phantom
     [$K_1$ []
     ]
     [$K_2$ []
            []
            []
     ] 
     [... [...]] 
     [$K_{2d+3}$ []]
     [... [...]]
   ]
   \end{forest}
   \caption{Generating new level with $m=1$}\label{pic:newlevel1}
   \end{figure}

   \item The last step is to label all new nodes by $K_1$ to $K_{2d+5}$ from left to right. 
\end{enumerate}

\subsection{Correspondence of Trees and Terms}
We summarize the correspondence between the trees and terms in the coefficients of $\widetilde{H}^{d+1}_K(u)$, $\widetilde{H}^{d+1}_K(u)$ in (\ref{eq:formulaH^d}), (\ref{eq:formulaG^d}). We shall mainly focus on these coefficient $H^{d+1}_{K_1...K_{2d+3}}$, $G^{d+1}_{K_1...K_{2d+3}}$.  

A distinguished feature of (\ref{eq:formulaH^d}) is the apperance of the inequality constrains $1_{\CA}$, $1_{\CB}$. It's very important to keep track on these inequality constrains. Because without them, some denominators $\Omega$ will vanish. And there will be a large number of trivial solutions to the defining equation of (\ref{eq:mainnumbertheorydef}), which prevent (\ref{eq:mainnumbertheoryineq}) to be true.

\begin{lemma}\label{l:formulaH^d} We have the following description of $H^{d+1}_{K_1...K_{2d+3}}$, $G^{d+1}_{K_1...K_{2d+3}}$
\begin{enumerate}
    \item In general, $H^{d+1}_{K_1...K_{2d+3}}$ and $G^{d+1}_{K_1...K_{2d+3}}$ is a sum of rational functions of $K_1,...,K_{2d+3},K$ multiplied by indicator functions. Any $T\in\CG_d$ corresponds to one term in $H^{d+1}_{K_1...K_{2d+3}}$ or $G^{d+1}_{K_1...K_{2d+3}}$.
    
    \item For all subtree $T'$ of a tree $T$ in a Feynman diagram, denote $$S_{T'}=\text{Alternating sum of the indices of the bottom of $T'$}$$ (the alternating sum of a sequence $K_1$, $K_2$, ... is $K_1-K_2+K_3-...$). And for any node $*$ in a diagram denote by $T_*$ the subtree which has root $*$. 
    
    Then the corresponding term of $T$ in $H^{d+1}_{K_1...K_{2d+3}}$ is
    \begin{equation}
        \frac{\epsilon_{\CA}}{\Omega_{3,K}(T)\Omega_{5,K}(T)\cdots\Omega_{2d+1,K}(T)}.
    \end{equation}
    
    And the corresponding term of $T$ in $G^{d+1}_{K_1...K_{2d+3}}$ is
    \begin{equation}
        \frac{\epsilon_{\CB}}{\Omega_{5,K}(T)^2\Omega_{7,K}(T)\cdots\Omega_{2d+1,K}(T)}.
    \end{equation}
    
    Here
    \begin{equation}\label{eq:formulaOmega}
    \Omega_{2k+1,K}(T)=2\pi\Omega_{2k+1,K}(S_{T_{*_1}},S_{T_{*_2}},\cdots,S_{T_{*_{2k+1}}})
    \end{equation}
    for $k=1,2,...,d$, where $*_{j}$ is the $j$-th nodes from left to right on the level $k-1$ of the tree $T$. $\epsilon_{\CA}$ are functions of $K_1$, ..., $K_{2d+3}$, $K$ taking value in $\{-1,0,1\}$ and $\epsilon_{\CB}$ is a sum of four such functions.

    \item Let $l^T_{k}$ be the number such that $*_{l^T_{k}}$ is the only node of level $k$ that has a $3-$branching. Assume $*'$, $*''$ and $*'''$ are three children of the branching on level $k-1$. Let $S^1_{k}=S_{T_{*'}}$, $S^2_{k}=S_{T_{*''}}$, $S^3_{k}=S_{T_{*'''}}$, $\CA_{k}(T)=\CB_{k}(T)=\{S^1_{k}\ne S^2_{k},S^2_{k}\ne S^3_{k},\Omega_{2k+3,K}(T)\ne0\}$, 
    $\CB^{(s)}_{0}(T)=\{S^{s}_{1}=S^{s+1}_{1}=S^{s+2}_{1}\}$ $(1\le s\le 3)$, $\CB^{(4)}_{0}(T)=\{S^{4}_{1}=S^{5}_{1}=K\}$. Then in above equations $$\epsilon_{\CA}=(-1)^{\sum_{k=1}^{d-1}l^T_{k}}\prod^{d-1}_{k=1}1_{\CA_{k}(T)},$$ $$\epsilon_{\CB}=\sum_{s=1}^{4}{(-1)^{s+\sum^{d-1}_{k=2}l^T_{k}}1_{\CB^{(s)}_{0}(T)}\prod^{d-1}_{k=1}1_{\CB_{k}(T)}}.$$ 
    Thus
    \begin{equation}\label{eq:formulaH^d}
        H^{d+1}_{K_1...K_{2d+3}}=\sum_{T\in\CG_d}\frac{(-1)^{\sum_{k=0}^{d-1}l^T_{k}}\prod^{d-1}_{k=0}1_{\CA_{k}(T)}}{\Omega_{3,K}(T)\Omega_{5,K}(T)\cdots\Omega_{2d+1,K}(T)}(K_1,...,K_{2d+3}).
    \end{equation}
    
    \begin{equation}\label{eq:formulaG^d}
        G^{d+1}_{K_1...K_{2d+3}}=\sum_{s=1}^{4}\sum_{T\in\CG_d}\frac{(-1)^{s+\sum^{d-1}_{k=2}l^T_{k}}1_{\CB^{(s)}_{0}(T)}\prod^{d-1}_{k=1}1_{\CB_{k}(T)}}{\Omega_{5,K}(T)^2\Omega_{7,K}(T)\cdots\Omega_{2d+1,K}(T)}(K_1,...,K_{2d+3})d_{K_1}\cdots d_{K_{2d+3}}
    \end{equation}
\end{enumerate}
\end{lemma}
\begin{proof} The result of $G$-coefficients can be proved in the same way as the that of $H$-coefficients. So we only consider the proof for $H$-coefficients. 

Note that (1) follows easily from (2), (3). We shall prove (2), (3) by an induction on $d$. Assume that (\ref{eq:formulaH^d}) holds for $1,2,...,d+1$, let's prove the case $d+2$, i.e.

\begin{equation}
   H^{d+2}_{K_1...K_{2d+5}}=\sum_{T\in\CG_{d+1}}\frac{(-1)^{\sum_{k=0}^{d}l^T_{k}}\prod^{d}_{k=0}1_{\CA_{k}}}{\Omega_{3,K}(T)\Omega_{5,K}(T)\cdots\Omega_{2d+3,K}(T)}(K_1,...,K_{2d+5})
\end{equation}

By (\ref{eq:recurrenceH^d_K}), we have 

\begin{equation}\label{eq:formulaH^d'}
\begin{split}
    H^{d+2}_{K_1...K_{2d+5}}&=\sum_{j=1}^{2d+3}\frac{(-1)^j1_{K_j\ne K_{j+1},K_{j+1}\ne K_{j+2},\Omega_{2d+3}(\cdots)\ne0}H^{d+1}_{K_1...(K_j-K_{j+1}+K_{j+2})...K_{2d+5}}}{2\pi\Omega_{2d+3,K}(K_1,...,K_j-K_{j+1}+K_{j+2},...,K_{2d+5})}
\end{split}
\end{equation}

Now let's calculate the sum by substituting the induction assumption into (\ref{eq:formulaH^d'}).

\begin{equation}\label{eq:formulaTerm_j^d+1}
\begin{split}
    H^{d+2}_{K_1...K_{2d+5}}=&\sum_{T\in\CG_d}\frac{(-1)^{\sum_{k=0}^{d-1}l^T_{k}}\prod^{d-1}_{k=0}1_{\CA_{k}}(-1)^j1_{K_j\ne K_{j+1},K_{j+1}\ne K_{j+2},\Omega_{2d+3}(\cdots)\ne0}} {\Omega_{3,K}(T)\Omega_{5,K}(T)\cdots\Omega_{2d+1,K}(T)(K_1,...,K_j-K_{j+1}+K_{j+2},...,K_{2d+5})}\\
    &\times\frac{1}{2\pi\Omega_{2d+3,K}(K_1,...,K_j-K_{j+1}+K_{j+2},...,K_{2d+5})}
\end{split}
\end{equation}

We want to show that the term corresponding to $T$ in above summation equals to 

\begin{equation}\label{eq:formulaTerm_j^d+1'}
    \frac{(-1)^{\sum_{k=0}^{d}l^{T_j}_{k}}\prod^{d}_{k=0}1_{\CA_{k}}}{\Omega_{3,K}(T_j)\Omega_{5,K}(T_j)\cdots\Omega_{2d+3,K}(T_j)}(K_1,...,K_{2d+5}),
\end{equation} 
the term corresponding to $j-$th tree $T_j$ generated from $T$. To do this, we only need to verify following equations.

\begin{equation}\label{eq:Omega=Omega}
    \Omega_{2d+3,K}(T_j)(K_1,...,K_{2d+5})=2\pi\Omega_{2d+3,K}(K_1,...,K_j-K_{j+1}+K_{j+2},...,K_{2d+5}),
\end{equation}

\begin{equation}\label{eq:1=1}
    l_{d+1}^{T_j}=(-1)^j,\ \CA_{d+1}(T_j)=\{K_j\ne K_{j+1},K_{j+1}\ne K_{j+2},\Omega_{2d+3}(\cdots)\ne0\},
\end{equation}

\begin{equation}\label{eq:Omega=Omega'}
    \Omega_{2k+3,K}(T_j)(K_1,...,K_{2d+5})=\Omega_{2k+3,K}(T)(K_1,...,K_j-K_{j+1}+K_{j+2},...,K_{2d+5})\ \ \text{for}\ k\le d,
\end{equation}
and
\begin{equation}\label{eq:1=1'}
    l_{k}^{T_j}=l_k^T,\ 1_{\CA_{k}(T_j)}(K_1,...,K_{2d+5})=1_{\CA_{k}(T)}(K_1,...,K_j-K_{j+1}+K_{j+2},...,K_{2d+5})\ \ \text{for}\ k\le d-1.
\end{equation}

\underline{Proof of (\ref{eq:Omega=Omega}) and (\ref{eq:1=1}):}
When generate next level, for $K_i$ node on the bottom, according to $i<j$, $i>j$ and $i=j$, $K_i$ should have one child labelled by $K_i$, one child labelled by $K_{i+2}$ or three child labelled by $K_j$, $K_{j+1}$, $K_{j+2}$, in the new graph $T_j$.

Thus in $\Omega_{2d+3,K}(T_j)$, all $S_{T_{*_{i}}}=K_i$, $K_{i+2}$ or $K_j-K_{j+1}+K_{j+2}$ according to $i<j$, $i>j$ or $i=j$. Thus by (\ref{eq:formulaOmega}) 
\begin{equation}
    \Omega_{2d+3,K}(T_j)(K_1,...,K_{2d+5})=2\pi\Omega_{2d+3,K}(K_1,...,K_j-K_{j+1}+K_{j+2},...,K_{2d+5})
\end{equation}

Hence, we have proved (\ref{eq:Omega=Omega}). (\ref{eq:1=1}) can be proved similarly.

\underline{Proof of (\ref{eq:Omega=Omega'}) and (\ref{eq:1=1'}):}

Because $\Omega_{2k+3,K}(T)$ and  $1_{\CA_{k}(T)}$ are functions of $S_{T_{*}}$ and by obvious reason $l_{k}^{T_j}=l_k^T$, we only need to verify that 

$$S_{T_{j,*_n}}(K_1,...,K_{2d+5})=S_{T_{*_n}}(K_1,...,K_j-K_{j+1}+K_{j+2},...,K_{2d+5}).$$

Denoted by $*_1,...,*_{2k+3}$ the $2k+3$ nodes on level $k$. One of the subtrees $T_{*_{n_0}}$ has bottom containing $j$-th node $K_j$ on the bottom of $T$. Denote the subtree rooting at $*$ in $T_j$ by $T_{j,*}$. Denote the bottom of $T_{*_n}$ by $\{K_{i_n+1},...,K_{i_{n+1}}\}$ for $0=i_1<i_2<\cdots<i_{2k+4}=2d+5$. 

$\bullet$ For $n<n_0$ the bottom of $T_{j,*_n}$ is also $\{K_{i_n+1},...,K_{i_{n+1}}\}$ and $i_{n+1}<j$. So
\begin{align*}
    &S_{T_{j,*_n}}(K_1,...,K_{2d+5})=K_{i_n+1}-K_{i_n+2}+...+K_{i_{n+1}}\\
    =&S_{T_{*_n}}(K_1,...,K_j-K_{j+1}+K_{j+2},...,K_{2d+5}).
\end{align*}

$\bullet$ For $n>n_0$, the bottom of $T_{j,*_n}$ is $\{K_{i_n+3},...,K_{i_{n+1}+2}\}$ (the label of all nodes is shifted by 2 comparing to that of their parents). So
\begin{align*}
    &S_{T_{j,*_n}}(K_1,...,K_{2d+5})=K_{i_n+3}-K_{i_n+4}+...+K_{i_{n+1}+2}
    \\
    =&S_{T_{*_n}}(K_1,...,K_j-K_{j+1}+K_{j+2},...,K_{2d+5})
\end{align*}

$\bullet$ For $n=n_0$, the bottom of $T_{j,*_n}$ is $\{K_{i_n+1},...,K_j-K_{j+1}+K_{j+2},...,K_{i_{n+1}+2}\}$ (the label of all nodes before $K_j$ is unchanged while the label of nodes after $K_j$ is shifted by 2, comparing to that of their parents). So
\begin{align*}
    &S_{T_{*_{n_0}}}(K_1,...,K_j-K_{j+1}+K_{j+2},...,K_{2d+5})=K_{i_{n_0}+1}-...+(K_j-K_{j+1}+K_{j+2})+...+K_{i_{n_0+1}}\\
    &=K_{i_{n_0}+1}-...+K_j-K_{j+1}+K_{j+2}+...+K_{i_{n_0+1}}\\
    &=S_{T_{j,*_{n_0}}}(K_1,K_2,...,K_{2d+5})
\end{align*}

This completes the induction step.
\end{proof}

\subsection{A Vanishing Property of Normal Form Coefficients} In this section, we describe a vanishing property of the coefficients of normal form transformation, which states that $H_{K_1...K_{2d+3}}^{d+1}$ vanishes outside of a lower dimensional subvariety of the resonance surface $\mathscr{R}=\{(K_1,...,K_{2d+3}):\mathcal{S}_{K}(K_1,...,K_{2d+3})=\Omega_{K}(K_1,...,K_{2d+3})=0\}$. 

\begin{theorem}\label{MainTheorem1}
For nonlinear Schrodinger equation (\ref{1NLS}), the coefficients $H_{K_1...K_{2d+3}}^{d+1}$ of its normal form vanishes if $(K_1,...,K_{2d+3})\in\mathscr{R}\bigcap\left(\bigcap_{T\in\CG_{d}}\bigcap_{k=0}^{d-1}\CA_{k}(T)\right)$ for any $d\le P$. Here $\CA_{k}(T)$ is defined in Lemma \ref{l:formulaH^d} and
\begin{equation}
    \mathscr{R}=\{(K_1,...,K_{2d+3}):\mathcal{S}_{K}(K_1,...,K_{2d+3})=\Omega_{K}(K_1,...,K_{2d+3})=0\}
\end{equation}
\end{theorem}
\begin{remark}
Because $\CA_{k}(T)$ is the complement of hypersurfaces of the form $S_{T_*}\ne S_{T_{*'}}$ and $\Omega_{2d+1,K}(T)$, above theorem basically asserts that most of $H_{K_1...K_{2d+3}}^{d+1}$ vanish. i.e. They vanish if $(K_1,\cdots,K_{2d+3})$ does not lie on a lower dimensional subvariety of the resonance surface $\mathscr{R}$.
\end{remark}
\begin{remark}
To see this cancellation in the quintic form, recall that
\begin{align*}
\begin{split}
&- i \partial_t c^1_K = \frac{i\epsilon^4}{2\pi L^{4}} \sum_{\substack{\mathcal{S}_{5,K} =0}} d_{K_1} \overline{d_{K_2}} d_{K_3} \overline{d_{K_4}} d_{K_5} e(\int_0^t\widetilde{\Omega}_{3,K}(s)ds)\Big[\frac{1_{K_1\ne K,K_5\ne K, K_2\ne K_3, K_3\ne K_4}}{  \Omega_{3,K}(K_1,K_2-K_3+K_4,K_5)}\\
&-\frac{1_{K_1\ne K_2, K_2\ne K_3, K_4\ne K_5,K_5\ne K}}{  \Omega_{3,K}(K_1-K_2+K_3,K_4,K_5)}-\frac{1_{K_1\ne K_2,K_1\ne K, K_3\ne K_4, K_4\ne K_5}}{  \Omega_{3,K}(K_1,K_2,K_3-K_4+K_5)}\Big]
\\
&+\frac{\epsilon^4}{L^{4}} \sum_{\substack{\mathcal{S}_{3,K} =0\\ \Omega_{3,K}\ne0}} \frac{1}{2\pi \Omega_{3,K}} d_{K_1} \overline{d_{K_2}} d_{K_3} D_{3,K}(t)e(\int_0^t\widetilde{\Omega}_{3,K}(s)ds).
\end{split}
\end{align*}
Apply this theorem, we have for any $(K_1,...,K_5)\in\{\mathcal{S}_{K}(K_1,...,K_5)=\Omega_{K}(K_1,...,K_5)=0\}$, 

\begin{align*}
&\frac{1}{  \Omega_{3,K}(K_1,K_2-K_3+K_4,K_5)}\\
&-\frac{1}{  \Omega_{3,K}(K_1-K_2+K_3,K_4,K_5)}-\frac{1}{  \Omega_{3,K}(K_1,K_2,K_3-K_4+K_5)}=0.
\end{align*}
\end{remark}

\begin{proof} The proof of this theorem is based on induction on the level of trees in the Feynman diagram. Because $\Omega_{2k+3,K}(T)$ depends only on the alternating sums of the bottom of subtrees of $T$ rooting at level $k$, we use the strategy of summing over all trees whose subtrees have the same bottoms. This sum is easier as $k$ is larger. So we use induction on $k$ to reduce a difficult problem to a easier one.

\textbf{Step 1 (Setting up the Notation):} To initiate the induction, let's first introduce some notation. 

\underline{Chain of Refinements:} We denote by $\textbf{P}_{2k+3}$ the set of all $2k+3-$partitions $\mathscr{P}$ of $\{K_1,...,K_{2d+3}\}$ of the following form 
$$\mathscr{P}=\{\{K_1,...,K_{i_1}\},\{K_{i_1+1},...,K_{i_2}\},\cdots,\{K_{i_{2k+2}+1},...,K_{2d+3}\}\}$$ 
for all $0=i_{0}<i_1<\cdots<i_{2k+2}<i_{2k+3}=2d+3$ and $i_{s}\equiv s\ \text{mod}(2)$.

If $\mathscr{P}\in\textbf{P}_{2k+3}$ and $\mathscr{P}'\in\textbf{P}_{2k+5}$ is produced by dividing a set in $\mathscr{P}$ into three new sets, we say $\mathscr{P}'$ is a refinement of $\mathscr{P}$, denoted by $\mathscr{P}\le\mathscr{P}'$. If $\mathscr{P}_{2j+3}\in\mathbf{P}_{2j+3}$ and $\mathscr{P}_{3}\le\mathscr{P}_{5}\le\cdots\le\mathscr{P}_{2d+3}$, then we call $\mathscr{P}_{3},\mathscr{P}_{5},\cdots,\mathscr{P}_{2d+3}$ a full chain of refinements in $\mathbf{P}$. If $\mathscr{P}_{2k+3}\in\mathbf{P}_{2k+3}$ and $\mathscr{P}_{2k+3}\le\mathscr{P}_{2k+5}\le\cdots\le\mathscr{P}_{2d+3}$, then we call $\mathscr{P}_{2k+3},\mathscr{P}_{2k+5},\cdots,\mathscr{P}_{2d+3}$ a chain of refinements starting at $\mathscr{P}_{2k+3}$. In above definitions $\mathscr{P}_{2d+3}=\{\{K_1\},\cdots,\{K_{2d+3}\}\}$.

\underline{Equivalence of Trees and Full Chains of Refinements:} In what follows, we call the bottom of the subtree with root node at $*$ in $T\in\CG_{d}$ the bottom associated with $*$.

For nodes on level $k$ in $T\in\CG_d$, all their associated bottoms of nodes on level $k$ constitute a partition $\mathscr{P}_{2k+3}$ of $\{K_1,...,K_{2d+3}\}$ in $\mathbf{P}_{2k+3}$. And $\{\mathscr{P}_{2k+3}\}_{k=0}^d$ constitute a full chain of refinements in $\mathbf{P}$. Thus we get a canonical full chain of refinement associated with $T$. On the other hand, given a full chain of refinements $\{\mathscr{P}_{2k+3}\}_{k=0}^d$, we may view all sets in $\mathscr{P}_{2k+3}$ as nodes of level $2k+3$. And if $\mathscr{P}_{2k+5}$ is produced by dividing a set in $\mathscr{P}_{2k+3}$ into three new sets, then we view these three new sets as the children of the old set. We also introduce a virtual root node $R$ and let the three sets in $\mathscr{P}_3$ be children of $R$. So we can also produce a canonical tree associated with a full chain of refinements. Thus these two objects are equivalent.

For a chain of refinements starting at $\mathscr{P}_{2k+3}$, we can also construct a tree $T$. But this tree may not look like trees in a Feynman diagram (the root node $R$ should have $2k+3$ children). A node in $T$ is on level $2j+3$ if its corresponding set is in $\mathscr{P}_{2j+3}$. For any $\mathscr{P}\in \textbf{P}_{2k+3}$ we define 
\begin{equation}
\begin{split}
&\CG_{d,k,\mathscr{P}}=\{\textit{full chain of refinements starting at }\mathscr{P}\}
\\
=&\{\textit{trees } T\textit{constructed from chains of refinement starting at }\mathscr{P}\}
\end{split}
\end{equation}
$\Omega_{2k+3,K}(T)$ is a function of the alternating sums of indices on the bottoms associated with nodes on level $k$. If we view trees as full chain of refinement, these alternative sums are alternating sums of sets in the partition in $\mathbf{P}_{2k+3}$.

\underline{Reformulation of (\ref{eq:formulaH^d}):} A important observation is 
\begin{equation}\label{eq:3.14}
    H^{d+1}_{K_1...K_{2d+3}}=\sum_{\mathscr{P}\in\textbf{P}_{3}}\sum_{T\in\CG_{d,1,\mathscr{P}}} \frac{(-1)^{l^T_{0}+\cdots +l^T_{d-1}}}{\Omega_{3,K}(T)\cdots\Omega_{2d+1,K}(T)}(K_1,...,K_{2d+3})
\end{equation}
This follows easily from (\ref{eq:formulaH^d}), if we are willing to view the sum over all trees as sum over all full chains of refinements, because summing over all chains clearly equals to summing over all chains with and fixed start followed by summing over all possible start.

Define sum over all chains of refinements starting at $\mathscr{P}$ as the following
\begin{equation}
    \Sigma_{k,\mathscr{P}}=\sum_{T\in\CG_{d,k,\mathscr{P}}}\frac{(-1)^{l^T_{k}+\cdots +l^T_{d-1}}}{\Omega_{2k+3,K}(T)\cdots\Omega_{2d+1,K}(T)}(K_1,...,K_{2d+3}).
\end{equation}
Since we have identified a chain of refinements with a tree $T$ and defined the notation of level in $T$, $l^T_{j}$ and $\Omega_{2k+3,K}(T)$ can be defined in the same way as in Lemma \ref{l:formulaH^d}.

So 
\begin{equation}
    H^{d+1}_{K_1...K_{2d+3}}=\sum_{\mathscr{P}\in\textbf{P}_{3}} \Sigma_{1,\mathscr{P}}.
\end{equation}

Basically $\Sigma_{k,\mathscr{P}}$ is what we called sum over all graph with fixed associated bottom of nodes of level $k$. 

Define $l_{1}=K_1-K_2$, $l_{2}=K_2-K_3$, ...,$l_{2d+2}=K_{2d+2}-K_{2d+3}$. 

We use an induction on $k$ to show that 

\begin{equation}\label{formula:induction}
    \Sigma_{k,\mathscr{P}}=\frac{1}{2^{d-k}l_1\cdots l_{i_1-1}\hat{l}_{i_1}l_{i_1+1}\cdots \hat{l}_{i_2}\cdots \hat{l}_{i_{2k+2}}\cdots l_{2d+2}}
\end{equation}

Here $\hat{l}$ indicates the absence of this term. 

We will first show that above conclusion holds when $k=d-1$, then assume the conclusion for $k+1$ and get the conclusion for $k$ 

\textbf{Step 2 (Verifying the Induction for $k=d-1$):} For $\mathscr{P}\in\textbf{P}_{2d+1}$, $\mathscr{P}$ takes the form $$\{\{K_1\},\{K_{2}\},\cdots, \{K_{i},K_{i+1},K_{i+2}\},\cdots,\{K_{2d+3}\}\}.$$

By definition there are only one 
element in $\CG_{d,d-1,\mathscr{P}}$. Thus 
\begin{equation}
    \Sigma_{d-1,\mathscr{P}}=\frac{(-1)^{l^T_{d-1}}}{\Omega_{2d+1,K}(T)(K_1,...,K_{2d+3})}.
\end{equation}

By definition of $l_{j}^T$ in Lemma \ref{l:formulaH^d}, nodes with label $K_{l^T_{d-1}}$ are the only nodes that have children. And its three children are $K_{l^T_{d-1}},K_{l^T_{d-1}+1},K_{l^T_{d-1}+2}$. Thus $l^T_{d-1}=i$.

By definition of $\Omega_{2d+1,K}(T)$

\begin{equation}
\begin{split}
    \Omega_{2d+1,K}(T)(K_1,...,K_{2d+3})&=\Omega_{2d+1,K}(K_1,K_{2},\cdots, K_{i}-K_{i+1}+K_{i+2},\cdots,K_{2d+3})\\
    &=K_1^2-K_2^2+...+(-1)^{i-1}(K_{i}-K_{i+1}+K_{i+2})^2+...+K_{2d+3}^2-K^2\\
    &=(-1)^{i}(K_{i}^2-K_{i+1}^2+K_{i+2}^2-(K_{i}-K_{i+1}+K_{i+2})^2)\\
    &=2(-1)^{i}l_{i}l_{i+1}
\end{split}
\end{equation}

Here in the second equality, we apply the condition $\Omega_{K}(K_1,...,K_{2d+3})=0$. 

Thus

\begin{equation}
    \Sigma_{d-1,\mathscr{P}}=\frac{(-1)^{i}}{2(-1)^{i}l_{i}l_{i+1}}=\frac{1}{2l_{i}l_{i+1}}.
\end{equation}

This is exactly the equation (\ref{formula:induction}) for $k=d-1$.

\textbf{Step 3 (Applying the Induction Hypothesis):} Assume that in $T\in\CG_{d,k,\mathscr{P}}$  is a chain of refinements $\{\mathscr{P},\mathscr{P}_{2k+5},\cdots,\mathscr{P}_{2d+3}\}$ starting at $\mathscr{P}$ with  
$$\mathscr{P}=\{\{K_1,...,K_{i_1}\},\{K_{i_1+1},...,K_{i_2}\},\cdots,\{K_{i_{2k+2}+1},...,K_{2d+3}\}\}.$$ 
Then all possibility of $\mathscr{P}_{2k+5}$ are listed below (depending on which set is divided into three or in other word, which nodes on level $k$ have children)

     (1) For $1\le i<j\le i_1$, $i$ odd, $j$ even,
     \begin{align*}
     \mathscr{P}_{1,ij}&=\{\{K_1,...,K_{i}\},\{K_{i+1},...,K_{j}\},\{K_{j+1},...,K_{i_1}\},\{K_{i_1+1},...,K_{i_2}\},\\
     &\cdots,\{K_{i_{2k+2}+1},...,K_{2d+3}\}\}.
     \end{align*}

     (2) For $i_{1}+1\le i<j\le i_2$, $i$ even, $j$ odd,
     \begin{align*}
     \mathscr{P}_{2,ij}&=\{\{K_{1},...,K_{i_1}\},\{K_{i_1+1},...,K_{i}\},\{K_{i+1},...,K_{j}\},\{K_{j+1},...,K_{i_2}\},\\
     &\cdots,\{K_{i_{2k+2}+1},...,K_{2d+3}\}\}.
     \end{align*}
     $$\cdots$$
     
     (2k+3) For $i_{2k+2}+1\le i<j\le 2k+3$, $i$ odd, $j$ even,
     \begin{align*}
     \mathscr{P}_{2k+3,ij}&=\{\{K_1,...,K_{i_1}\},\{K_{i_1+1},...,K_{i_2}\},\\
     &\cdots,\{K_{i_{2k+2}+1},...,K_{i}\},\{K_{i+1},...,K_{j}\},\{K_{j+1},...,K_{2k+3}\}\}.
     \end{align*}

Denote by $\textbf{Ref}_{\mathscr{P}}$ the set of all $\mathscr{P}_{s,ij}$. If $\mathscr{P}_{2k+5}=\mathscr{P}_{s,ij}$, then we know $K_s$ on level $k$ of $T$ has children. Thus $l^T_{k}=s$

We can split the sum in the following way, (summing over all chains start at $\mathscr{P}$ equals to summing over all chains with second step $\mathscr{P}_{2k+5}$ followed by summing over all possibility of $\mathscr{P}_{2k+5}$)

\begin{equation}
    \sum_{T\in\CG_{d,k,\mathscr{P}}}=\sum_{\mathscr{P}'\in\textbf{Ref}_{\mathscr{P}}}\sum_{T\in\CG_{d,k+1,\mathscr{P}'}}.
\end{equation}

Now notice that $\Omega_{2k+3,K}(T)$ depends only on the alternative sum of the sets in $\mathscr{P}_{2k+5}$ which implies that $\Omega_{2k+3,K}(T)$ is independent of the inner sum. So we have 

\begin{equation}
\begin{split}
    \Sigma_{k,\mathscr{P}}&=\sum_{T\in\CG_{d,k,\mathscr{P}}}\frac{(-1)^{l^T_{k}+\cdots +l^T_{d-1}}}{\Omega_{2k+3,K}(T)\cdots\Omega_{2d+1,K}(T)}(K_1,...,K_{2d+3})\\
    &=\sum_{\mathscr{P}'\in\textbf{Ref}_{\mathscr{P}}}\frac{(-1)^{l^{T}_{k}}}{\Omega_{2k+3,K}(T)}\sum_{T\in\CG_{d,k+1,\mathscr{P}'}}\frac{(-1)^{l^T_{k+1}+\cdots +l^T_{d-1}}}{\Omega_{2k+5,K}(T)\cdots\Omega_{2d+1,K}(T)}(K_1,...,K_{2d+3})
\end{split}
\end{equation}

Since $\textbf{Ref}_{\mathscr{P}}$ equals to the set of all $\{\mathscr{P}_{s,ij}\}$, we get 

\begin{equation}
\begin{split}
    \Sigma_{k,\mathscr{P}}&=\sum_{T\in\CG_{d,k,\mathscr{P}}}\frac{(-1)^{l^T_{k}+\cdots +l^T_{d-1}}}{\Omega_{2k+3,K}(T)\cdots\Omega_{2d+1,K}(T)}(K_1,...,K_{2d+3})
    \\
    &=\frac{1}{\Omega_{2k+3,K}(T)}\Bigg(\sum_{\substack{1\le s\le 2k+3\\\text{$s$ odd}}}(-1)^{s}\sum_{\substack{i_{s-1}+1\le i<j\le i_{s}\\\text{$i$ odd $j$ even}}}\sum_{T\in\CG_{d,k+1,\mathscr{P}_{2k+3,ij}}}\\
    &\frac{(-1)^{l^T_{k+1}+\cdots +l^T_{d-1}}}{\Omega_{2k+5,K}(T)\cdots\Omega_{2d+1,K}(T)}(K_1,...,K_{2d+3})
    \\
    &+\sum_{\substack{1\le s\le 2k+3\\\text{$s$ even}}}(-1)^{s}\sum_{\substack{i_{s-1}+1\le i<j\le i_{s}\\\text{$i$ even $j$ odd}}}\sum_{T\in\CG_{d,k+1,\mathscr{P}_{2k+3,ij}}}\frac{(-1)^{2^{d-k-1}l^T_{k+1}+\cdots +l^T_{d-1}}}{\Omega_{2k+5,K}(T)\cdots\Omega_{2d+1,K}(T)}(K_1,...,K_{2d+3})\Bigg).
    \end{split}
\end{equation}
Apply the induction hypothesis (\ref{formula:induction}),

\begin{equation}\label{eq:lastone}
\begin{split}
    \Sigma_{k,\mathscr{P}} &=\frac{1}{\Omega_{2k+3,K}(T)}\Bigg(\sum_{\substack{1\le s\le 2k+3\\\text{$s$ odd}}}\sum_{\substack{i_{s-1}+1\le i<j\le i_{s}\\\text{$i$ odd $j$ even}}} \frac{1}{2^{d-k-1}l_1\cdots \hat{l}_{i_1}\cdots \hat{l}_{i_{s-1}}\cdots \hat{l}_{i}\cdots \hat{l}_{j}\cdots \hat{l}_{i_{s}}\cdots l_{2d+2}}\\
    &-\sum_{\substack{1\le s\le 2k+3\\\text{$s$ odd}}}\sum_{\substack{i_{s-1}+1\le i<j\le i_{s}\\\text{$i$ even $j$ odd}}} \frac{1}{2^{d-k-1}l_1\cdots \hat{l}_{i_1}\cdots \hat{l}_{i_{s-1}}\cdots \hat{l}_{i}\cdots \hat{l}_{j}\cdots \hat{l}_{i_{s}}\cdots l_{2d+2}}\Bigg)
    \\
    &=\frac{1}{\Omega_{2k+3,K}(T)}\left(\sum_{\substack{1\le s\le 2k+3\\\text{$s$ odd}}}\sum_{\substack{i_{s-1}+1\le i<j\le i_{s}\\\text{$i$ odd $j$ even}}} \frac{l_i l_j}{2^{d-k-1}l_1\cdots \hat{l}_{i_1}\cdots l_{2d+2}}-\sum_{\substack{1\le s\le 2k+3\\\text{$s$ even}}}\sum_{\substack{i_{s-1}+1\le i<j\le i_{s}\\\text{$i$ even $j$ odd}}}\cdots\right)
    \\
    &=\frac{1}{2^{d-k-1}\Omega_{2k+3,K}(T)l_1\cdots \hat{l}_{i_1}\cdots l_{2d+2}}\left(\sum_{\substack{1\le s\le 2k+3\\\text{$s$ odd}}}\sum_{\substack{i_{s-1}+1\le i<j\le i_{s}\\\text{$i$ odd $j$ even}}}l_i l_j-\sum_{\substack{1\le s\le 2k+3\\\text{$s$ even}}}\sum_{\substack{i_{s-1}+1\le i<j\le i_{s}\\\text{$i$ even $j$ odd}}}l_i l_j\right)
\end{split}
\end{equation}

By $\Omega_{K}(K_1,...,K_{2d+3})=0$, we have 

\begin{equation}\label{eq:lastone'}
\begin{split}
    \Omega_{2k+3,K}(T)&=(K_1-...+K_{i_1})^2-(K_{i_1+1}-...+K_{i_2})^2+\cdots+(K_{i_{2k+2}}-...+K_{2d+3})^2-K^2\\
    &=\big((K_1-...+K_{i_1})^2-K_1^2+...-K_{i_1})^2\big)-\big((K_{i_1+1}-...+K_{i_2})^2-K_{i_1+1}^2+...-K_{i_2}\big)\\
    &+\cdots+\big((K_{i_{2k+2}}-...+K_{2k+3})^2-K_{i_{2k+2}}^2+...-K_{2d+3}^2\big)\\
    &=2\sum_{\substack{1\le i<j\le i_1\\\text{$i$ odd, $j$ even}}}l_i l_j-
    2\sum_{\substack{i_1+1\le i<j\le i_2\\\text{$i$ even, $j$ odd}}}l_i l_j+
    2\sum_{\substack{i_2+1\le i<j\le i_3\\\text{$i$ odd, $j$ even}}}l_i l_j-\cdots+
    2\sum_{\substack{i_{2k+2}+1\le i<j\le 2d+3\\\text{$i$ even, $j$ odd}}} l_i l_j\\
    &=2\sum_{\substack{1\le s\le 2k+3\\\text{$s$ odd}}}\sum_{\substack{i_{s-1}+1\le i<j\le i_{s}\\\text{$i$ odd $j$ even}}}l_i l_j-2\sum_{\substack{1\le s\le 2k+3\\\text{$s$ even}}}\sum_{\substack{i_{s-1}+1\le i<j\le i_{s}\\\text{$i$ even $j$ odd}}}l_i l_j
\end{split}
\end{equation}

Here in the second equality, we applied the condition $\Omega_{K}(K_1,...,K_{2d+3})=0$.

Substituting (\ref{eq:lastone'}) in (\ref{eq:lastone}), we get 
\begin{equation}
    \Sigma_{k,\mathscr{P}}=\frac{1}{2^{d-k}l_1\cdots \hat{l}_{i_1}\cdots \hat{l}_{i_2}\cdots \hat{l}_{i_{2k+2}}\cdots l_{2d+2}}
\end{equation}

So we finish the proof of (\ref{formula:induction}).

\textbf{Step 4 (Concluding the Proof):} Now we take $k=1$ in (\ref{formula:induction}). There are three nodes on level 0. So their associated bottoms are $\mathscr{P}_{i_1 i_2}=\{\{K_1,...,K_{i_1}\},\{K_{i_1+1},...,K_{i_2}\},\{K_{i_2+1},...,K_{2d+3}\}\}$. So (\ref{formula:induction}) reads

\begin{equation}
    \Sigma_{1,\mathscr{P}}=\frac{1}{2^dl_1\cdots \hat{l}_{i_1}\cdots \hat{l}_{i_2}\cdots l_{2d+2}}
\end{equation}

By (\ref{eq:3.14})

\begin{equation}
\begin{split}
     H^{d+1}_{K_1...K_{2d+3}}&=\sum_{\mathscr{P}\in\textbf{P}_{3}} \Sigma_{1,\mathscr{P}}\\
    &=\sum_{\substack{1\le i_1<i_2\le 2d+3\\\text{$i_1$ odd, $i_2$ even}}} \frac{1}{2^{d}l_1\cdots \hat{l}_{i_1}\cdots \hat{l}_{i_2}\cdots l_{2d+2}}\\
    &=\frac{1}{2^{d}l_1 l_2\cdots l_{2d+1} l_{2d+2}}\sum_{\substack{1\le i_1<i_2\le 2d+3\\\text{$i_1$ odd, $i_2$ even}}}l_{i_1}l_{i_j}\\
    &=0
\end{split}
\end{equation}

The last line follows from 

\begin{equation}
    2\sum_{\substack{i<j\\\text{$i$ odd, $j$ even}}}l_{i}l_{j}=K_{1}^2-K_{2}^2+\cdots+K_{2d+3}^2-(K_{1}-K_{2}+\cdots+K_{2d+3})^2=0.  
\end{equation}
This completes the proof.
\end{proof}

\section{The Long Time Dynamics of 1D Cubic NLS}\label{sec:dynamics}

This section is devoted to the proof of Theorem \ref{th:main}. As an application of the general tools we developed in section \ref{sec:Feynman}, we shall prove a long time dynamics result of (\ref{1NLS}). 

\subsection{Ideas of the Proof}\label{sec:ideaproof} The intuition behind this result is from the analysis of the effect of normal form transformation. As explained in subsection \ref{sec:nftcubic}, normal form transformation has a similar effect as integration by parts in the context of oscillatory integral. 
To quantify the effect of normal form transformation, notice that each time we gain either $L^{-1}$ or $\epsilon^2$ factor when we do normal form. In following explanation, we often abuse $L^+$ and $L^0$, although sometimes their difference is crucial. 
\smallskip
\newline
\underline{Ideas from number theory:} Consider a sum over lattice $\mathbb{Z}^k_{L}$,  
\begin{equation}\label{eq:latticesum}
\sum_{f_1=0,\cdots, f_k=0} a(K_1)\cdots a(K_\ell).
\end{equation}
Let $f_1,\cdots,f_k$ be linear functions in $K_1,\cdots,K_\ell$ and $a(K_1),\cdots, a(K_\ell)$ be rapidly decaying functions, then this sum is of order $L^{\ell-k}$. Because $a_{K_j}$ is rapidly decaying, we think it as a indicator function on interval $1_{[-1,1]}(K_j)$. 
Let $K_j=\frac{k_j}{L}$ for $j=1,\cdots,\ell$, then (\ref{eq:latticesum}) is equal to number of integer solutions of 
\begin{equation}\label{eq:latticesum2}
\{f_1(k_1,\cdots,k_l)=0,\cdots, f_k(k_1,\cdots,k_l)=0, \ \ |k_j|\le L\}.
\end{equation}
Since $f_j$ are linear functions, this is the number of integer points on a $l-k$ linear subspace. For $\ell$ variables with $k$ linear restrictions, each restriction reduces one freedom. We have $l-k$ freedoms and each freedom takes $L$ possible values then there should be in total $L^{l-k}$ possibilities.

Consider $f_1,\cdots,f_k$ being polynomial functions with $2\sum_{j=1}^k deg f_j < l$, number theorists believe that the number of integer points of (\ref{eq:latticesum2}) is $L^{l-\sum_{j=1}^k deg f_j}$. This principle has been verified in may concrete situations, while itself is a major conjecture in number theory, see for example \cite{BR} and \cite{HB}. Intuitively, since $|k_j|\le L$, $|f|\le L^{deg f}$, which means $f$ has $L^{deg f}$ possible values but when we set $f=0$, it reduces a freedom that can take $L^{deg f}$ values, hence it reduces $deg f$ freedoms.

The assumption $2\sum_{j=1}^k deg f_j < l$ is crucial to obtain the equidistribution result. For example consider
\begin{equation}\label{eq:latticecubic}
k_1^3+k_2^3=k_3^3, \ \ \ |k_j|\le L.
\end{equation}
We have a set of trivial solution $k_1=-k_2$ and $k_3=0$, hence the number of solutions is at least $L$, whereas the equidistribution result only gives a constant number of solutions.
However, the number theorists believe that, after remove a lower dimensional subvariety of (\ref{eq:latticesum2}), the equidistribution result should still hold and solutions on this lower dimensional subvariety are the trivial solutions. In above example, the subvariety should be $\{k_3=0\}$. If $k_3\ne 0$, by the special case of the Fermat's last theorem, (\ref{eq:latticecubic}) has no solution. Therefore, we know the number of solutions can be bounded by any positive constant. Typically the number of nontrivial solutions is $L^{l-\sum_{j=1}^k deg f_j}$ and the number of trivial solutions is $L^{\frac{\ell}{2}}$. The total number of solutions is $L^{l-\sum_{j=1}^k deg f_j}+L^{\frac{\ell}{2}}$. When $2\sum_{j=1}^k deg f_j< \ell$, the equidistributed solutions dominates, and this is where the condition, $2\sum_{j=1}^k deg f_j < \ell$, comes from.
\smallskip
\newline
\underline{The time scale:} In \cite{FGH} and \cite{BGHS1}, they considered the NLS equation with dimension $d\ge 2$ where for fixed $K\in\mathbb{Z}^d_L$, the resonance set
\begin{equation}
\left\{ \begin{array}{l}
K_1-K_2+K_3=K, \\
|K_1|^2-|K_2|^2+|K_3|^2=|K|^2,\\
K_1,K_2,K_3\in\mathbb{Z}^d_L.
\end{array} \right.
\end{equation}
has $3d$ variables. In their case the number of trivial solutions can be bounded by $L^{d}$, and the number of nontrivial solutions is of order $L^{2d-2}$ ($2d-2=3d-d-2$, $d$ linear equations, $1$ quadratic equation). When $d\ge 2$, we have $L^{2d-2}\ge L^d$, i.e. the number of nontrivial solutions dominates. Consider $d=2$ with cubic nonlinearity, then the cubic term is 
$$\frac{\epsilon^2}{L^{2d}}\sum_{\substack{\mathcal{S}_{3,K}=0}} a_{K_1}a_{K_2}a_{K_3} e(\Omega_{3,K}t).$$
Separate this into resonance and nonresonance interaction and remove the nonresonance interaction to higher order term by doing normal form. We have the cubic resonance interaction
$$\frac{\epsilon^2}{L^{2d}}\sum_{\substack{\mathcal{S}_{3,K}=0\\\Omega_{3,K}=0}} a_{K_1}a_{K_2}a_{K_3} e(\Omega_{3,K}t).$$ 
This sum is of order $L^{2d-2}$ from the number theory analysis, hence the cubic resonance is of order $\frac{\epsilon^2}{L^{2}}$. For higher order terms, i.e. the ones from substituting the equation of $\partial_t d_K$, we gain $\epsilon^2$ from the cubic nonlinearity. Thus after doing normal form transformation once, one gains either $L^{-2}$ or $\epsilon^2$. Combining this with the assumption $\epsilon^2 L^+\ll 1$, after doing normal form $P$ times, 
$$-i\partial_t a_{K}=\frac{\epsilon^2}{L^2}\left(\frac{1}{L^{2d-2}}\sum_{\substack{\mathcal{S}_{3,K}=0\\\Omega_{3,K}=0}} a_{K_1}a_{K_2}a_{K_3}\right) +O(\frac{\epsilon^4}{L^2})+\cdots+O(\frac{\epsilon^{2P}}{L^2})+\epsilon^{2P+2}.$$
Consider the time of existence up to $\frac{L^2}{\epsilon^2}$, the contribution of $O(\frac{\epsilon^4}{L^2})$ can be bounded by $\int_{0}^t O(\frac{\epsilon^4}{L^2})\le \epsilon^2$. Therefore, when $t\le \frac{L^2}{\epsilon^2}$, the dynamics of $a_K$ can be approximately described by (rescale $t$ by $s=\frac{\epsilon^2}{L^2}t$)
$$-i\partial_s a_{K}=\frac{1}{L^{2d-2}}\sum_{\substack{\mathcal{S}_{3,K}=0\\\Omega_{3,K}=0}} a_{K_1}a_{K_2}a_{K_3},$$
which in fact, after applying the circle method, converges to a integral, see Theorem 5 in \cite{BGHS1}.

In our case, the cubic resonance interaction is
\begin{equation}
\left\{ \begin{array}{l}
K_1-K_2+K_3=K, \\
K_1^2-K_2^2+K_3^2=K^2,\\
K_1,K_2,K_3\in\mathbb{Z}_L.
\end{array} \right.
\end{equation}
 Due to lack of number of variables, the number of trivial solutions dominates, which is of order $L$. Combining with the factor $\frac{\epsilon^2}{L^2}$, our cubic resonance interaction is of order $\frac{\epsilon^2}{L}$, which means the time which we can describe the solution is only up to $\frac{L}{\epsilon^2}$. To go beyond this time scale, we change the phase of the solution by $d_K=a_Ke(\cdots)$ as in subsection \ref{sec:changephase} to remove the cubic resonance interaction and transform the cubic nonlinearity into
$$\frac{\epsilon^2}{L^2}\sum_{\substack{\mathcal{S}_{3,K}=0\\\Omega_{3,K}\ne 0}} d_{K_1}d_{K_2}d_{K_3} e(\int^t_{0} \widetilde{\Omega}_{3,K}),$$
which can be transformed into quintic form by normal form transformation. And the quintic resonance interaction will be our main term that describes the long time dynamics up to time $t<\frac{L^2}{\epsilon^4}$, which is the wave turbulence time scale.
\smallskip
\newline
\underline{Outline of the Proof:} First, we shall further analyze the main term $H^2_K(u)$. In \cite{FGH} and \cite{BGHS1}, they approximate their main term using the circle method. We may also follow their strategy but because of the integrable structure, we have the cancellation property (Theorem \ref{MainTheorem1}) that we proved in section \ref{sec:Feynman}, which states that outside of a submanifold of the resonance surface, the normal form coefficients vanish. Apply this cancellation to $H^2_K(u)$, $H^2_K(u)=0$ if $(K_1,\cdots,K_5)\in A_1\bigcap A_2\bigcap A_3$.  
And this will significantly simplify the approximation of the main term. The first following section is devoted to the simplification of the main term $H^2_K(u)$. This step is a refinement of Theorem \ref{MainTheorem1} for quintic term. Notice that the vanishing property reduces the number of nontrivial solutions for which the normal form coefficient is non-vanishing. So it allows us to get a gain of $L^{-3}$ instead of $L^{-2}$ in $H^d_K(u)$, $d\ge3$. But in our main term, the $H^2_K(u)$, although the number of nontrivial solutions with non-vanishing coefficient is of order $L^{2+}/L$, there are $L^2$ trivial solutions which prevent us from upgrading our time scale to $\frac{L^3}{\epsilon^4}$. 

The next step is to treat all other terms as errors. We apply the explicit formula of the normal form coefficients in Lemma \ref{l:formulaH^d} to reduce the problem into a number theoretical problem (\ref{eq:mainnumbertheoryineq}). Because in (\ref{eq:mainnumbertheoryineq}) there are $1$ linear equation and $d$ quadratic equations, number of nontrivial solutions are expected to be $L^{2d+1-1-2\times d}=L^+$. Therefore, we should be more careful about the potential possibility of $L^{d}$ trivial solutions ($K_1=K_2$, $K_3=K_4$, $\cdots$). We have to keep track inequality constrains in the explicit formulas to eliminate this possibility. A useful tool of getting bound and taking advantage of these constrains is the Lemma \ref{l:numbertheory2}.

Combining all above estimates, we are able to conclude Proposition \ref{boundphase} by boostrap argument, which describes the dynamics of $d_K$ that is $u_K$ with a phase correction $\int_0^t|b_K(s)|^2ds$. To obtain the dynamics of $u_K$, we need a refined estimate of $|d_K(s)|^2$ so that we may replace it by some known quantities. We will do energy estimate on $c_{K}$ which give us an estimate of $|c_K(s)|^2$ and then show that $|c_K|$ differs from $|d_{K}|$ by some lower order error. Finally, we obtain the desired estimate on $\int_0^t|b_K(s)|^2ds$ in terms of $\int_0^t|c_K(s)|^2ds$.  

\subsection{Identification of the Main Term} In this section, we shall calculate the coefficients of $H^2_{K}$ carefully. In following derivation of the dynamics of $d_K$. $H^2_{K}$ will be thought as the main term. All other terms will be errors.   

Recall that 
\begin{align*}
\begin{split}
\frac{\epsilon^{4}}{L^{4}}H_K^2(u)&=\frac{\epsilon^4}{2\pi L^{4}} \sum_{\substack{\mathcal{S}_{5,K} =0\\ \Omega_{5,K}=0}}\Big[\frac{1_{A_2}}{ \Omega_{3,K}(K_1,K_2-K_3+K_4,K_5)}-\frac{1_{A_1}}{  \Omega_{3,K}(K_1-K_2+K_3,K_4,K_5)}\\
&-\frac{1_{A_3}}{  \Omega_{3,K}(K_1,K_2,K_3-K_4+K_5)}\Big] d_{K_1} \overline{d_{K_2}} d_{K_3} \overline{d_{K_4}} d_{K_5} e(\int_0^t\widetilde{\Omega}_{5,K}(s)ds),
\end{split}
\end{align*}
where
$$A_1=\{K_1\ne K_2,K_2\ne K_3,K_4\ne K_5,K_5\ne K\},$$
$$A_2=\{K_1\ne K,K_2\ne K_3,K_3\ne K_4,K_5\ne K\},$$
$$A_3=\{K_1\ne K_2,K_1\ne K,K_3\ne K_4,K_4\ne K_5\}.$$

For simplicity, let $H=H^2_{K_1...K_5}$ and we will fix $K_1$, $K_2$, ..., $K_5$ in most part of this section. We know that

$$H=\frac{1_{l_1\ne0,l_2\ne0,l_1+l_3\ne0,l_4\ne0}}{l_1l_2}+\frac{1_{l_1+l_3\ne0,l_2+l_4\ne0,l_2\ne0,l_3\ne0}}{l_2l_3}+\frac{1_{l_1\ne0,l_2+l_4\ne0,l_3\ne0,l_4\ne0}}{l_3l_4},$$
if we use the change of variables that we introduced in the earlier section 

\begin{equation}
\left\{ \begin{array}{l}
l_1=K_1-K_2\\
l_2=K_2-K_3\\
l_3=K_3-K_4\\
l_4=K_4-K_5
\end{array} \right.
\end{equation}

The resonance surface equation $\Omega_{5,K}=0$ becomes $l_1l_2+l_3l_4+l_1l_4=0$.

By the principle of inclusion and exclusion, we know that 

\begin{equation}
    H=H1_{l_1\ne0,...,l_4\ne0}+\sum_{i=1}^{4}H1_{l_i=0}-\sum_{i,j=1}^{4}H1_{l_i=l_j=0}+\cdots
\end{equation}

From Theorem \ref{MainTheorem1}, we know that the following rational function 
\begin{align*}
\begin{split}
\frac{1}{\Omega_{3,K}(K_1,K_2-K_3+K_4,K_5)}-\frac{1}{\Omega_{3,K}(K_1-K_2+K_3,K_4,K_5)}
&-\frac{1}{\Omega_{3,K}(K_1,K_2,K_3-K_4+K_5)}
\end{split}
\end{align*}
vanishes on the resonance surface $\mathcal{S}_{5,K} =0$, $\Omega_{5,K}=0$, which implies $H1_{l_1\ne0,...,l_4\ne0}=0$

All other terms also vanish, except

(1) $1_{l_1=0}H=\frac{1_{l_2\ne0,l_3\ne0}}{l_2l_3}$. In this case by $l_1=0$, $l_1l_2+l_3l_4+l_1l_4=0$ is equivalent to $l_3l_4=0$. Then by $l_3\ne0$, we have $l_4=0$. Thus in this case $1_{l_1=0}H=1_{l_1=0,l_4=0}H$.

(2) $1_{l_2=0}H=\frac{1_{l_1\ne0,l_3\ne0,l_4\ne0}}{l_3l_4}$. For similar reason as in (1), $1_{l_2=0}H=1_{l_2=0,l_1+l_3=0}H$.

(3) $1_{l_3=0}H=\frac{1_{l_1\ne0,l_2\ne0,l_4\ne0}}{l_1l_2}$. For similar reason as in (1), $1_{l_3=0}H=1_{l_3=0,l_2+l_4=0}H$.

(4) $1_{l_4=0}H=\frac{1_{l_2\ne0,l_3\ne0}}{l_2l_3}$. For the same reason as (1), $1_{l_4=0}H=1_{l_1=0,l_4=0}H$.

(5) $1_{l_1=0,l_4=0}H=\frac{1_{l_2\ne0,l_3\ne0}}{l_2l_3}$.

From (1), (4), (5), we have 
\begin{align*}
&\sum_{\substack{\mathcal{S}_{5,K} =0\\ \Omega_{5,K}=0}}(H^2_{K_1...K_5}1_{l_1=0}+H^2_{K_1...K_5}1_{l_4=0}-H^2_{K_1...K_5}1_{l_1=0,l_4=0}) d_{K_1} \overline{d_{K_2}} d_{K_3} \overline{d_{K_4}} d_{K_5} e(\int_0^t\widetilde{\Omega}_{5,K}(s)ds)
\\
=&(2-1)\sum_{\substack{K_1\ne K\\ K_2\ne K}} \frac{|d_{K_1}|^2}{(K_1-K)}\frac{|d_{K_2}|^2}{(K-K_2)}d_K
=-\Big(\sum_{\substack{K_1\ne K}} \frac{|d_{K_1}|^2}{(K_1-K)}\Big)^2d_K.
\end{align*}

From (2), (3), we have

\begin{align*}
&\sum_{\substack{\mathcal{S}_{5,K} =0\\ \Omega_{5,K}=0}}(H^2_{K_1...K_5}1_{l_2=0}+H^2_{K_1...K_5}1_{l_3=0}) d_{K_1} \overline{d_{K_2}} d_{K_3} \overline{d_{K_4}} d_{K_5} e(\int_0^t\widetilde{\Omega}_{5,K}(s)ds)
\\
=&2\sum_{\substack{K_1\ne K\\ K_2\ne K_1}} \frac{|d_{K_1}|^2}{(K_2-K_1)}\frac{|d_{K_2}|^2}{(K_1-K)}d_K
\\
=&-2\sum_{\substack{K_1\ne K}} \frac{|d_{K_1}|^4}{(K_1-K)^2}d_K+\sum_{\substack{K_1\ne K\\ K_2\ne K_1}} \frac{|d_{K_1}|^2|d_{K_2}|^2}{(K_2-K_1)}\Big(\frac{1}{K_1-K}-\frac{1}{K_2-K}\Big)d_K
\\
=&-3\sum_{\substack{K_1\ne K}} \frac{|d_{K_1}|^4}{(K_1-K)^2}d_K+\Big(\sum_{\substack{K_1\ne K}} \frac{|d_{K_1}|^2}{(K_1-K)}\Big)^2d_K.
\end{align*}

Therefore, 
\begin{align}\label{eq:mainterm}
\frac{\epsilon^{4}}{L^{4}}H_K^2(u)=\frac{3\epsilon^4}{L^4}\sum_{\substack{K_1\ne K}} \frac{|d_{K_1}|^4}{(K_1-K)^2}d_K.
\end{align}

\subsection{Estimate of Errors in Normal Form Transformation}
In this section, we shall obtain estimates of $H^{d}_K$, $d\ge 2$, $F^{d}_K$, $d\ge 1$,  $\widetilde{H}^{P+1}_K$,  $G^{d}_K$, $d\ge 3$, $E^{d}_K$, $d\ge 2$, $\widetilde{G}^{P+1}_K$, $S^{d}_K$ $d\ge 3$. Note that after introducing $d_K$, the equations (\ref{eq:profiled}) of $d_K$ contains two inequalities. All terms produced by normal form transformations inherits inequality constrains from this, which prevent the presence of trivial solutions $K_1=K_2$, $K_2=K_3$, ..., in Lemma \ref{l:numbertheory3}.

\begin{lemma}\label{l:numbertheory1}
Fix $d,P\in\mathbb{N}$, $T\in\CG_{d-1}$ or $\CG_{P}$, $\ell>1$. Let $\{d_{K}\}_{K\in\mathbb{Z}_L}\in X^{\ell}$ i.e. $||d_{K}||_{X^\ell}=\sup_{K\in\mathbb{Z}_L}|\langle K\rangle^\ell d_{K}|<+\infty$. Then for any $K\in\mathbb{Z}_L$, we have

$$\langle K\rangle^\ell\sum_{\substack{\mathcal{S}_{2d+1,K}=0\\\Omega_{2d+1,K}=0}}\frac{\prod^{d-2}_{k=0}1_{\CA_{k}(T)}}{|\Omega_{3,K}(T)\Omega_{5,K}(T)\cdots\Omega_{2d-1,K}(T)|}(K_1,...,K_{2d+1})d_{K_1}\cdots d_{K_{2d+1}}\lesssim L^{2(d-1)+}||d_{K}||_{X^\ell}^{2d+1},$$

$$\langle K\rangle^\ell\sum_{\substack{\mathcal{S}_{2d+1,K}=0\\\Omega_{2d+1,K}\ne0}}\frac{\prod^{d-2}_{k=0}1_{\CA_{k}(T)}}{|\Omega_{3,K}(T)\Omega_{5,K}(T)\cdots\Omega_{2d+1,K}(T)|}(K_1,...,K_{2d+1})d_{K_1}\cdots d_{K_{2d+1}}\lesssim L^{2d+}||d_{K}||_{X^\ell}^{2d+1},$$
and

$$\langle K\rangle^\ell\sum_{\substack{\mathcal{S}_{2P+3,K} =0}}\frac{\prod^{P-1}_{k=0}1_{\CA_{k}(T)}}{|\Omega_{3,K}(T)\Omega_{5,K}(T)\cdots\Omega_{2P+1,K}(T)|}(K_1,...,K_{2P+3})d_{K_1}...d_{K_{2P+3}}\le L^{2(P+1)+}||d_{K}||_{X^\ell}^{2P+1}.$$
\end{lemma}
\begin{proof} We only show how to prove the first inequality. The second and the third ones can be handled similarly. 
It suffices to show that
\begin{equation}\label{eq:lemma2mainreduction}
    \sum_{\substack{\mathcal{S}_{2d+1,K}=0\\\Omega_{2d+1,K}=0}}\frac{\prod^{d-2}_{k=0}1_{\CA_{k}(T)}}{|\Omega_{3,K}(T)\Omega_{5,K}(T)\cdots\Omega_{2d-1,K}(T)|}\langle K_1\rangle^{-\ell}\cdots \langle K_{2d+1}\rangle^{-\ell}\lesssim L^{2d-2+}\langle K\rangle^{-\ell}.
\end{equation}

Since $K_1,...,K_{2d+1}\in\mathbb{Z}_L$, $\Omega_{2k+1,K}(T)=\frac{\mu_{k}}{L^2}$, $\mu_{k}\in\mathbb{Z}$, $k=1,\cdots,d-1$. Split the summation into two parts,

\begin{equation}\label{eq:lemma2main}
\begin{split}
     &\sum_{\substack{\mathcal{S}_{2d+1,K}=0\\\Omega_{2d+1,K}=0}}\frac{\prod^{d-2}_{k=0}1_{\CA_{k}(T)}}{|\Omega_{3,K}(T)\Omega_{5,K}(T)\cdots\Omega_{2d-1,K}(T)|}\langle K_1\rangle^{-\ell}\cdots \langle K_{2d+1}\rangle^{-\ell}
     \\
    =&\sum_{\mu_1\in\mathbb{N}_+}\cdots\sum_{\mu_{d-1}\in\mathbb{N}_+}\frac{L^{2d-2}}{\mu_{1}\cdots\mu_{d-1}}\sum_{\substack{\mathcal{S}_{2d+1,K}=0\\|\Omega_{3,K}(T)|=\frac{\mu_{1}}{L^2},\cdots,|\Omega_{2d-1,K}(T)|=\frac{\mu_{d-1}}{L^2},\Omega_{2d+1,K}=0}} \Big(\prod^{d-2}_{k=0}1_{\CA_{k}(T)}\langle K_1\rangle^{-\ell}\cdots \langle K_{2d+1}\rangle^{-\ell}\Big)
    \\
    =&\sum_{0\le\mu_k\le L^{10d},\forall 1\le k\le d-1}+\sum_{\mu_k\ge L^{10d},\exists k\le d-1}
\end{split}    
\end{equation}

\textbf{Estimate First Sum in (\ref{eq:lemma2main}):} For the first sum, since $K_1-\cdots+K_{2d+1}=K$, we have $\langle K_1\rangle^{-\ell}\cdots \langle K_{2d+1}\rangle^{-\ell}\le \langle K\rangle^{-l}$ (consider largest $K_{j}$ and bound it by $\langle K\rangle^{-\ell}$, and bound all other $\langle K_j\rangle^{-\ell}$ by $1$). Therefore, 

\begin{align}
    &\sum_{0\le\mu_k\le L^{10d},\forall 1\le k\le d-1}\frac{L^{2d-2}}{\mu_{1}\cdots\mu_{d-1}}\sum_{\substack{\mathcal{S}_{2d+1,K}=0\\|\Omega_{3,K}(T)|=\frac{\mu_{1}}{L^2},\cdots,|\Omega_{2d-1,K}(T)|=\frac{\mu_{d-1}}{L^2},\Omega_{2d+1,K}=0}} \Big(\prod^{d-2}_{k=0}1_{\CA_{k}(T)}\langle K_1\rangle^{-\ell}\cdots \langle K_{2d+1}\rangle^{-\ell}\Big)\notag
    \\
    \le& \frac{L^{2d-2}}{\langle K\rangle^{\ell}} \sum_{0\le\mu_k\le L^{10d},\forall 1\le k\le d-1}\frac{1}{\mu_{1}\cdots\mu_{d-1}}\sum_{\substack{\mathcal{S}_{2d+1,K}=0\\|\Omega_{3,K}(T)|=\frac{\mu_{1}}{L^2},\cdots,|\Omega_{2d-1,K}(T)|=\frac{\mu_{d-1}}{L^2},\Omega_{2d+1,K}=0}} \prod^{d-2}_{k=0}1_{\CA_{k}(T)}\notag
    \\
    =& \frac{L^{2d-2}}{\langle K\rangle^{\ell}} \sum_{0\le\mu_k\le L^{10d},\forall 1\le k\le d-1}\frac{1}{\mu_{1}\cdots\mu_{d-1}}\#\{(K_1,...,K_{2d+1})\in\cap_{k=0}^{d-2}\CA_{k}(T):\mathcal{S}_{2d+1,K}=0,|\Omega_{3,K}(T)|=\frac{\mu_{1}}{L^2}\notag
    \\
    &,\cdots,|\Omega_{2d-1,K}(T)|=\frac{\mu_{d-1}}{L^2},\Omega_{2d+1,K}=0\}\notag
    \\
    =& \frac{L^{2d-2}}{\langle K\rangle^{\ell}} \sum_{0\le|\mu_k|\le L^{10d},\forall 1\le k\le d-1}\frac{1}{\mu_{1}\cdots\mu_{d-1}} \#_{\mu_1,...,\mu_{d-1};K}\label{eq:firstsumlemma2}
\end{align}

Here in the last step we introduce the notation 
\begin{equation}\label{eq:mainnumbertheorydef}
\begin{split}
&\#_{\mu_1,...,\mu_{d-1},K}=\#\{(K_1,...,K_{2d+1})\in\cap_{k=0}^{d-2}\CA_{k}(T):\mathcal{S}_{2d+1,K}=0,|\Omega_{3,K}(T)|=\frac{\mu_{1}}{L^2}
    \\
    &,\cdots,|\Omega_{2d-1,K}(T)|=\frac{\mu_{d-1}}{L^2},\Omega_{2d+1,K}=0\}.
\end{split}
\end{equation}
We want to show that 
\begin{equation}\label{eq:mainnumbertheoryineq}
    \#_{\mu_1,...,\mu_{d-1};K}\lesssim \mu_{1}^{+}\cdots\mu_{d-1}^{+}.
\end{equation}

Replacing $K_j$ by $K_j+K$ in the definition of $\#_{\mu_1,...,\mu_{d-1};K}$, 
we get $\#_{\mu_1,...,\mu_{d-1};K}=\#_{\mu_1,...,\mu_{d-1};0}$. Denote $\#_{\mu_1,...,\mu_{d-1}}=\#_{\mu_1,...,\mu_{d-1};0}$. We want to show that

\begin{equation}\label{eq:estimatesolution}
    \#_{\mu_1,...,\mu_{d-1}}\lesssim \mu_{1}^{+}\cdots\mu_{d-1}^{+}.
\end{equation}

To prove this, we need the following three lemmas.

\begin{lemma}\label{l:numbertheory2}
For any $\mu,\ \mu\in \mathbb{Z}$,
\begin{equation}
    \#\{(K_1,K_2,K_3)\in\mathbb{Z}_L^3:K_1\ne K_2,K_2\ne K_3, K_1-K_2+K_3=\frac{\mu}{L^2}, K_1^2-K_2^2+K_3^2=\frac{\mu'}{L^2}\}\lesssim |\mu'-\mu^2|^+.
\end{equation}
\end{lemma}
\begin{proof} Let $K_j=\frac{k_j}{L}$ with $k_j\in\mathbb{Z}$, $j=1,2,3$, we have
\begin{equation}
\begin{split}
    &\#\{(K_1,K_2,K_3)\in\mathbb{Z}_L^3:K_1\ne K_2,K_2\ne K_3, K_1-K_2+K_3=\frac{\mu}{L^2}, K_1^2-K_2^2+K_3^2=\frac{\mu'}{L^2}\}
    \\
    =&\#\{(k_1,k_2,k_3)\in\mathbb{Z}^3:k_1\ne k_2,k_2\ne k_3, k_1-k_2+k_3=\mu, k_1^2-k_2^2+k_3^2=\mu'\}
    \\
    =&\#\{(k_1,k_2,k_3)\in\mathbb{Z}^3:k_1\ne k_2,k_2\ne k_3, k_1-k_2+k_3=\mu, (k_1-k_2)(k_2-k_3)=\mu'-\mu^2\}
    \\
    =&\sum_{ab=|\mu'-\mu^2|} \#\{(k_1,k_2,k_3)\in\mathbb{Z}^3:k_1\ne k_2,k_2\ne k_3, k_1-k_2+k_3=\mu, k_1-k_2=a, k_2-k_3=b\}
    \\
    \le&\sum_{ab=|\mu'-\mu^2|} 1=\#\{(a,b)\in\mathbb{Z}^2:ab=|\mu'-\mu^2|\}
\end{split}
\end{equation}
Now we need the following lemma,
\begin{lemma}\label{lemma3}
For any number $n\in\mathbb{Z}$, the number of its divisors is bounded by $n^+$. 
\end{lemma}
\begin{proof}
We follow proposition 1.15 in \cite{numbertheorybook}. Let $n=p_1^{\alpha_1}\cdots p_r^{\alpha_r}$, then $$\#\{k:k|n\}=(\alpha_1+1)\cdots(\alpha_r+1).$$ 

Notice that $ln(n)=\alpha_1 ln(p_1)+\cdots+\alpha_r ln(p_r)$, so we have  $\alpha_i\le\frac{ln(n)}{ln(2)}$. And we also have $ln(1+\alpha_i)\le \alpha_i$. We will apply the first inequality when $p_i$ is small and the second when $p_i$ is large. Set a threshold $c$ which will be given later. Define $m=\max\{i:ln(p_{i})<c\}$ ($m=0$ if no such $i$). Thus $ln(m)<ln(p_{m})<c$.
\begin{align*}
    ln(\#\{k:k|n\})&=ln(\alpha_1+1)+\cdots+ln(\alpha_r+1)\\
    &=\sum_{k=1}^m ln(\alpha_k+1)+\sum_{k=m}^r ln(\alpha_k+1)\\
    &\le m\cdot ln(\frac{ln(n)}{ln(2)}+1)+\sum_{k=m}^r \alpha_k\\
    &\le e^c ln(\frac{ln(n)}{ln(2)}+1)+\frac{1}{c}\sum_{k=m}^r \alpha_k ln(p_k)\\
    &\le e^c ln(\frac{ln(n)}{ln(2)}+1)+\frac{1}{c}ln(n)
\end{align*}

Here the fourth inequlity is because $ln(m)<c$ and $ln(n)=\alpha_1 ln(p_1)+\cdots+\alpha_r ln(p_r)$.

Take $c=(1-\delta)lnln(n)$. We have 
\begin{align*}
    ln(\#\{k:k|n\})&\le ln^{1-\delta}(n) ln(\frac{ln(n)}{ln(2)}+1)+(1-\delta)^{-1}\frac{ln(n)}{lnln(n)}\lesssim \frac{ln(n)}{lnln(n)}\lesssim n^+.
\end{align*}
\end{proof}

Recall $ab=|\mu'-\mu^2|$, by Lemma \ref{lemma3}, there are $|\mu'-\mu^2|^+$ possible values of $a$ and $b$. Thus $\#\{(a,b)\in\mathbb{Z}^2:ab=|\mu'-\mu^2|\}\le\#\{\text{divisor of } |\mu'-\mu^2|\}^2\le |\mu'-\mu^2|^+$. This completes the proof of Lemma \ref{l:numbertheory2}.
\end{proof}

Given Lemma \ref{l:numbertheory2}, we can prove the following lemma.

\begin{lemma}\label{l:numbertheory3}
Fix $T\in\CG_{d-1}$. If $(K_1,...,K_{2d+1})\in$ $\{(K_1,...,K_{2d+1})\in\cap_{k=0}^{d-2}\CA_{k}(T):\mathcal{S}_{2d+1,0}=0,|\Omega_{3,0}(T)|=\frac{\mu_{1}}{L^2},\cdots,|\Omega_{2d-1,0}(T)|=\frac{\mu_{d-1}}{L^2},\Omega_{2d+1,0}=0\}$, then there are at most $\mu_{1}^{+}\cdots\mu_{d-1}^{+}$ possible values of $S_{T_{*}}$ (see Lemma \ref{l:formulaH^d} for its definition) for any node $*\in T$.
\end{lemma}
\begin{proof}
This lemma is proved by induction on the level of the nodes.

For the three nodes on level $0$, $*_1$, $*_2$, $*_3$, we have $S_{T_{*_1}}-S_{T_{*_2}}+S_{T_{*_3}}=0$ and $S_{T_{*_1}}^2-S_{T_{*_2}}^2+S_{T_{*_3}}^2=\pm\frac{\mu_1}{L^2}$. Then by Lemma \ref{l:numbertheory2}, there are at most $\mu_{1}^{+}\cdots\mu_{d-1}^{+}$ possible values of them. Note that the inequality constrains in $\CA_{1}(T)$ guarantees $S_{T_{*_1}}\ne S_{T_{*_2}}$ and $S_{T_{*_2}}\ne S_{T_{*_3}}$.

Assume that this lemma is true for nodes of level less than $k$, let us consider nodes on $k$-th level.

Let $*_1$,..., $*_{l^T_{k-1}}$,..., $*_{2k+1}$ be all nodes of level $k$. $*_{l^T_{k-1}}$ has three children $*_{l^T_{k-1}}'$, $*_{l^T_{k-1}+1}'$, $*_{l^T_{k-1}+2}'$. $*_j$ has only one child $*_j'$ for $j<l^T_{k-1}$, one child $*_{j+2}'$ for $j>l^T_{k-1}$.

If $j\ne l^T_{k-1}$, $S_{*_j'}=S_{*_j}$ or $S_{*_{j+2}'}=S_{*_j}$ depending on $j<l^T_{k-1}$ or $j>l^T_{k-1}$. Thus by induction assumption, there are at most $\mu_{1}^{+}\cdots\mu_{d-1}^{+}$ possible values of them.

If $j= l^T_{k-1}$, $S_{*_j'}-S_{*_{j+1}'}+S_{*_{j+1}'}=S_{*_j}$ and $S_{*_j'}^2-S_{*_{j+1}'}^2+S_{*_{j+1}'}^2=\pm\frac{\mu_{k}}{L^2}-S_{*_1}^2+\cdots$. By induction assumption, the right hand side of previous equations have only $\mu_{1}^{+}\cdots\mu_{d-1}^{+}$ possible values. By the inequality constrains in $\CA_{k}$, $S_{*_j'}\ne S_{*_{j+1}'}$, $S_{*_{j+1}'}\ne S_{*_{j+2}'}$. Thus by Lemma \ref{l:numbertheory2}, there are at most $\mu_{1}^{+}\cdots\mu_{d-1}^{+}$ possible values of  $S_{*_j'}$, $S_{*_{j+1}'}$, $S_{*_{j+1}'}$. So we complete the induction.
\end{proof}

If $*$ is the $j$-th nodes on the bottom of $T$, then $S_{T_{*}}=K_j$. Thus there are at most $\mu_{1}^{+}\cdots\mu_{d-1}^{+}$ possible values of all $K_j$, $1\le j\le 2d+1$. This fact implies $\#_{\mu_1,...,\mu_{d-1}}\lesssim \mu_{1}^{+}\cdots\mu_{d-1}^{+}$. So from (\ref{eq:firstsumlemma2}), we know that the first sum in (\ref{eq:lemma2main}) can be bounded by

\begin{equation}\label{eq:lemma2firstpart}
\begin{split}
    &\frac{L^{2d-2}}{\langle K\rangle^{\ell}} \sum_{0\le\mu_k\le L^{10d},\forall 1\le k\le d-1}\frac{1}{\mu_{1}\cdots\mu_{d-1}}\mu_{1}^{+}\cdots\mu_{d-1}^{+}
    \\
    \lesssim& L^{2d-2+}\langle K\rangle^{-\ell} 
\end{split}    
\end{equation}

So the first sum can be bounded by the right hand side of (\ref{eq:lemma2mainreduction}).

\textbf{Estimate Second Sum in (\ref{eq:lemma2main}):} Since $\mu_{k}\ge L^{10d}\Leftrightarrow |\Omega_{2k+1,K}|\ge L^{10d-2}$, by a minute of reflection, we know that the second sum equals to 
\begin{equation}
    \sum_{\substack{\mathcal{S}_{2d+1,K}=0\\\Omega_{2d+1,K}=0\\\exists k,\ |\Omega_{2k+1,K}(T)|\ge L^{10d-2} }}\frac{\prod^{d-2}_{k=0}1_{\CA_{k}(T)}}{|\Omega_{3,K}(T)\Omega_{5,K}(T)\cdots\Omega_{2d-1,K}(T)|}\langle K_1\rangle^{-\ell}\cdots \langle K_{2d+1}\rangle^{-\ell}
\end{equation}

By inequality constrains in $\CA_{j}$, $\Omega_{2j+1,K}(T)\ne0$. Thus $|\Omega_{2j+1,K}(T)|\ge \frac{1}{L^2}$ by the fact that $(K_1,...,K_{2d+1})\in\mathbb{Z}_L^{2d+1}$.

So above expression can be bounded by 
\begin{equation}\label{eq:lemma2longequation}
\begin{split}
    &\sum_{\substack{\mathcal{S}_{2d+1,K}=0\\\Omega_{2d+1,K}=0\\\exists k,\ |\Omega_{2k+1,K}(T)|\ge L^{10d-2} }}L^{2d-2} L^{-(10d-2)}\langle K_1\rangle^{-\ell}\cdots \langle K_{2d+1}\rangle^{-\ell}
    \\
    \lesssim& L^{-8d}\sum_{\mathcal{S}_{2d+1,K}=0}\langle K_1\rangle^{-\ell}\cdots \langle K_{2d+1}\rangle^{-\ell}
    \\
    \lesssim& L^{-8d}\sum_{j=1}^{2d+1}\sum_{\substack{\mathcal{S}_{2d+1,K}=0\\|K_j|=\max_{m}|K_m|}} \langle K_1\rangle^{-\ell}\cdots \langle K_{2d+1}\rangle^{-\ell}
    \\
    \lesssim& L^{-8d}\sum_{\substack{\mathcal{S}_{2d+1,K}=0\\|K_1|=\max_{m}|K_m|}} \langle K_1\rangle^{-\ell}\cdots \langle K_{2d+1}\rangle^{-\ell}
    \\
    \lesssim& L^{-8d}\langle K\rangle^{-\ell}\sum_{\substack{\mathcal{S}_{2d+1,K}=0\\|K_1|=\max_{m}|K_m|}} \langle K_2\rangle^{-\ell}\cdots \langle K_{2d+1}\rangle^{-\ell}
    \\
    \lesssim& L^{-8d}\langle K\rangle^{-\ell} \sum_{(K_2,...,K_{2d+1})\in\mathbb{Z}_L^{2d}} \langle K_2\rangle^{-\ell}\cdots \langle K_{2d+1}\rangle^{-\ell}
    \\
    =& L^{-8d}\langle K\rangle^{-\ell} \left(\sum_{K_2\in\mathbb{Z}_L} \langle K_2\rangle^{-\ell}\right)^{2d}
    \\
    \lesssim& L^{-6d}\langle K\rangle^{-\ell}
\end{split}
\end{equation}
Here the third inequality uses the symmetry between $K_j$, $1\le j\le 2d+1$. The last inequality follows from the elementry inequality $\sum_{K_2\in\mathbb{Z}_L} \langle K_2\rangle^{-\ell}\le L$.

So the second sum can also be bounded by the right hand side of (\ref{eq:lemma2mainreduction}). Thus (\ref{eq:lemma2mainreduction}) is proved. We have finished the first conclusion of this lemma. 

\textbf{Proof of Other Conclusions:} The proof of the second conclusion of this lemma is essentially the same. Now let's describe how to prove the third one. As same as (\ref{eq:lemma2main}), we split the sum into two parts. The second part can be bounded by the same strategy as described above. The first sum can be bounded by

\begin{align}
    &\sum_{0\le\mu_k\le L^{10(P+1)},\forall 1\le k\le P}\frac{L^{2P}}{\mu_{1}\cdots\mu_{P}}\sum_{\substack{\mathcal{S}_{2P+3,K}=0\\|\Omega_{3,K}(T)|=\frac{\mu_{1}}{L^2},\cdots,|\Omega_{2P+1,K}(T)|=\frac{\mu_{P}}{L^2}}} \Big(\prod^{P-1}_{k=0}1_{\CA_{k}(T)}\langle K_1\rangle^{-\ell}\cdots \langle K_{2P+3}\rangle^{-\ell}\Big).\notag
\end{align}

Replacing $K_j$ by $K_j+K$, we get

\begin{align}
    &\sum_{0\le\mu_k\le L^{10(P+1)},\forall 1\le k\le P}\frac{L^{2P}}{\mu_{1}\cdots\mu_{P}}\sum_{\substack{\mathcal{S}_{2P+3,0}=0\\|\Omega_{3,0}(T)|=\frac{\mu_{1}}{L^2},\cdots,|\Omega_{2P+1,0}(T)|=\frac{\mu_{P}}{L^2}}} \Big(\prod^{P-1}_{k=0}1_{\CA_{k}(T)}\langle K_1+K\rangle^{-\ell}\cdots \langle K_{2P+3}+K\rangle^{-\ell}\Big).\notag
\end{align}

Note we no longer have the restriction $\Omega_{2d+3,K}=0$. Consider

$$\{(K_1,...,K_{2d+1})\in\cap_{k=0}^{d-2}\CA_{k}(T):\mathcal{S}_{2d+1,0}=0,|\Omega_{3,0}(T)|=\frac{\mu_{1}}{L^2},\cdots,|\Omega_{2d-1,0}(T)|=\frac{\mu_{d-1}}{L^2}\}.$$

A similar induction works as in Lemma \ref{l:numbertheory3}, except for the last step, due to the absence of $\Omega_{2d+3,K}=0$. But we know that for nodes $*_1$, ...,$*_{2P+1}$ on level $P-1$, $S_{*_1}$, ..., $S_{*_1}$ has only $\mu_1^+\cdots\mu_{P}^+$ possible values except for $S_{*_{l^T_{P-1}}}$, $S_{*_j}=K_j$ or $K_{j+2}$ according to $j<l^T_{P-1}$ or $j>l^T_{P-1}$. Thus $K_1$, ..., $K_{l^T_{P-1}}-K_{l^T_{P-1}+1}+K_{l^T_{P-1}+2}$, ..., $K_{2P+3}$ all have only $\mu_1^+\cdots\mu_{P}^+$ possible values.

Let $j_0$ satisfies $|K_{j_0}+K|=\max_{j}|K_j+K|$. If $j_0\notin\{l^T_{P-1},l^T_{P-1}+1,l^T_{P-1}+2\}$, then bound $\langle K_{j_0}\rangle^{-\ell}$ by $\langle K\rangle^{-\ell}$ and other $\langle K_{j}\rangle^{-\ell}$ by $1$, the first sum can be further bounded by
\begin{equation}
\begin{split}
    &\frac{L^{2P}}{\langle K\rangle^{\ell}} \sum_{0\le\mu_k\le L^{10(P+1)},\forall 1\le k\le d-1}\frac{\mu_1^+\cdots\mu_{P}^+}{\mu_{1}\cdots\mu_{d-1}}\sum_{\substack{K_{l^T_{P-1}}-K_{l^T_{P-1}+1}+K_{l^T_{P-1}+2}\in \mathscr{K}}} \langle K_{l^T_{P-1}}+K\rangle^{-\ell}
    \\
    &\times\langle K_{l^T_{P-1}+1}+K\rangle^{-\ell}\langle K_{l^T_{P-1}+2}+K\rangle^{-\ell}
    \\
    =&\frac{L^{2P}}{\langle K\rangle^{\ell}} \sum_{0\le\mu_k\le L^{10(P+1)},\forall 1\le k\le d-1}\frac{\mu_1^+\cdots\mu_{P}^+}{\mu_{1}\cdots\mu_{d-1}}\sum_{\substack{a\in \mathscr{K}}}\sum_{K_{l^T_{P-1}}, K_{l^T_{P-1}+1}\in\mathbb{Z}}\langle K_{l^T_{P-1}}+K\rangle^{-\ell}
    \\
    &\times\langle K_{l^T_{P-1}+1}+K\rangle^{-\ell}\langle a-K_{l^T_{P-1}}+K_{l^T_{P-1}+1}+K\rangle^{-\ell}
    \\
    \lesssim & \frac{L^{2P}}{\langle K\rangle^{\ell}} \sum_{0\le\mu_k\le L^{10(P+1)},\forall 1\le k\le d-1}\frac{\mu_1^+\cdots\mu_{P}^+}{\mu_{1}\cdots\mu_{d-1}}\sum_{K_{l^T_{P-1}}, K_{l^T_{P-1}+1}\in\mathbb{Z}} \langle K_{l^T_{P-1}}+K\rangle^{-\ell}\langle K_{l^T_{P-1}+1}+K\rangle^{-\ell}
    \\
    \lesssim & \frac{L^{2P+}}{\langle K\rangle^{\ell}} (\sum_{K_{l^T_{P-1}}\in\mathbb{Z}} \langle K_{l^T_{P-1}}+K\rangle^{-\ell})^2\lesssim \frac{L^{2P+2+}}{\langle K\rangle^{\ell}},
\end{split}
\end{equation}
where $\mathscr{K}$ is some set with cardinality $\le \mu_1^+\cdots\mu_{P}^+$. The last inequality follows from the elementry inequality $\sum_{K_2\in\mathbb{Z}_L} \langle K_{l^T_{P-1}}+K\rangle^{-\ell}\le L$. 
 
If $j_0\in\{l^T_{P-1},l^T_{P-1}+1,l^T_{P-1}+2\}$, without loss of generality assume that $j_0=l^T_{P-1}+2$, we have $\langle K_{l^T_{P-1}+2}+K\rangle\ge \langle K\rangle$. After bounding all $\langle K_{j}\rangle^{-\ell}$ by $1$, the first sum can be further bounded by
\begin{equation}
\begin{split}
    &L^{2P} \sum_{0\le\mu_k\le L^{10(P+1)},\forall 1\le k\le d-1}\frac{\mu_1^+\cdots\mu_{P}^+}{\mu_{1}\cdots\mu_{d-1}}\sum_{\substack{K_{l^T_{P-1}}-K_{l^T_{P-1}+1}+K_{l^T_{P-1}+2}\in \mathscr{K}\\\langle K_{l^T_{P-1}+2}+K\rangle\ge \langle K\rangle}} \langle K_{l^T_{P-1}}+K\rangle^{-\ell}
    \\
    &\times\langle K_{l^T_{P-1}+1}+K\rangle^{-\ell}\langle K_{l^T_{P-1}+2}+K\rangle^{-\ell}
    \\
    =&\frac{L^{2P}}{\langle K\rangle^{\ell}} \sum_{0\le\mu_k\le L^{10(P+1)},\forall 1\le k\le d-1}\frac{\mu_1^+\cdots\mu_{P}^+}{\mu_{1}\cdots\mu_{d-1}}\sum_{K_{l^T_{P-1}}, K_{l^T_{P-1}+1}\in\mathbb{Z}} \langle K_{l^T_{P-1}}+K\rangle^{-\ell}\langle K_{l^T_{P-1}+1}+K\rangle^{-\ell}
    \\
    \lesssim & \frac{L^{2P}}{\langle K\rangle^{\ell}} \sum_{0\le\mu_k\le L^{10(P+1)},\forall 1\le k\le d-1}\frac{\mu_1^+\cdots\mu_{P}^+}{\mu_{1}\cdots\mu_{d-1}}\sum_{K_{l^T_{P-1}}, K_{l^T_{P-1}+1}\in\mathbb{Z}} \langle K_{l^T_{P-1}}+K\rangle^{-\ell}\langle K_{l^T_{P-1}+1}+K\rangle^{-\ell}
    \\
    \lesssim & \frac{L^{2P+}}{\langle K\rangle^{\ell}} (\sum_{K_{l^T_{P-1}}\in\mathbb{Z}} \langle K_{l^T_{P-1}}+K\rangle^{-\ell})^2\lesssim \frac{L^{2P+2+}}{\langle K\rangle^{\ell}}
\end{split}
\end{equation}
where $\mathscr{K}$ is some set with cardinality $\le \mu_1^+\cdots\mu_{P}^+$.
So we get a bound on the first sum, thus we finished the proof of the last conclusion. This completes the proof of Lemma \ref{l:numbertheory1}.
\end{proof} 

\begin{lemma}\label{boundNH} For any $3\le d\leq P$, $\ell>1$,
$$
\| H^{d}_K (u) \|_{X^\ell_K} \lesssim L^{2(d-1)+} \| u \|_{X^\ell}^{2d+1}~.
$$
\end{lemma}
\begin{proof}
Recall that 
\begin{equation}\label{eq:lemma6,1}
    H^{d}_{K}(u)=\sum_{\substack{\mathcal{S}_{2d+1,K} =0\\\Omega_{2d+1,K}=0}} H^d_{K_1K_2...K_{2d+1}}d_{K_1}\overline{d_{K_2}}...d_{K_{2d+1}}e(\int_0^tD_{2d+1,K}(s)ds),
\end{equation}

By (\ref{eq:formulaH^d}),
\begin{equation}
    H^{d}_{K_1...K_{2d+1}}=\sum_{T\in\CG_{d-1}}\frac{(-1)^{\sum_{k=0}^{d-2}l^T_{k}}\prod^{d-2}_{k=0}1_{\CA_{k}(T)}}{\Omega_{3,K}(T)\Omega_{5,K}(T)\cdots\Omega_{2d-1,K}(T)}(K_1,...,K_{2d+1}).
\end{equation}

By (\ref{eq:lemma6,1}), 
\begin{equation}
\begin{split}
    |H^{d}_{K}(u)|&\le\sum_{\substack{\mathcal{S}_{2d+1,K} =0\\\Omega_{2d+1,K}=0}} |H^d_{K_1K_2...K_{2d+1}}||d_{K_1}\overline{d_{K_2}}...d_{K_{2d+1}}|
    \\
    &\le \sum_{T\in\CG_{d-1}}\sum_{\substack{\mathcal{S}_{2d+1,K} =0\\\Omega_{2d+1,K}=0}}\frac{\prod^{d-2}_{k=0}1_{\CA_{k}(T)}}{|\Omega_{3,K}(T)\Omega_{5,K}(T)\cdots\Omega_{2d-1,K}(T)|}(K_1,...,K_{2d+1})|d_{K_1}\overline{d_{K_2}}...d_{K_{2d+1}}|.
\end{split}
\end{equation}

By the first conclusion in Lemma \ref{l:numbertheory1},
\begin{equation}
\begin{split}
    |H^{d}_{K}(u)|\leq&|\CG_{d-1}|||d_{K}||_{X^\ell_{K}}^{2d+1}\frac{L^{2(d-1)+}}{\langle K\rangle^{\ell}}
    \lesssim ||u||_{X^\ell}^{2d+1}\frac{L^{2(d-1)+}}{\langle K\rangle^{\ell}},
\end{split}
\end{equation}
where $|\CG_{d-1}|$ is the number of trees in $\CG_{d-1}$.
Thus $||H^{d}_{K}(u)||_{X^\ell}\lesssim L^{2(d-1)+}||u||_{X^\ell}^{2d+1}$. This completes the proof. 
\end{proof}

\begin{lemma} \label{boundNF} For any $1\le d\leq P$, $\ell>1$,
$$
\| F^{d}_K (u) \|_{X^\ell_K} \lesssim L^{2d+} \| u \|_{X^\ell}^{2d+1}~.
$$
\end{lemma}

\begin{proof}
Recall that 
\begin{equation}
    F^{d}_{K}(u)=\sum_{\substack{\mathcal{S}_{2d+1,K}=0\\\Omega_{2d+1,K}\ne0}} H^{d}_{K_1K_2...K_{2d+1}}d_{K_1}\overline{d_{K_2}}...d_{K_{2d+1}} \frac{1}{2\pi\Omega_{2d+1,K}} e(\int_0^t\widetilde{\Omega}_{2d+1,K}(s)ds).
\end{equation}

Denote 
\begin{equation}
    F^d_{K_1K_2...K_{2d+1}}=\frac{H^d_{K_1K_2...K_{2d+1}}1_{\Omega_{2d+1,K}\ne0}}{2\pi\Omega_{2d+1,K}(K_1,...,K_{2d+1})},
\end{equation}
then
\begin{equation}\label{notation:F'}
    F^{d}_{K}(u)=\sum_{\substack{\substack{\mathcal{S}_{2d+1,K} =0\\\Omega_{2d+1,K}\ne0}}} F^d_{K_1K_2...K_{2d+1}}b_{K_1}\overline{b_{K_2}}...b_{K_{2d+1}} e(\int_0^t\widetilde{\Omega}_{2d+1,K}(s)ds).
\end{equation}

By Lemma \ref{l:formulaH^d},
\begin{equation}\label{eq:formulaF^d}
    F^{d}_{K_1...K_{2d+1}}=\sum_{T\in\CG_{d-1}}\frac{(-1)^{\sum_{k=0}^{d-2}l^T_{k}}\prod^{d-2}_{k=0}1_{\CA_{k}(T)}}{\Omega_{3,K}(T)\Omega_{5,K}(T)\cdots\Omega_{2d+1,K}(T)}(K_1,...,K_{2d+1}).
\end{equation}

By (\ref{notation:F'}), 
\begin{equation}
\begin{split}
    |F^{d}_{K}(u)|&\le\sum_{\substack{\mathcal{S}_{2d+1,K} =0\\\Omega_{2d+1,K}\ne0}} |F^d_{K_1K_2...K_{2d+1}}||b_{K_1}\overline{b_{K_2}}...b_{K_{2d+1}}|
    \\
    &\le \sum_{T\in\CG_{d-1}}\sum_{\substack{\mathcal{S}_{2d+1,K} =0\\\Omega_{2d+1,K}\ne0}}\frac{\prod^{d-2}_{k=0}1_{\CA_{k}(T)}}{|\Omega_{3,K}(T)\Omega_{5,K}(T)\cdots\Omega_{2d-1,K}(T)|}(K_1,...,K_{2d+1})|d_{K_1}\overline{d_{K_2}}...d_{K_{2d+1}}|.
\end{split}
\end{equation}

By the second conclusion in Lemma \ref{l:numbertheory1}, 
\begin{equation}
\begin{split}
    |F^{d}_{K}(u)|\leq|\CG_{d-1}|||d_{K}||_{X^\ell_{K}}^{2d+1}\frac{L^{2d+}}{\langle K\rangle^{\ell}}
    \lesssim ||u||_{X^\ell}^{2d+1}\frac{L^{2d+}}{\langle K\rangle^{\ell}},
\end{split}
\end{equation}

where $|\CG_{d-1}|$ is the number of trees in $\CG_{d-1}$.
Thus $||F^{d}_{K}(u)||_{X^\ell}\lesssim L^{2(d-1)+}||u||_{X^\ell}^{2d+1}$.
\end{proof}

\begin{lemma}\label{boundNtildeH}
For any $\ell>1$
$$
\| \widetilde{H}^{P+1}_K \|_{X^\ell_K} \lesssim L^{2P+2+} \| u \|_{X^\ell}^{2P+3}~.
$$
\end{lemma}
\begin{proof}
Recall that 
\begin{equation}\label{eq:lemma8,1}
    \widetilde{H}^{P+1}_{K}(u)=\sum_{\substack{\mathcal{S}_{2P+3,K} =0}} H^{P+1}_{K_1K_2...K_{2P+3}}d_{K_1}\overline{d_{K_2}}...d_{K_{2P+3}} e(\int_0^t\widetilde{\Omega}_{2P+3,K}(s)ds),
\end{equation}

By (\ref{eq:formulaH^d}),
\begin{equation}
    H^{P+1}_{K_1...K_{2P+3}}=\sum_{T\in\CG_{P}}\frac{(-1)^{\sum_{k=0}^{P-1}l^T_{k}}\prod^{P-1}_{k=0}1_{\CA_{k}(T)}}{\Omega_{3,K}(T)\Omega_{5,K}(T)\cdots\Omega_{2P+1,K}(T)}(K_1,...,K_{2P+3}).
\end{equation}

By (\ref{eq:lemma8,1}), 
\begin{equation}
\begin{split}
    |\widetilde{H}^{P+1}_{K}(u)|&\le\sum_{\substack{\mathcal{S}_{2P+3,K} =0}} |H^{P+1}_{K_1K_2...K_{2P+3}}||d_{K_1}\overline{d_{K_2}}...d_{K_{2P+3}}|
    \\
    &\le \sum_{T\in\CG_{P}}\sum_{\substack{\mathcal{S}_{2P+3,K} =0}}\frac{\prod^{P-1}_{k=0}1_{\CA_{k}(T)}}{|\Omega_{3,K}(T)\Omega_{5,K}(T)\cdots\Omega_{2P+1,K}(T)|}(K_1,...,K_{2P+3})|d_{K_1}\overline{d_{K_2}}...d_{K_{2P+3}}|.
\end{split}
\end{equation}

By the last conclusion in Lemma \ref{l:numbertheory1}, 
\begin{equation}
\begin{split}
    |\widetilde{H}^{P+1}_{K}(u)|\leq|\CG_{P}|||d_{K}||_{X^\ell_{K}}^{2P+3}\frac{L^{2P+2+}}{\langle K\rangle^{\ell}}
    \lesssim ||u||_{X^\ell}^{2P+3}\frac{L^{2P+2+}}{\langle K\rangle^{\ell}},
\end{split}
\end{equation}
where $|\CG_{P}|$ is the number of trees in $\CG_{P}$.
Thus $||\widetilde{H}^{P+1}_{K}(u)||_{X^\ell}\lesssim L^{2P+2+}||u||_{X^\ell}^{2P+3}$. 
\end{proof}

\begin{lemma}\label{l:numbertheory4}
Fix $d,P\in\mathbb{N}$, $T\in\CG_{d-1}$ or $\CG_{P}$, $\ell>1$, $s=1,2,3,4$. Let $\{d_{K}\}_{K\in\mathbb{Z}_L}\in X^{\ell}$ i.e. $||d_{K}||_{X^\ell}=\sup_{K\in\mathbb{Z}_L}|\langle K\rangle^\ell d_{K}|<+\infty$. Then for any $K\in\mathbb{Z}_L$, we have

$$\langle K\rangle^\ell\sum_{\substack{\mathcal{S}_{2d+1,K}=0\\\Omega_{2d+1,K}=0}}\frac{1_{\CB^{(s)}_{0}(T)}\prod^{d-2}_{k=1}1_{\CB_{k}(T)}}{|\Omega_{5,K}(T)^2\Omega_{7,K}(T)\cdots\Omega_{2d-1,K}(T)|}(K_1,...,K_{2d+1})d_{K_1}\cdots d_{K_{2d+1}}\lesssim L^{2(d-1)+}||d_{K}||_{X^\ell}^{2d+1},$$

$$\langle K\rangle^\ell\sum_{\substack{\mathcal{S}_{2d+1,K}=0\\\Omega_{2d+1,K}\ne0}}\frac{1_{\CB^{(s)}_{0}(T)}\prod^{d-2}_{k=1}1_{\CB_{k}(T)}}{|\Omega_{5,K}(T)^2\Omega_{7,K}(T)\cdots\Omega_{2d+1,K}(T)|}(K_1,...,K_{2d+1})d_{K_1}\cdots d_{K_{2d+1}}\lesssim L^{2d+}||d_{K}||_{X^\ell}^{2d+1},$$
and

$$\langle K\rangle^\ell\sum_{\substack{\mathcal{S}_{2P+3,K} =0}}\frac{1_{\CB^{(s)}_{0}(T)}\prod^{P-1}_{k=1}1_{\CB_{k}(T)}}{|\Omega_{5,K}(T)^2\Omega_{7,K}(T)\cdots\Omega_{2P+1,K}(T)|}(K_1,...,K_{2P+3})d_{K_1}...d_{K_{2P+3}}\le L^{2(P+1)+}||d_{K}||_{X^\ell}^{2P+1}.$$
\end{lemma}
\begin{proof} Because the idea of this proof is identical to Lemma \ref{l:numbertheory1}, so the proof here will be sketchy. We first prove the first inequality. In the end of the proof, we explain how to prove the second and the third ones similarly. As same as before, it suffices to show that
\begin{equation}\label{eq:lemma8mainreduction}
    \sum_{\substack{\mathcal{S}_{2d+1,K}=0\\\Omega_{2d+1,K}=0}}\frac{1_{\CB^{(s)}_{0}(T)}\prod^{d-2}_{k=1}1_{\CB_{k}(T)}}{|\Omega_{5,K}(T)^2\Omega_{7,K}(T)\cdots\Omega_{2d-1,K}(T)|}\langle K_1\rangle^{-\ell}\cdots \langle K_{2d+1}\rangle^{-\ell}\lesssim L^{2d-2+}\langle K\rangle^{-\ell}.
\end{equation}

Since $K_1,...,K_{2d+1}\in\mathbb{Z}_L$, $\Omega_{2k+1,K}(T)=\frac{\mu_{k}}{L^2}$, $\mu_{k}\in\mathbb{Z}$. The same as before, split the summation into two parts, 

\begin{equation}\label{eq:lemma8main}
\begin{split}
     &\sum_{\substack{\mathcal{S}_{2d+1,K}=0\\\Omega_{2d+1,K}=0}}\frac{1_{\CB^{(s)}_{0}(T)}\prod^{d-2}_{k=1}1_{\CB_{k}(T)}}{|\Omega_{5,K}(T)^2\Omega_{7,K}(T)\cdots\Omega_{2d-1,K}(T)|}\langle K_1\rangle^{-\ell}\cdots \langle K_{2d+1}\rangle^{-\ell}
     \\
    =&\sum_{\mu_2\in\mathbb{N}_+}\cdots\sum_{\mu_{d-1}\in\mathbb{N}_+}\frac{L^{2d-2}}{\mu_{2}^2\cdots\mu_{d-1}}\sum_{\substack{\mathcal{S}_{2d+1,K}=0\\|\Omega_{5,K}(T)|=\frac{\mu_{2}}{L^2},\cdots,|\Omega_{2d-1,K}(T)|=\frac{\mu_{d-1}}{L^2},\Omega_{2d+1,K}=0}} \Big(1_{\CB^{(s)}_{0}(T)}\prod^{d-2}_{k=1}1_{\CB_{k}(T)}\langle K_1\rangle^{-\ell}\cdots \Big)
    \\
    =&\sum_{0\le\mu_k\le L^{10d},\forall 2\le k\le d-1}+\sum_{\mu_k\ge L^{10d},\exists k\le d-1}
\end{split}    
\end{equation}

\textbf{Estimate First Sum in (\ref{eq:lemma8main}):} For the first sum, using the same argument as before, since $K_1-\cdots+K_{2d+1}=K$, we have $\langle K_1\rangle^{-\ell}\cdots \langle K_{2d+1}\rangle^{-\ell}\le \langle K\rangle^{-\ell}$. (Simply consider largest $K_{j}$ and bound it by $\langle K\rangle^{-\ell}$. Bound all other $\langle K_j\rangle^{-\ell}$ by $1$) So we have,

\begin{align}
    &\sum_{0\le\mu_k\le L^{10d},\forall 2\le k\le d-1}\frac{L^{2d-2}}{\mu_{2}^2\cdots\mu_{d-1}}\sum_{\substack{\mathcal{S}_{2d+1,K}=0\\|\Omega_{5,K}(T)|=\frac{\mu_{2}}{L^2},\cdots,|\Omega_{2d-1,K}(T)|=\frac{\mu_{d-1}}{L^2},\Omega_{2d+1,K}=0}} \Big(1_{\CB^{(s)}_{0}(T)}\prod^{d-2}_{k=0}1_{\CB_{k}(T)}\langle K_1\rangle^{-\ell}\cdots \Big)\notag
    \\
    \le& \frac{L^{2d-2}}{\langle K\rangle^{\ell}} \sum_{0\le\mu_k\le L^{10d},\forall 2\le k\le d-1}\frac{1}{\mu_{2}^2\cdots\mu_{d-1}}\sum_{\substack{\mathcal{S}_{2d+1,K}=0\\|\Omega_{5,K}(T)|=\frac{\mu_{2}}{L^2},\cdots,|\Omega_{2d-1,K}(T)|=\frac{\mu_{d-1}}{L^2},\Omega_{2d+1,K}=0}} 1_{\CB^{(s)}_{0}(T)}\prod^{d-2}_{k=0}1_{\CB_{k}(T)}\notag
    \\
    =& \frac{L^{2d-2}}{\langle K\rangle^{\ell}} \sum_{0\le\mu_k\le L^{10d},\forall 2\le k\le d-1}\frac{1}{\mu_{2}^2\cdots\mu_{d-1}}\#\{(K_1,...,K_{2d+1})\in\cap_{k=0}^{d-2}\CB_{k}(T)\cap\CB^{(s)}_{0}(T):\mathcal{S}_{2d+1,K}=0,\notag
    \\
    &|\Omega_{5,K}(T)|=\frac{\mu_{2}}{L^2},\cdots,|\Omega_{2d-1,K}(T)|=\frac{\mu_{d-1}}{L^2},\Omega_{2d+1,K}=0\}\notag
    \\
    =& \frac{L^{2d-2}}{\langle K\rangle^{\ell}} \sum_{0\le|\mu_k|\le L^{10d},\forall 2\le k\le d-1}\frac{1}{\mu_{2}\cdots\mu_{d-1}} \#_{\mu_2,...,\mu_{d-1};K}\label{eq:firstsumlemma8}
\end{align}

Here in the last step we introduce the notation $\#_{\mu_2,...,\mu_{d-1},K}$. We want to show that
\begin{equation}
    \#_{\mu_2,...,\mu_{d-1};K}\lesssim \mu_{2}^{+}\cdots\mu_{d-1}^{+}.
\end{equation}

Replacing $K_j$ by $K_j+K$ in the definition of $\#_{\mu_2,...,\mu_{d-1};K}$, 
we get $\#_{\mu_1,...,\mu_{d-1};K}=\#_{\mu_1,...,\mu_{d-1};0}$. Denote $\#_{\mu_2,...,\mu_{d-1}}=\#_{\mu_2,...,\mu_{d-1};0}$. We want to show that

\begin{equation}\label{eq:estimatesolution'}
    \#_{\mu_2,...,\mu_{d-1}}\lesssim \mu_{2}^{+}\cdots\mu_{d-1}^{+}.
\end{equation}

We prove (\ref{eq:estimatesolution'}) inductively.

\begin{lemma}\label{l:numbertheory5}
Fix $T\in\CG_{d-1}$. If $(K_1,...,K_{2d+1})\in$ $\{(K_1,...,K_{2d+1})\in\cap_{k=0}^{d-2}\CB_{k}(T)\cap\CB^{(s)}_{0}(T):\mathcal{S}_{2d+1,0}=0,|\Omega_{5,0}(T)|=\frac{\mu_{2}}{L^2},\cdots,|\Omega_{2d-1,0}(T)|=\frac{\mu_{d-1}}{L^2},\Omega_{2d+1,0}=0\}$, then there are at most $\mu_{2}^{+}\cdots\mu_{d-1}^{+}$ possible values of $S_{T_{*}}$ (see Lemma \ref{l:formulaH^d} for its definition) for any node $*\in T$.
\end{lemma}
\begin{proof}
This lemma is proved by induction on the level of the nodes. In the proof, we only consider the case that $s=1$. In all other cases the proof has no essential changes. 

Unlike the proof of Lemma \ref{l:numbertheory3}, we don't have $\Omega$ on level $0$. For the five nodes on level $1$, $*_1$, $*_2$, $*_3$, $*_4$, $*_5$, we have $S_{T_{*_1}}-S_{T_{*_2}}+S_{T_{*_3}}-S_{T_{*_4}}+S_{T_{*_5}}=0$ and $S_{T_{*_1}}^2-S_{T_{*_2}}^2+S_{T_{*_3}}^2-S_{T_{*_4}}^2+S_{T_{*_5}}^2=\pm\frac{\mu_2}{L^2}$. Unlike the proof of Lemma \ref{l:numbertheory3}, we have several equations in $\CB^{(s)}_{0}(T)$, i.e. $S_{T_{*_1}}=S_{T_{*_2}}=S_{T_{*_3}}$. 

Applying them, we get $S_{T_{*_3}}-S_{T_{*_4}}+S_{T_{*_5}}=0$ and $S_{T_{*_3}}^2-S_{T_{*_4}}^2+S_{T_{*_5}}^2=\pm\frac{\mu_2}{L^2}$. Then by Lemma \ref{l:numbertheory2}, there are at most $\mu_{1}^{+}\cdots\mu_{d-1}^{+}$ possible values of $S_{T_{*_1}}$, $S_{T_{*_2}}$, $S_{T_{*_3}}$, $S_{T_{*_4}}$, $S_{T_{*_5}}=0$. Note that the inequality constrains in $\CB_{1}(T)$ guarantees $S_{T_{*_3}}\ne S_{T_{*_4}}$ and $S_{T_{*_4}}\ne S_{T_{*_5}}$.

Assume that this lemma is true for nodes of level less than $k$, let us consider nodes on $k$-th level. The following argument is identically the same as that in the proof of Lemma \ref{l:numbertheory3}.

Let $*_1$,..., $*_{l^T_{k-1}}$,..., $*_{2k+1}$ be all nodes of level $k$. $*_{l^T_{k-1}}$ has three children $*_{l^T_{k-1}}'$, $*_{l^T_{k-1}+1}'$, $*_{l^T_{k-1}+2}'$. $*_j$ has only one child $*_j'$ for $j<l^T_{k-1}$, one child $*_{j+2}'$ for $j>l^T_{k-1}$.

If $j\ne l^T_{k-1}$, $S_{*_j'}=S_{*_j}$ or $S_{*_{j+2}'}=S_{*_j}$ depending on $j<l^T_{k-1}$ or $j>l^T_{k-1}$. Thus by induction assumption, there are at most $\mu_{1}^{+}\cdots\mu_{d-1}^{+}$ possible values of them.

If $j= l^T_{k-1}$, $S_{*_j'}-S_{*_{j+1}'}+S_{*_{j+1}'}=S_{*_j}$ and $S_{*_j'}^2-S_{*_{j+1}'}^2+S_{*_{j+1}'}^2=\pm\frac{\mu_{k}}{L^2}-S_{*_1}^2+\cdots$. By induction assumption, the right hand side of previous equations have only $\mu_{1}^{+}\cdots\mu_{d-1}^{+}$ possible values. By the inequality constrains in $\CB_{k}$, $S_{*_j'}\ne S_{*_{j+1}'}$, $S_{*_{j+1}'}\ne S_{*_{j+2}'}$. Thus by Lemma \ref{l:numbertheory2}, there are at most $\mu_{1}^{+}\cdots\mu_{d-1}^{+}$ possible values of  $S_{*_j'}$, $S_{*_{j+1}'}$, $S_{*_{j+1}'}$. So we complete the induction.
\end{proof}

If $*$ is the $j$-th nodes on the bottom of $T$, then $S_{T_{*}}=K_j$. Thus there are at most $\mu_{1}^{+}\cdots\mu_{d-1}^{+}$ possible values of all $K_j$, $1\le j\le 2d+1$. This fact implies $\#_{\mu_1,...,\mu_{d-1}}\lesssim \mu_{1}^{+}\cdots\mu_{d-1}^{+}$. So from (\ref{eq:firstsumlemma8}), we know that the first sum in (\ref{eq:lemma8main}) can be bounded by

\begin{equation}\label{eq:lemma8firstpart}
\begin{split}
    &\frac{L^{2d-2}}{\langle K\rangle^{\ell}} \sum_{0\le\mu_k\le L^{10d},\forall 1\le k\le d-1}\frac{1}{\mu_{1}\cdots\mu_{d-1}}\mu_{1}^{+}\cdots\mu_{d-1}^{+}
    \\
    \lesssim& L^{2d-2+}\langle K\rangle^{-\ell} 
\end{split}    
\end{equation}

So the first sum can be bounded by the right hand side of (\ref{eq:lemma8mainreduction}).

\textbf{Estimate Second Sum in (\ref{eq:lemma8main}):} Since $\mu_{k}\ge L^{10d}\Leftrightarrow |\Omega_{2k+1,K}|\ge L^{10d-2}$, by a minute of reflection, we know that the second sum equals to 
\begin{equation}
    \sum_{\substack{\mathcal{S}_{2d+1,K}=0\\\Omega_{2d+1,K}=0\\\exists k,\ |\Omega_{2k+1,K}(T)|\ge L^{10d-2} }}\frac{1_{\CB^{(s)}_{0}(T)}\prod^{d-2}_{k=1}1_{\CB_{k}(T)}}{|\Omega_{5,K}(T)^2\cdots\Omega_{2d-1,K}(T)|}\langle K_1\rangle^{-l}\cdots \langle K_{2d+1}\rangle^{-l}
\end{equation}
By inequality constrains in $\CB_{j}$, $\Omega_{2j+1,K}(T)\ne0$. Thus $|\Omega_{2j+1,K}(T)|\ge \frac{1}{L^2}$ since $(K_1,...,K_{2d+1})\in\mathbb{Z}_L^{2d+1}$.
So above expression can be bounded by 
\begin{equation}
\begin{split}
    &\sum_{\substack{\mathcal{S}_{2d+1,K}=0\\\Omega_{2d+1,K}=0\\\exists k,\ |\Omega_{2k+1,K}(T)|\ge L^{10d-2} }}L^{2d-2} L^{-(10d-2)}1_{\CB^{(s)}_{0}(T)}\langle K_1\rangle^{-\ell}\cdots \langle K_{2d+1}\rangle^{-\ell}
    \\
    &\lesssim\sum_{\substack{\mathcal{S}_{2d+1,K}=0\\\Omega_{2d+1,K}=0\\\exists k,\ |\Omega_{2k+1,K}(T)|\ge L^{10d-2} }}L^{2d-2} L^{-(10d-2)}\langle K_1\rangle^{-\ell}\cdots \langle K_{2d+1}\rangle^{-\ell}
    \\
    \lesssim& L^{-6d}\langle K\rangle^{-\ell}
\end{split}
\end{equation}
The last inequality follows from the same argument as in (\ref{eq:lemma2longequation}) in Lemma \ref{l:numbertheory1}. Note that $\CB^{(s)}_{0}(T)$ produces two equations. If we take advantage of it in the second step, we can get a gain of $L^{-2}$, which we don't need here.  

The second sum can also be bounded by the right hand side of (\ref{eq:lemma8mainreduction}). Thus (\ref{eq:lemma8mainreduction}) is proved. We have finished the first conclusion of this lemma. 

\textbf{Proof of Other Conclusions:} The proof of the second conclusion of this lemma is essentially the same. The proof of the third one follows from of the argument of the last part in Lemma \ref{l:numbertheory1}, because the same induction argument works in Lemma \ref{l:numbertheory5} as in Lemma \ref{l:numbertheory3}. So we complete the proof of Lemma \ref{l:numbertheory4}.
\end{proof} 

\begin{lemma}\label{boundNG} For any $3\le d\leq P$, $\ell>1$,
$$
\| G^{d}_K (u) \|_{X^\ell_K} \lesssim L^{2(d-1)+} \| u \|_{X^\ell}^{2d+1}~.
$$
\end{lemma}
\begin{proof}
Replacing all the letter 'H' by 'G' and 'Lemma \ref{l:numbertheory1}' by 'Lemma \ref{l:numbertheory4}' in the proof of Lemma \ref{boundNH} we get a proof of this lemma.
\end{proof}

\begin{lemma} \label{boundNE} For any $2\le d\leq P$, $\ell>1$,
$$
\| E^{d}_K (u) \|_{X^\ell_K} \lesssim L^{2d+} \| u \|_{X^\ell}^{2d+1}~.
$$
\end{lemma}
\begin{proof}
Replacing all the letter 'F' by 'E' and 'Lemma \ref{l:numbertheory1}' by 'Lemma \ref{l:numbertheory4}' in the proof of Lemma \ref{boundNF} we get a proof of this lemma.
\end{proof}

\begin{lemma}\label{boundNtildeG}
For any $\ell>1$
$$
\| \widetilde{G}^{P+1}_K \|_{X^\ell_K} \lesssim L^{2d+2+} \| u \|_{X^\ell}^{2d+3}~.
$$
\end{lemma}
\begin{proof}
Replacing all the letter '$\widetilde{H}$' by '$\widetilde{G}$' and 'Lemma \ref{l:numbertheory1}' by 'Lemma \ref{l:numbertheory4}' in the proof of Lemma \ref{boundNtildeH} we get a proof of this lemma.
\end{proof}

\begin{lemma}\label{boundNS}
For any $\ell>1$, $3\le d\leq P$, assume that $\epsilon^2\| u \|_{X^\ell}^2\le\frac{1}{10}$ if $d=3$, then we have 
$$
\|S^{d}_K (u)\|_{X^\ell_K} \lesssim L^{2(d-1)+} \| u \|_{X^\ell}^{2d+1}~.
$$
\end{lemma}
\begin{proof}
Recall that 
\begin{equation}
\begin{split}
    &S^{d+2}_K(u)=\sum_{\substack{\mathcal{S}_{2d+3,K} =0}} (H^{d+1}_{K_1K_2...K_{2d+3}}+G^{d+1}_{K_1K_2...K_{2d+3}}) d_{K_1}\overline{d_{K_2}}...d_{K_{2d+3}}\frac{D_{2d+3,K}(t)}{2\pi\Omega_{2d+3,K}} e(\int_0^t\widetilde{\Omega}_{2d+3,K}(s)ds).
\end{split}
\end{equation}

Denote 
\begin{equation}
    S^{d+2}_{K_1K_2...K_{2d+3}}=\frac{(H^{d+1}_{K_1K_2...K_{2d+3}}+G^{d+1}_{K_1K_2...K_{2d+3}})1_{\Omega_{2d+3,K}\ne0}}{2\pi\Omega_{2d+3,K}(K_1,...,K_{2d+3})},
\end{equation}
then
\begin{equation}\label{notation:F''}
    S^{d+2}_K(u)=\sum_{\substack{\substack{\mathcal{S}_{2d+3,K} =0}}} S^{d+2}_{K_1K_2...K_{2d+3}}b_{K_1}\overline{b_{K_2}}...b_{K_{2d+3}}D_{2d+3,K}(t) e(\int_0^t\widetilde{\Omega}_{2d+3,K}(s)ds).
\end{equation}

By Lemma \ref{l:formulaH^d},
\begin{equation}
\begin{split}
    &S^{d+2}_{K_1...K_{2d+3}}=\sum_{T\in\CG_{d}}\frac{(-1)^{\sum_{k=0}^{d-1}l^T_{k}}\prod^{d-1}_{k=0}1_{\CA_{k}(T)}}{\Omega_{3,K}(T)\Omega_{5,K}(T)\cdots\Omega_{2d+3,K}(T)}(K_1,...,K_{2d+3}).
    \\
    +&\sum_{j=1}^{4}\sum_{T\in\CG_d}\frac{1^{s+\sum^{d-1}_{k=2}l^T_{k}}1_{\CB^{(s)}_{0}(T)}\prod^{d-1}_{k=1}1_{\CB_{k}(T)}}{\Omega_{5,K}(T)^2\Omega_{7,K}(T)\cdots\Omega_{2d+3,K}(T)}(K_1,...,K_{2d+3})d_{K_1}\cdots d_{K_{2d+3}}
\end{split}
\end{equation}

By (\ref{notation:F''}), 
\begin{equation}
\begin{split}
    |S^{d+2}_K(u)|&\le\sum_{\substack{\substack{\mathcal{S}_{2d+3,K} =0}}} |F^{d+1}_{K_1K_2...K_{2d+3}}||b_{K_1}\overline{b_{K_2}}...b_{K_{2d+3}}||D_{2d+3,K}(t)|
    \\
    &\le ||u||_{X^\ell}^2\sum_{T\in\CG_{d}}\sum_{\substack{\mathcal{S}_{2d+3,K} =0\\\Omega_{2d+3,K}\ne0}}\frac{\prod^{d-1}_{k=0}1_{\CA_{k}(T)}}{|\Omega_{3,K}(T)\Omega_{5,K}(T)\cdots\Omega_{2d+3,K}(T)|}(K_1,...,K_{2d+3})|d_{K_1}\overline{d_{K_2}}...d_{K_{2d+3}}|
    \\
    &+||u||_{X^\ell}^2\sum_{T\in\CG_{d}}\sum_{\substack{\mathcal{S}_{2d+3,K}=0\\\Omega_{2d+3,K}\ne0}}\frac{1_{\CB^{(s)}_{0}(T)}\prod^{d-1}_{k=1}1_{\CB_{k}(T)}}{|\Omega_{5,K}(T)^2\Omega_{7,K}(T)\cdots\Omega_{2d+3,K}(T)|}(K_1,...,K_{2d+3})|d_{K_1}\overline{d_{K_2}}\cdots d_{K_{2d+3}}|.
\end{split}
\end{equation}

By the second conclusion in Lemma \ref{l:numbertheory1} and Lemma \ref{l:numbertheory4}, 
\begin{equation}
\begin{split}
    |S^{d+2}_K(u)|\leq|\CG_{d}|||d_{K}||_{X^\ell_{K}}^{2d+5}\frac{L^{2d+2+}}{\langle K\rangle^{\ell}}
    \lesssim ||u||_{X^\ell}^{2d+5}\frac{L^{2d+2+}}{\langle K\rangle^{\ell}}
\end{split}
\end{equation}
where $|\CG_{d}|$ is the number of trees in $\CG_{d}$.
Thus $\|S^{d+2}_K \|_{X^\ell_K} \lesssim L^{2d+2+} \| u \|_{X^\ell}^{2d+5}$. Replacing $d$ by $d-2$, we complete the proof. 
\end{proof}

\subsection{Dynamics of $d_{K}$} Consider the continuous equation 
\begin{equation}\label{eq:cr}
-i\partial_t f(K)=|f(K)|^4f(K),\ \ \ f(K)|_{t=0}=f_0(K), \ \ \ K\in\R.
\end{equation}
We can solve this ODE and solution is $f(K)=e(\frac{1}{2\pi}|f_0(K)|^4t)f_0(K).$ In this section, we compare $f(K)$ with $c_K$ (\ref{eq:transformedmain}) and obtain a long time dynamics description for $d_K$ in Proposition \ref{boundphase}.

\begin{proposition}\label{boundphase}
Fix $\ell>1$, $0<\gamma<1$, and let $f_0\in X^{\ell,2}(\mathbb{R})$. Let $f(t,\xi)$ be the solution of (\ref{eq:cr}) over time interval $[0,M]$ with initial data $f_0$. Let 
$B\overset{def}{=}\|f(0)\|_{X^{\ell,2}(\RR)}.$ Let u be the solution of (\ref{1NLS}) with initial data $u_0 = \frac{1}{L} \sum_{\mathbb{Z}_L} f_0(K) e(Kx)$.
Then for $L$ sufficiently large and  $\epsilon^2L^\gamma$ sufficiently small, depending on $M$, $B$, there exists  a  constant
$C_{\gamma}$ such that or all $t\in [0, MT_R]$,
\begin{align*}
\begin{split}
\norm{d_K-f(\frac{t}{T_R}, K)}_{X^\ell(\mathbb{Z}_L)}\lesssim C_{\gamma}  \Big(\epsilon^2L^\gamma+L^{-1}\Big),
\end{split}
\end{align*}
where $T_R=\frac{L^{2}}{\epsilon^4 }$. 
\end{proposition}

\begin{proof}
We shall use bootstrap argument to show that $\sup_{0\le t\le T}||u_K(t)||_{X^\ell} \le 2B$. To do this, it suffices to show that
\begin{equation}\label{eq:4.46}
    \sup_{0\le t\le T}||u_K(t)||_{X^\ell} \le 2B\Rightarrow \sup_{0\le t\le T}||u_K(t)||_{X^\ell} < 2B 
\end{equation}
for any $T<MT_R$. To see why (\ref{eq:4.46}) implies $\sup_{0\le t\le T}||u_K(t)||_{X^\ell} \le 2B$, notice that initially $||u(0)||_{X^\ell}=||f(0)||_{X^\ell}\le B$, if at some time $t_0\le MT_R$, $||u(t_0)||_{X^\ell}\geq 2B$, by continuity there exists $0<T\leq t_0$ s.t. $||u_K(T)||_{X^\ell}=2B$ and $||u_K(T)||_{X^\ell}=\sup_{0\le t\le T}||u_K(t)||_{X^\ell}$. But by (\ref{eq:4.46}), $||u_K(T)||_{X^\ell}<2B$, which is a contradiction.

Notice that 
\begin{equation}\label{eq:prop1mainreduction}
    \sup_{0\le t\le T}||u_K(t)||_{X^\ell} \leq 2B\Rightarrow\sup_{0\le t \le T} \left\| d_K(t) - f\left(\frac{t}{T_R},K\right)\right\|_{X^\ell} \lesssim C_{\gamma}  \Big(\epsilon^2L^\gamma+L^{-1}\Big).
\end{equation}
implies (\ref{eq:4.46}). So we only need to prove (\ref{eq:prop1mainreduction}). 

In what follows, we sometimes denote $f(\frac{t}{T_R},K)$ by $f(K)$ for simplicity.

From Lemma~\ref{boundNF} we have
\begin{align*}
\begin{split}
\left\| d_K(t) - f\left(\frac{t}{T_R},K\right)\right\|_{X^\ell} \le &  \left\| c_K(t) - f\left(\frac{t}{T_R},K\right)\right\|_{X^\ell}  + \| d_K(t) -  c_K(t) \|_{X^\ell} \\
 \le &  \left\| c_K(t) - f\left(\frac{t}{T_R},K\right)\right\|_{X^\ell} + C_\gamma\left(\sum_{d=1}^P
\epsilon^{2d}   \|u\|^{2d+1}_{X^{\ell}} \right)  L^{\gamma}~\\
\le &  \left\| c_K(t) - f\left(\frac{t}{T_R},K\right)\right\|_{X^\ell} + C_\gamma\left(\sum_{d=1}^P\epsilon^{2d}  B^{2d+1}\right)  L^{\gamma}.
\end{split}
\end{align*}

To bound $w_K := c_K(t) - f(\frac{t}{T_R},K)$, from (\ref{eq:transformedmain}), we have
\begin{align}\label{eq:wk}
\begin{split}
    -i\partial_tw_K&=\underbrace{\frac{\epsilon^{4}}{L^{2}}\big(\frac{1}{L^2}H^2_{K}(u) - |f(K)|^4f(K)\big)}_{I}+\underbrace{\sum_{d=3}^P\frac{\epsilon^{2d}}{L^{2d}}(H^{d}_{K}(u)+G^{d}_{K}(u))}_{II}\\
    &+\underbrace{\sum_{d=3}^{P+1} \frac{\epsilon^{2d}}{L^{2d}} S^{d}_{K}(u)}_{III}+\underbrace{\frac{\epsilon^{2P+2}}{L^{2P+2}}(\widetilde{H}^{P+1}_{K}(u)+\widetilde{G}^{P+1}_{K}(u))}_{IV}\\
    &=I+II+III+IV,
\end{split}
\end{align}
where $I$ is the main term, while $II$, $III$, and $IV$ are errors.

\textbf{Estimate of Main Term $I$:} To estimate the main term, we use the method of symmetrization. Recall that, from (\ref{eq:mainterm}), we have 
\begin{align}
\frac{\epsilon^{4}}{L^{4}}H_K^2(u)&=\frac{3\epsilon^4}{L^4}\sum_{\substack{K_1\ne K}} \frac{|d_{K_1}|^4}{(K_1-K)^2}d_K\\
&=\frac{3\epsilon^4}{L^4}\sum_{\substack{K_1\ne 0}} \frac{|d_{K_1+K}|^4+|d_{K-K_1}|^4-2|d_K|^4}{2K_1^2}d_K+\frac{3\epsilon^4}{L^4}\Big(\sum_{\substack{K_1\ne 0}}\frac{1}{K_1^2}\Big)|d_K|^4d_K.
\end{align}

Therefore,
\begin{align*}
\begin{split}
I&=\frac{3\epsilon^4}{L^2}\left(L^{-2}\Big(\sum_{\substack{K_1\ne 0}}\frac{1}{K_1^2}\Big)|d_K|^4d_K-|f(K)|^4f(K)\right)\\
&+\frac{3\epsilon^4}{L^4}\sum_{\substack{K_1\ne 0}} \frac{|d_{K_1+K}|^4+|d_{K-K_1}|^4-2|d_K|^4}{2K_1^2}d_K\\
&=:I_1+I_2.
\end{split}
\end{align*}
Notice that
\begin{align*}
\begin{split}
I_1=\frac{3\epsilon^4}{L^2}\left(\Big(\sum_{\substack{\mu\ne 0\\ \mu\in\mathbb{Z}\\ \mu\leq L}}\frac{1}{\mu^2}\Big)|d_K|^4d_K-|f(K)|^4f(K)\right),
\end{split}
\end{align*}
so for $I_1$, we have
\begin{align*}
\begin{split}
\norm{I_1}_{X^\ell}\leq \frac{3\epsilon^4}{L^2}B^4\left\| d_K(t) - f\left(\frac{t}{T_R},K\right)\right\|_{X^\ell}.
\end{split}
\end{align*}
For $I_2$, recall that for $g\in X^{\ell,2}$
\begin{align}\label{eq:decay}
\begin{split}
    &\sum_{x\in\mathbb{Z}_L\backslash0} \frac{g(x)+g(-x)-2g(0)}{x^2}
    \\
    = & \sum_{\substack{x\in\mathbb{Z}_L\backslash0\\|x|<1}}\frac{g(x)+g(-x)-2g(0)}{x^2}+ \sum_{\substack{x\in\mathbb{Z}_L\backslash0\\|x|>1}} \frac{g(x)+g(-x)-2g(0)}{x^2}
    \\
    \lesssim & L\norm{g''}_{L^\infty}+\sum_{\substack{x\in\mathbb{Z}_L\backslash0\\|x|>1}} \frac{1}{x^2}\norm{g}_{L^\infty}
     \lesssim L\norm{g}_{X^{0,2}}.
     \end{split}
\end{align}

Therefore,
\begin{align*}
\begin{split}
I_2&=\frac{3\epsilon^4}{L^4}\sum_{\substack{K_1\ne 0}}\Big( \frac{|d_{K_1+K}|^4+|d_{K-K_1}|^4-2|d_K|^4}{2  K_1^2}d_K\\
&-\frac{|f_0(K_1+K)|^4+|f_0(K-K_1)|^4-2|f_0(K)|^4}{K_1^2}f(K)\Big) \\
&+\frac{3\epsilon^4}{L^4}\sum_{\substack{K_1\ne 0}}\frac{|f_0(K_1+K)|^4+|f_0(K-K_1)|^4-2|f_0(K)|^4}{K_1^2}f(K) \\
&\leq \frac{3\epsilon^4}{L^2}B^4\norm{d_K-f(K)}_{X^\ell}+\frac{3\epsilon^4}{L^3}\norm{f}^5_{X^{0,2}}.
\end{split}
\end{align*}
Here in last step, we apply (\ref{eq:decay}).

\textbf{Estimate of Errors:}

\underline{Bound on $II$}  

From Lemma \ref{boundNH} and Lemma \ref{boundNG},
\begin{align*}
\begin{split}
\norm{II}_{X^\ell}\leq \sum_{d=3}^P \frac{\epsilon^{2d}}{L^{2d}} L^{2(d-1)+}\norm{u}_{X^\ell}^{2d+1}=\sum_{d=3}^P \frac{\epsilon^{2d}}{L^{2-}} B^{2d+1}.
\end{split}
\end{align*}

\underline{Bound on $III$} 

From Lemma \ref{boundNS},
\begin{align*}
\begin{split}
\norm{III}_{X^\ell}\leq \sum_{d=3}^P \frac{\epsilon^{2d}}{L^{2d}} L^{2(d-1)+}\norm{u}_{X^\ell}^{2d+1}=\sum_{d=3}^P \frac{\epsilon^{2d}}{L^{2-}} B^{2d+1}.
\end{split}
\end{align*}

\underline{Bound on $IV$} 

From Lemma \ref{boundNtildeH} and Lemma \ref{boundNtildeG},
\begin{align*}
\begin{split}
\norm{IV}_{X^\ell}\leq  \frac{\epsilon^{2P+2}}{L^{2P+2}} L^{2P+2+}B^{2P+3}=\epsilon^{2P+2}L^+B^{2P+3}.
\end{split}
\end{align*}

Integrating (\ref{eq:wk}), we have
\begin{align*}
\begin{split}
&\left\| d_K(t) -f\left(\frac{t}{T_R},K\right)\right\|_{X^\ell} -C_{\gamma,B}\epsilon^{2}  L^{\gamma}\\
\le & \int_0^t \Bigg(\frac{3\epsilon^4}{L^2}B^4\left\| d_K(t) - f\left(\frac{t}{T_R},K\right)\right\|_{X^\ell}+
\frac{3\epsilon^4}{L^2}B^4\norm{d_K-f(K)}_{X^\ell}+C_{\gamma,B}\frac{3\epsilon^4}{L^3} \\
&+C_{\gamma,B}\frac{\epsilon^{6}L^+}{L^{2}}+ C_{\gamma,B}\frac{\epsilon^{6}L^+}{L^{2}} +C_{\gamma,B}\epsilon^{2P+2}L^+ \Bigg)ds
\end{split}
\end{align*}
From Gronwall's inequality, and $0\leq t\leq T_RM$, we obtain,

\begin{equation}
\begin{split}
    &\left\| d_K(t) -f\left(\frac{t}{T_R},K\right)\right\|_{X^\ell}
    \\
    \leq& C_{\gamma,B}\Big(\epsilon^{2}  L^{\gamma}+\frac{M}{L}+M\epsilon^{2}  L^{\gamma}+M\epsilon^{2P+2}L^{2+\gamma} \Big)e^{CB^4M}
    \\
    \lesssim& C_{\gamma}  \Big(\epsilon^2L^\gamma+L^{-1}\Big)
\end{split}
\end{equation}

Therefore, $\norm{d_K(t)}_{X^\ell}\leq 2B$ and this completes the proof.

\end{proof}

\subsection{A Refined Estimate of $|d_{K}|^2$}

In previous subsection, we have described the dynamics of $d_K$, when $t<T_R$. We know that $d_K$ is different from $u_K$ by a phase factor that depends on the module of the Fourier coefficients $|d_K|$. From Proposition \ref{boundphase}, we have $\norm{|d_K|-|f(K)|}_{X^\ell}\lesssim \epsilon^2L^++L^{-1}$. Given this estimate we know that 
$$\frac{\epsilon^2}{2\pi L^2}\left|\int_0^t |d_K|^2(s)-|f(K)|^2(s) ds\right|\lesssim \frac{\epsilon^2}{2\pi L^2}t(\epsilon^2L^++L^{-1}).$$ 
When $t\sim \frac{L^2}{\epsilon^4}$, this roughly gives a bound $L^++L^{-1}\epsilon^{-2}$. This bound do not allow us to replace $e\left(\frac{\epsilon^2}{2\pi L^2}\int_0^t |d_K|^2(s)ds\right)$ by $e\left(\frac{\epsilon^2}{2\pi L^2}\int_0^t |f(K)|^2(s)ds\right)$. Only a refined estimate like $\norm{|d_K|^2-(\textit{known})}_{X^\ell}\le \epsilon^4L^++\frac{\epsilon^2L^+}{L}$ allows us to estimate $u_K$ in terms of $d_K$, which still requires some nontrivial work. 

In this subsection, we prove the following proposition.
\begin{proposition}\label{prop2}
We have the following estimate for $|c_K|^2$ and $|d_K|^2$
\begin{equation}
\norm{|c_K|^2(t) -  P_{K}(f)(t)}_{X^\ell}\lesssim \epsilon^4L^++\frac{\epsilon^2L^{+}}{L}.
\end{equation}
with
\begin{equation}
\begin{split}
P_K(f)=&- \frac{2\epsilon^6}{L^6}\int_{0}^t\Im\Big(\left( H^3_K(f)  +  G^3_K(f)+S^3_K(f)\right)\overline{f(K)}+H_K^2(f)F_K^1(\overline{f}) \\
+& H_K^2(f)F_K^1(\overline{f})  +  H_K^2(f)E_K^1(\overline{f})\Big)(s) ds
\end{split}
\end{equation}
and

\begin{equation}
\norm{|d_K|^2(t) -  Q_{K}(f)(t)}_{X^\ell}\lesssim \epsilon^4L^++\frac{\epsilon^2L^{+}}{L}.
\end{equation}
with
\begin{equation}
\begin{split}
Q_K(f)=&\frac{2\epsilon^{2}}{L^{2}}\Re\left( F^{d}_{K}(f)\overline{f(K)}\right)- \frac{2\epsilon^6}{L^6}\int_{0}^t\Im\Big(\left( H^3_K(f)  + G^3_K(f)+S^3_K(f)\right)\overline{f(K)}
\\
+&H_K^2(f)F_K^1(\overline{f}) + H_K^2(f)F_K^1(\overline{f})  +  H_K^2(f)E_K^1(\overline{f})\Big)(s) ds
\end{split}
\end{equation}

\end{proposition}

\begin{proof}
Our final goal is to to estimate $|d_K(t)|^2$ to obtain the dynamics of $u_K$. Our strategy here is to first estimate $|c_K(t)|^2$  using the equation for $c_K$. Then estimate the difference between $|c_K(t)|^2$ and $|d_K(t)|^2$. For $c_K(t)$, recall we have
\begin{equation*}
    -i\partial_tc_K=\frac{\epsilon^{4}}{L^{4}}H_K^2(u)+\sum_{d=3}^P\frac{\epsilon^{2d}}{L^{2d}}(H^{d}_{K}(u)+G^{d}_{K}(u))+\sum_{d=3}^{P+1} \frac{\epsilon^{2d}}{L^{2d}} S^{d}_{K}(u)+\frac{\epsilon^{2P+2}}{L^{2P+2}}(\widetilde{H}^{P+1}_{K}(u)+\widetilde{G}^{P+1}_{K}(u)).
\end{equation*}
Therefore,
\begin{align*}
    \partial_t|c_K|^2&=2\Re(\partial_tc_K \overline{c}_K)\\
    &=-2\Im\Big(\frac{\epsilon^{4}}{L^{4}}H_K^2(u)\overline{c}_K+\sum_{d=3}^P\frac{\epsilon^{2d}}{L^{2d}}(H^{d}_{K}(u)+G^{d}_{K}(u))\overline{c}_K\\
    &+\sum_{d=3}^{P+1} \frac{\epsilon^{2d}}{L^{2d}} S^{d}_{K}(u)\overline{c}_K+\frac{\epsilon^{2P+2}}{L^{2P+2}}(\widetilde{H}^{P+1}_{K}(u)+\widetilde{G}^{P+1}_{K}(u))\overline{c}_K\Big).
\end{align*}
Since
$$ \overline{c}_K=\overline{d}_K-\sum_{d=1}^{P}\frac{\epsilon^{2d}}{L^{2d}}F^{d}_{K}(\overline{u})-\sum_{d=1}^{P}\frac{\epsilon^{2d}}{L^{2d}}E^{d}_{K}(\bar{u}),$$
we have 
\begin{align}\label{eq:cK2}
\begin{split}
    \partial_t|c_K|^2&=-2\Im\Bigg(\frac{\epsilon^{4}}{L^{4}}H_K^2(u)\Big(\overline{d}_K-\sum_{d=1}^{P}\frac{\epsilon^{2d}}{L^{2d}}F^{d}_{K}(\overline{u})-\sum_{d=1}^{P}\frac{\epsilon^{2d}}{L^{2d}}E^{d}_{K}(\bar{u})\Big)\\
    &+\sum_{d=3}^P\frac{\epsilon^{2d}}{L^{2d}}(H^{d}_{K}(u)+G^{d}_{K}(u))\Big(\overline{d}_K-\sum_{d=1}^{P}\frac{\epsilon^{2d}}{L^{2d}}F^{d}_{K}(\overline{u})-\sum_{d=1}^{P}\frac{\epsilon^{2d}}{L^{2d}}E^{d}_{K}(\bar{u})\Big)\\
    &+\sum_{d=3}^{P+1} \frac{\epsilon^{2d}}{L^{2d}} S^{d}_{K}(u)\Big(\overline{d}_K-\sum_{d=1}^{P}\frac{\epsilon^{2d}}{L^{2d}}F^{d}_{K}(\overline{u})-\sum_{d=1}^{P}\frac{\epsilon^{2d}}{L^{2d}}E^{d}_{K}(\bar{u})\Big)\\
    &+\frac{\epsilon^{2P+2}}{L^{2P+2}}(\widetilde{H}^{P+1}_{K}(u)+\widetilde{G}^{P+1}_{K}(u))\Big(\overline{d}_K-\sum_{d=1}^{P}\frac{\epsilon^{2d}}{L^{2d}}F^{d}_{K}(\overline{u})-\sum_{d=1}^{P}\frac{\epsilon^{2d}}{L^{2d}}E^{d}_{K}(\bar{u})\Big)\Bigg)\\
    &=-2\Im\left(\frac{\epsilon^{4}}{L^{4}}H_K^2(u)\overline{d}_K\right)    +2\Im \left(\frac{\epsilon^6}{L^6}\big(H^3_K(u)+G^3_K(u)+S^3_K(u)\big)\overline{d}_K\right)\\
    &+2\Im\left(\frac{\epsilon^{6}}{L^{6}}H_K^2(u)F_K^1(\overline{u})\right)     +2\Im\left(\frac{\epsilon^{6}}{L^{6}}H_K^2(u)E_K^1(\overline{u})   \right)  +\text{Higher Order Terms}\\
&=:\mathcal{N}_1+\mathcal{N}_2+\mathcal{N}_3+\mathcal{N}_4+\mathcal{N}_5.
\end{split}
\end{align}

Here we collect all terms with the factor of order lower than $\frac{\epsilon^6}{L^6}$ on the right hand side and denote them by $\mathcal{N}_1$, $\mathcal{N}_2$, $\mathcal{N}_3$, $\mathcal{N}_4$. The collection of all other terms is called higher order terms and is denoted by $\mathcal{N}_5$. Note that our final goal is $\norm{|c_K|^2-(\textit{known})}_{X^\ell}\lesssim \epsilon^4L^++\frac{\epsilon^2L^+}{L}$ for $t<\frac{L^2}{\epsilon^4}$. All terms of order greater than $\frac{\epsilon^8}{L^8}$ still have $\epsilon^4$ in front after integration over $t$ and thus can be treated as error. So we collect  all terms of order lower than $\frac{\epsilon^6}{L^6}$ for which further estimate should be done.

Notice that $\mathcal{N}_1=0$. This is because by (\ref{eq:mainterm}), we know that 
\begin{align}
\frac{\epsilon^{4}}{L^{4}}H_K^2(u)=\frac{3\epsilon^4}{L^4}\sum_{\substack{K_1\ne K}} \frac{|d_{K_1}|^4}{(K_1-K)^2}d_K.
\end{align}
Thus 
\begin{align*}
-2\Im\frac{\epsilon^{4}}{L^{4}}H_K^2(u)\overline{d}_K=-2\Im\frac{3\epsilon^4}{L^4}\sum_{\substack{K_1\ne K}} \frac{|d_{K_1}|^4}{(K_1-K)^2}|d_K|^2=0
\end{align*}
From Proposition \ref{boundphase}, we have
$$\sup_{0\leq t\leq MT_R}\Big||d_K(t)|-|f(\frac{t}{T_R},K)|\Big|\lesssim \epsilon^2L^\gamma+\frac{1}{L}.$$
To treat all other $\mathcal{N}_{j}$, $j=2,3,4$., we first introduce the following lemma.

\begin{lemma}\label{l:multilinearform}
For any multilinear form $F(x_1,\cdots,x_m):$ $X^\ell\times\cdots\times X^\ell\rightarrow X^\ell$ with norm $\norm{F}$, i.e.  $$\norm{F(x_1,\cdots,x_m)}_{X^\ell}\le \norm{F}\ \norm{x_1}_{X^\ell}\cdots\norm{x_m}_{X^\ell}.$$

If $\norm{d_{K_j}}_{X^\ell}\le B$ and $\norm{f(K)}_{X^\ell}\le B$, we have
\begin{equation}
\begin{split}
&\norm{F(d_{K_1},\cdots,d_{K_m})-F(f(K_1),\cdots,f(K_m))}_{X^\ell}\\
&\leq CB^{m-1}\norm{F}\norm{d_K-f(K)}_{X^\ell}\leq CB^{m-1}\norm{F}(\epsilon^2L^++\frac{1}{L}).
\end{split}
\end{equation}
\end{lemma}
\begin{remark}
By the polar identity, 
\begin{equation}
\begin{split}
&\norm{F(x_1,\cdots,x_m)}_{X^\ell}\le \norm{F}\ \norm{x_1}_{X^\ell}\cdots\norm{x_m}_{X^\ell},\ \forall x_1,\cdots, x_m.\\
&\Leftrightarrow \norm{F(x,\cdots,x)}_{X^\ell}\le \norm{F}\ \norm{x}_{X^\ell}\cdots\norm{x}_{X^\ell}, \ \forall x.
\end{split}
\end{equation}
Thus the results in Lemma \ref{boundNH}, \ref{boundNF}, \ref{boundNtildeH}, \ref{boundNG}, \ref{boundNE}, \ref{boundNtildeG}, \ref{boundNS} can be viewed as bounds on the norm of the multilinear forms $H^d_K(u)$, $F^d_K(u)$, $\widetilde{H}^d_K(u)$, $G^d_K(u)$ and $\widetilde{G}^d_K(u)$ and $S^d_K(u)$.
\end{remark}
\begin{proof}
\begin{align*}
&\norm{F(d_{K_1},\cdots,d_{K_m})-F(f(K_1),\cdots,f(K_m))}_{X^\ell}
\\
=&\norm{\int_0^1 \frac{d}{dt}F(td_{K_1}+(1-t)f(K_1),\cdots,td_{K_m}+(1-t)f(K_m)) dt}_{X^\ell}
\\
=&\int_0^1 \norm{\frac{d}{dt}F(td_{K_1}+(1-t)f(K_1),\cdots,td_{K_m}+(1-t)f(K_m))}_{X^\ell} dt
\\
=&\int_0^1 \sum_{j=1}^m \norm{\partial_{x_j} F(td_{K_1}+(1-t)f(K_1),\cdots,td_{K_m}+(1-t)f(K_m))  (d_{K_j}-f(K_j)}_{X^\ell} dt
\\
\lesssim&C(\norm{d_K}_{X^\ell}^{m-1}+\norm{f(K)}_{X^\ell}^{m-1})\norm{F}\norm{d_K-f(K)}_{X^\ell}
\\
\lesssim&CB^{m-1}\norm{F}\norm{d_K-f(K)}_{X^\ell}
\end{align*}
\end{proof}

Basically this lemma tells us that we may replace $d_K$ by $f(K)$ up to errors. For sake of simplicity, we introduce $H^d_K(f)$, $G^d_K(f)$ and $S^d_K(f)$, which is obtained simply by replacing all $d_{K_j}$ in $H^d_K(u)$, $G^d_K(u)$ and $S^d_K(u)$ by $f(K)$.

\underline{Bound on $\mathcal{N}_2$:} 
By Lemma \ref{boundNH}, \ref{boundNG} and \ref{boundNS}, we know that the norms of $H^3_K(u)$, $G^3_K(u)$ and $S^3_K(u)$ can be bounded by $L^{4+}$. Thus replace $d_K$ by $f(K)$ using Lemma \ref{l:multilinearform} we get 
\begin{align*}
\norm{\mathcal{N}_2- 2\left(\Im \frac{\epsilon^6}{L^6}\big(H^3_K(f)+G^3_K(f)+S^3_K(f)\big)\right)\overline{f(K)}}_{X^\ell}   &\lesssim \frac{\epsilon^6}{ L^{6}}L^{4+}(\epsilon^2L^++\frac{1}{L})\\
&\lesssim \frac{\epsilon^4}{ L^{2}}(\epsilon^4L^++\frac{\epsilon^2L^{+}}{L}).
\end{align*}

\underline{Bound on $\mathcal{N}_3$:}  
By (\ref{eq:mainterm}), Lemma \ref{boundNF}, we know that the norms of $H^2_K(u)$, $F^1_K(u)$ can be bounded by $L^{2+}$. Thus replace $d_K$ by $f(K)$ using Lemma \ref{l:multilinearform} we get
\begin{align*}
&\norm{\mathcal{N}_4-2\Im\left(\frac{\epsilon^{6}}{L^{6}}H_K^2(f)F_K^1(\overline{f})\right)}_{X^\ell} 
\\
=&\norm{2\Im\left(\frac{\epsilon^{6}}{L^{6}}H_K^2(u)F_K^1(\overline{u})\right)  -  2\Im\left(\frac{\epsilon^{6}}{L^{6}}H_K^2(f)F_K^1(\overline{f})\right)}_{X^\ell} 
\\
=&\norm{2\Im\left(\frac{\epsilon^{6}}{L^{6}}(H_K^2(u)  -  H_K^2(f))F_K^1(\overline{u})\right)  -  2\Im\left(\frac{\epsilon^{6}}{L^{6}}H_K^2(f) (F_K^1(\overline{f}-F_K^1(\overline{u}))\right)}_{X^\ell} 
\\
\lesssim& \frac{\epsilon^6}{ L^{6}}L^{4+}(\epsilon^2L^{+}+\frac{1}{L})
\\
\lesssim& \frac{\epsilon^4}{ L^{2}}(\epsilon^4L^++\frac{\epsilon^2L^{+}}{L}).
\end{align*}

\underline{Bound on $\mathcal{N}_4$:}  
By (\ref{eq:mainterm}), Lemma \ref{boundNE}, we know that the norms of $H^2_K(u)$, $E^1_K(u)$ can be bounded by $L^{2+}$. Thus replace $d_K$ by $f(K)$ using Lemma \ref{l:multilinearform} we get
\begin{align*}
&\norm{\mathcal{N}_3-2\Im\left(\frac{\epsilon^{6}}{L^{6}}H_K^2(f)E_K^1(\overline{f})\right)}_{X^\ell} 
\\
=&\norm{2\Im\left(\frac{\epsilon^{6}}{L^{6}}H_K^2(u)E_K^1(\overline{u})\right)  -  2\Im\left(\frac{\epsilon^{6}}{L^{6}}H_K^2(f)E_K^1(\overline{f})\right)}_{X^\ell} 
\\
=&\norm{2\Im\left(\frac{\epsilon^{6}}{L^{6}}(H_K^2(u)  -  H_K^2(f))E_K^1(\overline{u})\right)  -  2\Im\left(\frac{\epsilon^{6}}{L^{6}}H_K^2(f) (E_K^1(\overline{f}-E_K^1(\overline{u}))\right)}_{X^\ell} 
\\
\lesssim& \frac{\epsilon^6}{ L^{6}}L^{4+}(\epsilon^2L^{+}+\frac{1}{L})
\\
\lesssim& \frac{\epsilon^4}{ L^{2}}(\epsilon^4L^++\frac{\epsilon^2L^{+}}{L}).
\end{align*}

\underline{Bound on $\mathcal{N}_5$:}
From Lemma \ref{boundNH}, Lemma \ref{boundNF}, Lemma \ref{boundNtildeH}, Lemma \ref{boundNG}, Lemma \ref{boundNE}, Lemma \ref{boundNtildeG}, Lemma \ref{boundNS}, it's not very difficult to prove that the $X^\ell$ norm of  the sum of all other terms should be of the order $\frac{\epsilon^4L^+}{L^2} (\epsilon^4L^{+})$.

From above calculations we know that
\begin{equation}
\begin{split}
\partial_t |c_K|^2=&- 2\left(\Im \frac{\epsilon^6}{L^6}\big(H^3_K(f)  +  G^3_K(f)+S^3_K(f)\big)\overline{f(K)}\right) -  2\Im\left(\frac{\epsilon^{6}}{L^{6}}H_K^2(f)F_K^1(\overline{f})\right)  
\\
-&  2\Im\left(\frac{\epsilon^{6}}{L^{6}}H_K^2(f)F_K^1(\overline{f})\right)  -  2\Im\left(\frac{\epsilon^{6}}{L^{6}}H_K^2(f)E_K^1(\overline{f})\right)+O(\frac{\epsilon^4}{ L^{2}}(\epsilon^4L^++\frac{\epsilon^2L^{+}}{L}))
\end{split}
\end{equation}

Integrate in $t$, and note that because $t\lesssim\frac{L^2}{\epsilon^4}$, $\int^t_0 O(\frac{\epsilon^4}{ L^{2}}(\epsilon^4L^++\frac{\epsilon^2L^{+}}{L}))ds\le O(\epsilon^4L^++\frac{\epsilon^2L^{+}}{L})$. We have

\begin{equation}
\begin{split}
|c_K|^2(t)=&- \frac{2\epsilon^6}{L^6}\int_{0}^t\Im\Big(\left( H^3_K(f)  +  G^3_K(f)+S^3_K(f)\right)\overline{f(K)}+H_K^2(f)F_K^1(\overline{f}) 
\\
+& H_K^2(f)F_K^1(\overline{f})  +  H_K^2(f)E_K^1(\overline{f})\Big)(s) ds+O(\epsilon^4L^++\frac{\epsilon^2L^{+}}{L})
\end{split}
\end{equation}

Define 
\begin{equation}
\begin{split}
P_K(f)=&- \frac{2\epsilon^6}{L^6}\int_{0}^t\Im\Big(\left( H^3_K(f)  +  G^3_K(f)+S^3_K(f)\right)\overline{f(K)}+H_K^2(f)F_K^1(\overline{f}) \\
+& H_K^2(f)F_K^1(\overline{f})  +  H_K^2(f)E_K^1(\overline{f})\Big)(s) ds
\end{split}
\end{equation}

Then we know that 

\begin{equation}
\norm{|c_K|^2(t) -  P_{K}(f)(t)}_{X^\ell}\lesssim \epsilon^4L^++\frac{\epsilon^2L^{+}}{L}.
\end{equation}
We complete the prove of the first conclusion.

Now let us started to prove the second conclusion. Let us consider the difference between $|c_K|^2$ and $|d_K|^2$. By (\ref{eq:recurrencec^d_K}), we have 

\begin{equation}
    c_K=d_K-\sum_{d=1}^{P}\frac{\epsilon^{2d}}{L^{2d}}F^{d}_{K}(u)-\sum_{d=1}^{P}\frac{\epsilon^{2d}}{L^{2d}}E^{d}_{K}(u),
\end{equation}

So 
\begin{equation}
\begin{split}
&|c_K|^2-|d_K|^2
\\
=&-2\sum_{d=1}^P\frac{\epsilon^{2d}}{L^{2d}}\Re\left(F^{d}_{K}(u)+E^{d}_{K}(u)\right)\overline{d_K}+\left|\sum_{d=1}^{P}\frac{\epsilon^{2d}}{L^{2d}}(F^{d}_{K}(u)+F^{d}_{K}(u))\right|^2
\\
=&-\frac{2\epsilon^{2}}{L^{2}}\Re\left(F^{d}_{K}(u)+E^{d}_{K}(u)\right)\overline{d_K}+\text{Higher Order Term}
\\
=&-\frac{2\epsilon^{2}}{L^{2}}\Re\left( F^{d}_{K}(u)\overline{d_K}\right)+\text{Higher Order Term}
\end{split}
\end{equation}
Here the last step follows from the fact that $G^1_{K_1K_2K_3}=0$. Here we collect all terms with the factor of order lower than $\frac{\epsilon^6}{L^6}$ on the right hand side. The collection of all other terms is called higher order terms. Note that our final goal is $\norm{|d_K|^2-(\textit{known})}_{X^\ell}\le \epsilon^4L^++\frac{\epsilon^2L^+}{L}$ for $t<\frac{L^2}{\epsilon^4}$. All terms of order greater than $\frac{\epsilon^4}{L^4}$ have $\epsilon^4$ in front and thus can be treated as error. So we collect  all terms of order lower than $\frac{\epsilon^2}{L^2}$ for which further estimate should be done.

The first term can be treated in the same way as that of $\mathcal{N}_2$. We first notice that the by Lemma \ref{boundNF}, we know that the norms of $F^1_K(u)$ can be bounded by $L^{2+}$. Thus replace $d_K$ by $f(K)$ using Lemma \ref{l:multilinearform} we get 
\begin{align*}
&\norm{\text{first term}  +\frac{2\epsilon^{2}}{L^{2}}\Re\left( F^{d}_{K}(f)\overline{f(K)}\right)}_{X^\ell} 
\\  
=&\frac{2\epsilon^{2}}{L^{2}}\norm{\Re\left( F^{d}_{K}(u)\overline{d_K}\right)-\Re\left( F^{d}_{K}(f)\overline{f(K)}\right)}_{X^\ell} 
\\
\lesssim& \frac{\epsilon^2}{ L^{2}}L^{2+}(\epsilon^2L^++\frac{1}{L})
\\
\lesssim& \epsilon^4L^++\frac{\epsilon^2L^{+}}{L}.
\end{align*}

From Lemma \ref{boundNH}, Lemma \ref{boundNF}, Lemma \ref{boundNtildeH}, Lemma \ref{boundNG}, Lemma \ref{boundNE}, Lemma \ref{boundNtildeG}, Lemma \ref{boundNS}, it's easy to verify that the $X^\ell$ norm of the sum of all other terms should be bounded by $\epsilon^4L^+$.

Thus 
\begin{align*}
&\norm{|d_K|^2  - |c_K|^2 -\frac{2\epsilon^{2}}{L^{2}}\Re\left( F^{d}_{K}(f)\overline{f(K)}\right)}_{X^\ell}
\\
\lesssim& \epsilon^4L^++\frac{\epsilon^2L^{+}}{L}.
\end{align*}
Now apply the first conclusion and this completes the proof of the second conclusion.

\end{proof}

\subsection{Wrapping up} In this section, we shall prove Theorem \ref{th:main} which states that

\begin{equation}
\begin{split}
    &\norm{u_{K}(t)-f_0(K)e\left(\Big(K^2+\frac{\epsilon^2}{\pi L}\norm{u}_{L^2}^2+\frac{\epsilon^2}{\pi L^2}|f_0(K)|^4\Big)t-\frac{\epsilon^2}{2\pi L^2}\int^t_{0}Q_{K}(f)(s) ds\right)}_{X^\ell(\mathbb{Z}_{L}^n)}
    \\
    &\lesssim C_{\gamma}(\epsilon^2 L^\gamma+\frac{1}{L^{1-\gamma}}).
\end{split}
\end{equation}

From Proposition \ref{boundphase}, we know that in $[0,MT_R]$,
\begin{equation}
\begin{split}
    &\norm{u_{K}(t)-f_0(K)e\left(\Big(K^2+\frac{\epsilon^2}{\pi L}\norm{u}_{L^2}^2+\frac{\epsilon^2}{\pi L^2}|f_0(K)|^4\Big)t-\frac{\epsilon^2}{2\pi L^2}\int^t_{0}|d_{K}(s)|^2 ds\right)}_{X^\ell(\mathbb{Z}_{L}^n)}
    \\
    &\lesssim C_{\gamma}(\epsilon^2 L^\gamma+\frac{1}{L^{1-\gamma}})
\end{split}
\end{equation}
and
\begin{equation}
\norm{u_{K}(t)}_{X^\ell(\mathbb{Z}_{L}^n)}\le B.
\end{equation}

Thus by Proposition \ref{prop2}, we know that 
\begin{align*}
    &\norm{u_{K}(t)-f_0(K)e\left(\Big(K^2+\frac{\epsilon^2}{\pi L}\norm{u}_{L^2}^2+\frac{\epsilon^2}{\pi L^2}|f_0(K)|^4\Big)t-\frac{\epsilon^2}{2\pi L^2}\int^t_{0}Q_{K}(f)(s) ds\right)}_{X^\ell(\mathbb{Z}_{L}^n)}
    \\
    \le& \norm{f_0(K)e\left((\cdots)t-\frac{\epsilon^2}{2\pi L^2}\int^t_{0}|d_{K}(s)|^2 ds\right)-f_0(K)e\left((\cdots)t-\frac{\epsilon^2}{2\pi L^2}\int^t_{0}Q_{K}(f)(s) ds\right)}_{X^\ell(\mathbb{Z}_{L}^n)}\\
    +&\norm{u_{K}(t)-f_0(K)e\left((\cdots)t-\frac{\epsilon^2}{2\pi L^2}\int^t_{0}|d_{K}(s)|^2 ds\right)}_{X^\ell(\mathbb{Z}_{L}^n)}\\
    \lesssim&\norm{f_0(K)\left(  e\left(\frac{\epsilon^2}{2\pi L^2}\int^t_{0}(Q_{K}(f)(s)-|d_{K}(s)|^2) ds\right) -1\right)}_{X^\ell(\mathbb{Z}_{L}^n)}+\epsilon^2 L^++L^{-1+}
    \\
    \lesssim& \frac{\epsilon^2}{L^2}\int^t_{0}\norm{Q_{K}(f)(s)-|d_{K}(s)|^2}_{X^\ell(\mathbb{Z}_{L}^n)} ds+\epsilon^2 L^++L^{-1+}
    \\
    \lesssim&  \frac{\epsilon^2}{L^2}T_R(\epsilon^4L^++\frac{\epsilon^2L^{+}}{L}) +\epsilon^2 L^++L^{-1+}
    \\
    =&\epsilon^2 L^++L^{-1+}.
\end{align*}

The third step follows from $|e(x)-1|\le |x|$. This completes the proof of Theorem \ref{th:main}.


\begin{thebibliography}{10}

\bibitem{BIT}
A.~Babin, A.~A.~Ilyin, and E.~S.~Titi.
\newblock On the regularization mechanism for the periodic Korteweg-de Vries equation
\newblock {\em Comm. Pure Appl. Math.}, 64 (2011), no. 5, 591648.



\bibitem{Bourgain1993}
J.~Bourgain.
\newblock Fourier transform restriction phenomena for certain lattice subsets
  and applications to nonlinear evolution equations. {{I}}. {{Schr\"odinger}}
  equations.
\newblock {\em Geometric and Functional Analysis}, 3(2):107--156, 1993.

\bibitem{Bourgain1994}
J.~Bourgain. 
\newblock Periodic nonlinear Schr\"odinger equation and invariant measures, 
\newblock{\em Comm. Math. Phys.} 166 (1994), 1-26.

\bibitem{Bourgain1994'}
J.~Bourgain.
\newblock Approximation of Solutions of the Cubic Nonlinear Schrodinger Equations by Finite-Dimensional
Equations and Nonsqueezing Properties.
\newblock {\em Int. Math. Res. Not. IMRN}, 1994, no. 2, 79--90.


\bibitem{Bourgain2004}
J.~Bourgain.
\newblock A remark on normal forms and the “I-method” for periodic NLS.
\newblock {\em J. Anal. Math.} (2004) 94: 125.


\bibitem{BR}
J.~Br\"{u}dern, O.~Robert.
\newblock Rational points on linear slices of diagonal hypersurfaces.
\newblock {\em Nagoya Math. J.}, Volume 218, 51--100, 2015.

\bibitem{BGHS1}
T.~Buckmaster, P.~Germain, Z.~Hani, and J.~Shatah.
\newblock Effective dynamics of the nonlinear {S}chr\"{o}dinger equation on
  large domains.
\newblock {\em Comm. Pure Appl. Math.}, 71(7):1407--1460, 2018.

\bibitem{BGHS2}
T.~Buckmaster, P.~Germain, Z.~Hani, and J.~Shatah.
\newblock Analysis of the {(CR)} equation in higher dimensions.
\newblock {\em International Mathematics Research Notices Accepted}, 2017.

\bibitem{BGHS3} 
T. Buckmaster, P. Germain, Z. Hani, and J. Shatah.
\newblock Onset of the wave turbulence description of the longtime behavior of the nonlinear Schr\"odinger equation. 
{\em arXiv:1907.03667}.

\bibitem{Christ}
M.~Christ.
\newblock Power series solution of a nonlinear {S}chr{\"o}dinger equation in {\em Mathematical Aspects of Nonlinear Dispersive Equations}.
\newblock 131–155, {\em Ann. of
Math. Stud.}, 163, Princeton Univ. Press, Princeton, NJ, 2007.

\bibitem{CKSTT'}
J. Colliander, M. Keel, G. Staffilani, H. Takaoka, T. Tao
\newblock Global well-posedness and scattering for the energy-critical nonlinear Schr\"{o}dinger equation in $\mathbb{R}^3$
\newblock {\em Ann. of Math.}, (2) 167 (2008), 767 – 865.




\bibitem{DZ}
P.~Deift, X.~Zhou.
\newblock Long‐time asymptotics for solutions of the NLS equation with initial data in a weighted Sobolev space.
\newblock {\em Comm. Pure Appl. Math.}, 56: 1029--1077. doi:10.1002/cpa.3034, 2003.


\bibitem{ET}
M.~B.~Erdo\v{g}an, N.~Tzirakis.
\newblock Talbot effect for the cubic nonlinear Schr{\"o}dinger equation on the torus.
\newblock {\em Math. Res. Lett.}, 20 (2013), 1081--1090.



\bibitem{FGH}
E.~Faou, P.~Germain, and H.~Hani.
\newblock The weakly nonlinear large-box limit of the {{2D}} cubic nonlinear
  {{Schr{\"o}dinger}} equation.
\newblock {\em Journal of the American Mathematical Society}, 2015.

\bibitem{GKO}
Z.~Guo, S.~Kwon, and T.~Oh.
\newblock Poincar\'e-Dulac normal form reduction for unconditional well-posedness of the periodic cubic NLS.
\newblock {\em Comm. Math. Phys.}, 322(1):19--48, 2013.

\bibitem{GMS}
P. ~Germain, N. ~Masmoudi, J. ~Shatah.
\newblock Global solutions for 3D quadratic Schr{\"o}dinger equations. 
\newblock {\em Int. Math. Res. Not. IMRN}, 2009, no. 3, 414–432.


\bibitem{Hani}
Z.~Hani.
\newblock Long-time instability and unbounded {S}obolev orbits for some
  periodic nonlinear {S}chr\"odinger equations.
\newblock {\em Arch. Ration. Mech. Anal.}, 211(3):929--964, 2014.



\bibitem{HB}
D.~R. Heath-Brown.
\newblock A new form of the circle method, and its application to quadratic
  forms.
\newblock {\em J. Reine Angew. Math.}, 481:149--206, 1996.



\bibitem{IoPa1}
A.~D. Ionescu and B.~Pausader.
\newblock Nonlinear fractional {S}chr{\"o}dinger equations in one dimension.
\newblock {\em J. Funct. Anal.}, 266, 139--176 (2014).


\bibitem{IoPa2}
A.~D.~Ionescu, F.~Pusateri.
\newblock Global regularity for 2D water waves with surface tension.
\newblock {\em Memoirs of the American Mathematical Society}, 256, Memo 1227 (2018).



\bibitem{KST}
T.~Kappeler, B.~Schaad, and P.~Topalov.
\newblock Scattering-like phenomena of the
periodic defocusing NLS equation.
\newblock {\em Math. Res. Lett.}, 24(3), 1081--1090, 2015.


\bibitem{KP}
J.~Kato, F.~Pusateri.
\newblock A new proof of long-range scattering for critical nonlinear {S}chr{\"o}dinger equations.
\newblock {\em Differ. Int. Equ.}, 24(9-10), 923--940 (2011).

\bibitem{LS}
J.~Lukkarinen and H.~Spohn.
\newblock Weakly nonlinear {S}chr\"odinger equation with random initial data.
\newblock {\em Invent. Math.}, 183(1):79--188, 2011.

\bibitem{Nazarenko}
S.~Nazarenko.
\newblock {\em Wave turbulence}, volume 825 of {\em Lecture Notes in Physics}.
\newblock Springer, Heidelberg, 2011.

\bibitem{NMPZ1984}
S.~Novikov, S.V.~Manakov, L.P.~Pitaevskii, V.E.~Zakharov 
\newblock{\em Theory of Solitons: The Inverse Scattering Method}, Monographs in Contemporary Mathematics, Springer-Verlag, 1984.


\bibitem{OW}
T.~Oh, Y.~Wang.
\newblock Normal form approach to the one-dimensional periodic cubic nonlinear Schr{\"o}dinger equation in almost critical Fourier-Lebesgue spaces.
To appear in  {\em J. Anal. Math.}

\bibitem{numbertheorybook}
M.~Overholt
\newblock {\em A Course in Analytic Number Theory}.
\newblock Graduate Studies in Mathematics, Vol. 160, American Mathematical Society, Providence, RI, 2014.

\bibitem{ProcesiProcesi}
C.~Procesi and M.~Procesi.
\newblock A kam algorithm for the nonlinear schr\"odinger equation.
\newblock {\em Advances in Math.}, 272:399--470., 2015.

\bibitem{Shatah1}
J.~Shatah. 
\newblock Normal forms and quadratic nonlinear Klein–Gordon equations. 
\newblock {\em Comm. Pure Appl. Math.}, 38, 685--696 (1985).


\bibitem{TaoBook}
T.~Tao.
\newblock {\em Nonlinear dispersive equations}, volume 106 of {\em CBMS
  Regional Conference Series in Mathematics}.
\newblock Published for the Conference Board of the Mathematical Sciences,
  Washington, DC; by the American Mathematical Society, Providence, RI, 2006.
\newblock Local and global analysis.


\bibitem{ZLF}
V.~E. Zakharov, V.~S. L'vov, and G.~Falkovich.
\newblock {\em Kolmogorov spectra of turbulence I: Wave turbulence}.
\newblock Springer Science \& Business Media, 2012.

\bibitem{ZOCO}
V. E. Zakharov, A. V. Odesskii, M. Cisternino, and M. Onorato
\newblock Five-Wave Classical Scattering Matrix and Integrable Equations
\newblock {\em Theoretical and Mathematical Physics}, 180(1): 759–764 (2014)


\bibitem{ZSh} 
V.~E. Zakharov, A.~B. Shabat.
\newblock Exact theory of two-dimensional self-focusing and one-dimensional self-modulation of waves in nonlinear media. 
\newblock {\em Sov. Phys. J. Exp. Theor. Phys.}, 34, 62--69 (1972).

\bibitem{ZS}
V.~E. Zakharov, E.~I. Schulman.
\newblock On Additional Motion Invariants of Classical Hamiltonian Wave Systems.
\newblock {\em Phys. D}, 29, 283--320 (1988).

\bibitem{ZSc}
V.~E. Zakharov, E.~I. Schulman
\newblock Integrability of nonlinear systems and perturbation theory.
\newblock {\em What Is Integrability?} Springer Series in Nonlinear Dynamics, Springer, 185--250, 1990.

\bibitem{fange1}
F.~Zheng.
\newblock Long-term regularity of the periodic Euler-Poisson system for electrons in 2D.
\newblock {\em Comm. Math. Phys.}, 366(3):1135--1172, 2019.

\bibitem{fange2}
F.~Zheng.
\newblock Long-term regularity of 3D gravity water waves.
\newblock {\em 	arXiv:1910.01912}.



\end{thebibliography}
\end{document}